\documentclass[12pt,a4paper]{article}

\usepackage{amssymb,amsbsy,amsthm,amsmath,graphicx,epsfig}
\usepackage{mathabx,times,hyperref,eucal}
\usepackage{appendix}
\usepackage{fancyhdr}
\numberwithin{equation}{section}

\DeclareFontFamily{U}{FdSymbolC}{}
\DeclareFontShape{U}{FdSymbolC}{m}{n}{<-> s * FdSymbolC-Book}{}
\DeclareSymbolFont{fdarrows}{U}{FdSymbolC}{m}{n}
\DeclareMathSymbol{\spoon}{\mathrel}{fdarrows}{"6B}

\usepackage{sectsty}

\chapterfont{\Large}
\sectionfont{\large}

\subsectionfont{\normalsize}

\newtheorem{thm}{Theorem}[section]
\newtheorem{lem}[thm]{Lemma}
\newtheorem{prop}[thm]{Proposition}
\newtheorem{cor}[thm]{Corollary}

\newtheorem{dfn}[thm]{Definition}
\newtheorem{prob}[thm]{Problem}

\newtheorem{mainthm}{Theorem}

\newtheorem*{Szeg}{Szeg\H{o}'s theorem}

\theoremstyle{remark}
\newtheorem{ex}[thm]{Example}
\newtheorem{rmk}[thm]{Remark}

\newcommand{\bs}[1]{\boldsymbol{#1}}

\renewcommand{\rm}[1]{\mathrm{#1}}
\renewcommand{\cal}[1]{\mathcal{#1}}

\newcommand{\bbC}{\mathbf{C}}

\newcommand{\bbN}{\mathbf{N}}

\newcommand{\bbR}{\mathbf{R}}
\newcommand{\bbT}{\mathbf{T}}
\newcommand{\bbZ}{\mathbf{Z}}

\newcommand{\rmS}{\mathrm{S}}

\newcommand{\rmH}{\mathrm{H}}

\newcommand{\rmM}{\mathbf{M}}

\newcommand{\rmh}{\mathrm{h}}

\newcommand{\X}{\mathcal{X}}

\newcommand{\A}{\mathfrak{A}}
\newcommand{\B}{\mathfrak{B}}
\newcommand{\frD}{\mathfrak{D}}

\newcommand{\frL}{\mathfrak{L}}
\newcommand{\M}{\mathfrak{M}}
\newcommand{\N}{\mathfrak{N}}

\newcommand{\G}{\Gamma}
\renewcommand{\L}{\Lambda}
\renewcommand{\S}{\Sigma}

\renewcommand{\a}{\alpha}
\renewcommand{\b}{\beta}
\newcommand{\eps}{\varepsilon}
\renewcommand{\phi}{\varphi}
\newcommand{\g}{\gamma}
\renewcommand{\l}{\lambda}
\newcommand{\s}{\sigma}

\newcommand{\ol}[1]{\overline{#1}}

\renewcommand{\t}[1]{\widetilde{#1}}

\newcommand{\fin}{\nolinebreak\hspace{\stretch{1}}$\lhd$}

\newcommand{\bspi}{\boldsymbol{\pi}}

\newcommand{\dom}{\mathrm{dom}\,}
\newcommand{\vol}{\mathrm{vol}}

\newcommand{\tr}{\overline{\mathrm{tr}}}
\newcommand{\Tr}{\mathrm{tr}}
\newcommand{\Det}{\mathrm{det}}

\begin{document}

\title{Entropy and determinants for unitary representations}

\author{Tim Austin}


\maketitle
\RenewDocumentCommand{\headrulewidth}{}{0cm}
\thispagestyle{fancy}
\fancyhf{}
\fancyhf[FLO]{ {\small \quad \emph{Keywords}: Szeg\H{o}'s limit theorem, positive functional, completely positive map, amenability, soficity, hyperlinearity, Fuglede--Kadison determinant, Kolmogorov--Sinai entropy, random past, sofic entropy, subdiagonal subalgebra}

}

\begin{abstract}
Ergodic theory includes several notions of entropy for probability-preserving actions of countable groups.  These include Kolmogorov--Sinai entropy based on F\o lner sequences for amenable groups, entropy defined using a random ordering of the group, and Bowen's sofic entropy for sofic groups.

In this work we pursue these notions across an analogy between ergodic theory and representation theory. We arrive at new quantities associated to unitary representations of groups and representations of other C*-algebras.  Our main results show that these new quantities can often be evaluated as Fuglede--Kadison determinants.  The resulting determinantal formulas offer various non-commutative generalizations of Szeg\H{o}'s limit theorem for Toeplitz determinants.  They make contact with Arveson's theory of subdiagonal subalgebras, and also with some entropy formulas in the ergodic theory of actions by automorphisms of compact Abelian groups.
\end{abstract}
\vspace{7pt}

\setcounter{tocdepth}{1}
\tableofcontents

\pagestyle{plain}

\section{Introduction}

\subsection{A Szeg\H{o} limit theorem over amenable groups}

If $\phi$ is a positive definite function on $\bbZ$, then Bochner's theorem identifies it as the Fourier--Stieltjes transform of a finite Borel measure $\mu$ on the circle group $\bbT$.  Let $m_\bbT$ be the Lebesgue probability measure on $\bbT$. In this context, Szeg\H{o}'s limit theorem describes the asymptotic behaviour of finite-dimensional Toeplitz determinants obtained from $\phi$:

\begin{Szeg}
Let $D_n$ be the determinant of the Toeplitz matrix $[\phi(i-j)]_{i,j=1}^n$ for each $n$, and let $\mu_{\rm{ac}}$ be the absolutely continuous part of $\mu$.  Then
\begin{equation}\label{eq:Szeg}
D_n^{1/n} \to \exp\int_\bbT \log \frac{d\mu_{\rm{ac}}}{dm_\bbT}\,dm_\bbT \qquad \hbox{as}\ n\to\infty,
\end{equation}
taking the right-hand side to be $\exp(-\infty) = 0$ if necessary. \qed
\end{Szeg}

See~\cite[Chap. 1]{SimOPUCI} for a broader overview of this area and a history of some of the key contributions, starting with Szeg\H{o}'s and Verblunsky's. Szeg\H{o}'s theorem also has a generalization in which $\mu$ and $\phi$ take values among positive definite $k$-by-$k$ matrices~\cite[Sec. 2.13]{SimOPUCI}.

The right-hand side of~\eqref{eq:Szeg} admits the following `non-commutative' point of view.  Let $f := d\mu_{\rm{ac}}/d m_\bbT$, so this is an element of $L^1(\bbT)$.  We can regard $L^\infty(\bbT)$ as a von Neumann algebra of multiplication operators on $L^2(\bbT)$, and then integration with respect to $m_\bbT$ defines a tracial state on $L^\infty(\bbT)$.  If $f$ is essentially bounded, then the multiplication operator $M_f$ belongs to this von Neumann algebra; in general, $M_f$ can be defined as an operator affiliated to that algebra. The right-hand side of~\eqref{eq:Szeg} is the logarithm of the Fuglede--Kadison determinant of $M_f$ with respect to $m_\bbT$ (see Subsection~\ref{subs:tracial}).  This interpretation goes back at least to Arveson's conjectured generalization of Szeg\H{o}'s theorem for `subdiagonal subalgebras' in~\cite[property 4.4($\g)$]{Arv67}, which was proved after many years by Labuschagne in~\cite{Lab05}.  (We return to this connection in Subsection~\ref{subs:past}.)

Our first main result is a version of Szeg\H{o}'s theorem for matrix-valued positive definite functions on a countable amenable group $\G$.  Since $\G$ is countable, amenability means that it has a \textbf{right F\o lner sequence}: a sequence $(F_n)_{n\ge 1}$ of finite subsets such that
\begin{equation}\label{eq:right-Folner}
|F_n\triangle F_n g| = o(|F_n|) \qquad \hbox{as $n\to\infty$ for every $g \in \G$}.
\end{equation}
Inversion in the group converts a right F\o lner sequence into a left F\o lner sequence, so the existence of either may be used to define amenability.  See~\cite[Sec. 2.6]{BroOza08} for several other characterizations and properties of this class of groups.

Let $\phi:\G\to \rmM_k$ be positive definite, and for each finite subset $F$ of $\G$ consider the $F$-by-$F$ block matrix
\[\phi[F] := [\phi(g^{-1}h):\ g,h \in F].\]
If $\phi$ is associated to a unitary representation $\pi$ by the vectors $v_1$, $v_2$, \dots, $v_k \in H_\pi$ (see Subsection~\ref{subs:types}), then $\phi[F]$ is the Gram matrix of the tuple of vectors
\[[\pi(g)v_i:\ g \in F,\ i=1,\dots,k],\]
so it is positive semi-definite.

Let $\l$ be the left regular representation.  Then $\phi$ has a unique `Lebesgue decomposition' as $\phi_{\rm{sing}} + \phi_{\rm{ac}}$, where the minimal dilation of $\phi_{\rm{sing}}$ is disjoint from $\l$ and the minimal dilation of $\phi_{\rm{ac}}$ is contained in $\l^{\oplus \infty}$ (see Subsection~\ref{subs:cp-maps-compar}).  In addition, $\phi_{\rm{ac}}$ can be represented in terms of a self-adjoint operator affiliated to $\l^{\oplus k}(\G)'$ (see Proposition~\ref{prop:RadNik}), and using this we can define the Fuglede--Kadison determinant $\Delta \phi_{\rm{ac}}$ (see Definition~\ref{dfn:FK-det-ac}).

\begin{mainthm}\label{mainthm:amenable}
Any right F\o lner sequence $(F_n)_{n\ge 1}$ satisfies
\begin{equation}\label{eq:amenable}
(\det \phi[F_n])^{1/|F_n|} \to \Delta \phi_{\rm{ac}}\qquad \hbox{as}\ n\to\infty.
\end{equation}
\end{mainthm}

Notice that the contribution of the singular part $\phi_{\rm{sing}}$ to the left-hand side of~\eqref{eq:amenable} disappears in the limit, just as in Szeg\H{o}'s original theorem.

The literature contains a number of precedents lying between Szeg\H{o}'s theorem and Theorem~\ref{mainthm:amenable}.  Firstly, generalizations of Szeg\H{o}'s theorem to positive definite functions on $\bbZ^d$ were obtained in~\cite{HelLow58,Linnik75,Dok84} under various extra hypotheses.  More recently,~\cite[Thm. 3.2]{Den06} implies Theorem~\ref{mainthm:amenable} in case $\phi_{\rm{sing}} = 0$ and $\phi$ can be expressed in terms of a positive invertible element of the von Neumann algebra $\rmM_k(\l(\G)'')$.  We discuss the recent literature more fully in Subsection~\ref{subs:A-rmks}.  

\subsection{Analogies with entropy and ergodic theory}\label{subs:analogies}

For a $k$-tuple of vectors in a Hilbert space, their Gram matrix $Q$ specifies their lengths and relative positions~\cite[Sec. 7.2]{HorJohMA}.  In doing so, it is roughly analogous to the joint distribution of $k$ discrete random variables.  The quantity $\log \det Q$ is the analog of the joint Shannon entropy of those random variables.

More formally, if $(X_1,\dots,X_k)$ is a multivariate Gaussian random vector with covariance matrix $Q$, then its differential Shannon entropy equals $\log \det Q$ up to a normalization~\cite[Thm. 8.4.1]{CovTho06}.  From this point of view, the analogy sketched above is really between discrete and differential entropy.  However, the present paper concerns a purely `linear' setting, so we generally discuss this analogy without the extraneous construct of a Gaussian random vector.

Now let $\G$ be a countable group, and consider actions of two different kinds: measure-preserving actions on probability spaces, and unitary actions on Hilbert spaces.  These two settings enjoy a $\G$-equivariant version of the analogy between joint distributions and Gram matrices.  If $(X,\mu,T)$ is a measure-preserving $\G$-system and $\a:X\to A$ is a finite-valued observable, then it generates the shift-system $(A^\G,\a^\G_\ast \mu,S)$, where $\a^\G_\ast \mu$ is the law of the $\G$-indexed stochastic process $(\a\circ T^g:\ g \in \G)$.  Analogously, if $\pi$ is a unitary representation of $\G$ and $v_1,\dots,v_k$ are vectors in it, then these define the $\rmM_k$-valued positive definite function $[\langle \pi(\cdot)v_j,v_i\rangle]$ on $\G$.

This analogy has stimulated research in both areas. For example, it underlies Kechris' adaptation of the relation of weak containment to measure-preserving systems~\cite{KecGAEGA,BurKec20}.  Within this analogy, the logarithm of the limit on the left-hand side in Theorem~\ref{mainthm:amenable} is the analog of the Kolmogorov--Sinai entropy of a stationary process over an amenable group~\cite{Kie75}.

The analogy between ergodic theory and representation theory crystallizes into at least two different formal relationships.  On the one hand, any measure-preserving system gives rise to its Koopman representation~\cite[Sec. II.10]{KecGAEGA}.  On the other hand, any orthogonal real representation of $\G$ can be used to construct a measure-preserving action on a Gaussian Hilbert space, and this construction can be adapted to start with a unitary representation instead~\cite[Apps. C--E]{KecGAEGA}.  However, the first of these relationships does not correctly connect the notions of entropy that we study in this work, and the second introduces unnecessary complications.

On the one hand, the limit in~\eqref{eq:amenable} generally bears no relation to Kolmogorov--Sinai entropy when $\phi$ is associated to a Koopman representation; indeed, the quantity in~\eqref{eq:amenable} does not even define an invariant of unitary equivalence in general.

On the other hand, if we start with a unitary representation and construct the associated Gaussian system, then there are cases in which a suitable limit of log-determinants for the former should equal a `differential' analog of Kolmogorov--Sinai entropy for the latter.  For single transformations or actions of $\bbZ^d$, differential Kolmogorov--Sinai entropy for stationary real-valued processes was studied in~\cite{HamParTho08}, prompted by earlier work in information theory such as~\cite{Pin64}.  Section 6 of~\cite{HamParTho08} includes some exact calculations for Gaussian systems and linear transformations between them which boil down to applications of Szeg\H{o}'s theorem itself.  However, this differential version of Kolmogorov--Sinai entropy starts to behave quite wildly beyond those Gaussian examples, and for those examples alone we might as well stay within the setting of representation theory.  Overall, an analogy at the level of intuition seems more revealing for our work in this paper than either the Koopman or the Gaussian construction.

Beyond Theorem~\ref{mainthm:amenable}, the present paper also considers two other notions of entropy from ergodic theory, and develop their analogs for unitary representations of groups or representations of other C*-algebras.  The first of these notions is entropy defined using a `random past' as in~\cite[Thm. 3]{Kie75}, which leads to Theorem~\ref{mainthm:past}.  The second is `sofic entropy' from~\cite{Bowen10}, which leads to Theorem~\ref{mainthm:det-form}.  In both cases, we show that these notions of entropy are given by $\log \Delta \phi_{\rm{ac}}$ for some positive functional $\phi$, except in Theorem~\ref{mainthm:det-form} a certain degeneracy may occur and then the entropy equals $-\infty$ (also reflecting a known feature of sofic entropy).

Each of these theorems may be seen as a different `non-commutative' version of Szeg\H{o}'s theorem.  In ergodic theory, various predecessors of our results are already discussed this way in the literature, for example in~\cite{Lyons05,LiTho14,Hayes16}.

\subsection{Random orders and Schur complements}\label{subs:past}

Szeg\H{o}'s theorem has several proofs.  The classical ones often begin with a reformulation which implicitly uses the total ordering of $\bbZ$.  First, the Schur determinant formula give the relation
\[\frac{D_{n+1}}{D_n} = \|1_\bbT - P_n(1_\bbT)\|^2,\]
where $P_n$ is the orthogonal projection from $L^2(\mu)$ to $\rm{span}\{z, \dots, z^n\}$~\cite[Thm. 1.5.11]{SimOPUCI}.  It follows that
\[\lim_{n\to\infty}D_n^{1/n} = \lim_{n\to\infty}\frac{D_{n+1}}{D_n} = \|1_\bbT - P(1_\bbT)\|^2,\]
where $P$ is the orthogonal projection from $L^2(\mu)$ to $N := \overline{\rm{span}}\{z, z^2,\dots\}$ (taking the closure in $L^2(\mu)$).  This orthogonal projection is the closest point of $N$ to the function $1_\bbT$, so Szeg\H{o}'s theorem is equivalent to
\begin{equation}\label{eq:Szeg-again}
\inf_{f \in N}\int_\bbT |1-f|^2\ d\mu = \exp\int_\bbT \log\frac{d\mu_{\rm{ac}}}{dm_\bbT}\ dm_\bbT.
\end{equation}
This is the form in which Szeg\H{o}'s theorem most often appears in the literature on analytic functions, such as in~\cite[Sec. V.8]{GameUA}.  That reference gives essentially Szeg\H{o}'s own proof, which he published first for the special case $\mu \ll m$, and then much later for the general case by incorporating arguments of Kolmogorov and Krein to handle $\mu_{\rm{sing}}$; see also~\cite[Secs. 2.4--5]{SimOPUCI}.

Viewed from ergodic theory, the reformulation~\eqref{eq:Szeg-again} is the analog of the formula for the entropy rate of a stationary finite-valued stochastic process $(\xi_n)_{n=-\infty}^\infty$ in terms of its `past':
\begin{equation}\label{eq:entropy-past}
\rmh(\xi) = \rmH(\xi_0\mid \xi_{-1},\xi_{-2},\dots).
\end{equation}
See, for instance,~\cite[Thm. 4.14]{Walters--book}.

To generalize Szeg\H{o}'s theorem to positive definite functions on $\bbZ^d$, one can use the `past' defined by a lexicographic ordering.  This approach was developed by Helson and Lowdenslager in~\cite{HelLow58,HelLow61} (with some later refinements in~\cite{Lyons03}).  In ergodic theory, the same use of the lexicographic ordering appeared in some early works on the entropy of measure-preserving $\bbZ^d$-actions such as~\cite{Con7273,KatzWei72}.

This idea generalizes naturally to any countable group $\G$ that admits a left-invariant total order.  But some countable groups do not, including all groups that are not torsion-free.  To remove the need for this assumption, one can instead couple to a stationary \emph{random} ordering of the group.  All countable groups admit at least one of these: the `Bernoulli random order' (Example~\ref{ex:Bern} below).

In ergodic theory, conditioning on the `past' of a stationary random order can sometimes serve as a replacement for the formula~\eqref{eq:entropy-past}.  If $\G$ is amenable, this idea leads to another classical formula for the entropy rate of a process which goes back to Kieffer~\cite{Kie75} and Stepin~\cite{Stepin78}.  Our next main theorem develops the analog of this idea for positive definite functions.  We find once again that the resulting quantity always agrees with the expected Fuglede--Kadison determinant, even if $\G$ is not amenable.

Before formulating the theorem precisely, let us motivate it further via a related finite-dimensional calculation.  Let $n$ and $k$ be positive integers.  Consider $nk$ vectors in some Hilbert space $H$ indexed as an $n$-by-$k$ array, say
\[V = [x_{m,i}:\ m=1,\dots,n,\ i=1,\dots,k].\]
Let $V_m$ be the $k$-tuple $[x_{m,1},\dots,x_{m,k}]$ for each $m$.  Then the Gram matrix of $V$ obtains an $n$-by-$n$ block structure: $V^\ast V = [V_p^\ast V_m:\ m,p= 1,\dots,n]$.  In addition, let $R_m$ be the orthogonal projection onto the subspace
\[\rm{span}\{x_{p,i}:\ 1 \le p < m,\ i=1,\dots,k\} \qquad (m=1,\dots,n),\]
and let $R_m^\perp := I_H - R_m$.  We can express the determinant of $V^\ast V$ in terms of these subspaces by an iterated appeal to Schur's determinantal formula~\cite[Subsec. 0.8.5]{HorJohMA}.  Within the analogy between Gram matrices and joint distributions from Subsection~\ref{subs:analogies}, this formula is the analog of the chain rule for discrete Shannon entropy.  Taking logarithms and normalizing, the result is
\begin{equation}\label{eq:iterated-Schur}
\frac{1}{n}\log \det(V^\ast V) =\frac{1}{n}\sum_{m=1}^n \log \det ((R_m^\perp V_m)^\ast (R_m^\perp V_m)).
\end{equation}

Formula~\eqref{eq:iterated-Schur} remains valid under any re-ordering of the tuples $V_1$, \dots, $V_n$ (bearing in mind that each $R_m$ depends on the whole order).  So we can take an expectation over a uniform random order on the right-hand side.  Then the symmetry of the random order lets us replace the average of $n$ terms with a single average over orders.  Using $\omega$ to denote a permutation of $1$, \dots, $n$, and now writing $R_\omega$ for the orthogonal projection onto the random subspace
\[\rm{span}\{x_{p,i}:\ 1\le \omega(p) < \omega(1),\ i=1,\dots,k\},\]
we arrive at
\begin{equation}\label{eq:iterated-Schur-2}
\frac{1}{n}\log \det(V^\ast V) = \frac{1}{n!}\sum_{\omega \in S_n} \log \det ((R_\omega^\perp V_1)^\ast (R_\omega^\perp V_1)).
\end{equation}

Our next theorem is an infinite-dimensional, equivariant generalization of~\eqref{eq:iterated-Schur-2}.  Let $\G$ be a countable group, let $\Omega$ be the compact space of total orders on $\G$, and let $\mu$ be a left-invariant Borel probability measure on $\Omega$.  Let $\phi:\G\to \rmM_k$ be positive definite, and suppose it is associated to the representation $\pi$ by the $k$-tuple $x_1$, \dots, $x_k$ in $H_\pi$.  Finally, for each $\omega$, let $R_\omega$ be the orthogonal projection from $H_\pi$ to the closed subspace
\[\ol{\rm{span}}\{\pi(g)x_i:\ g <_\omega e,\ i=1,\dots,k\}.\]

\begin{mainthm}\label{mainthm:past}
In the situation above, we have
\begin{equation}\label{eq:past}
\exp\int \log \det [\langle R_\omega^\perp x_j,R_\omega^\perp x_i\rangle]\ d\mu(\omega) = \Delta \phi_{\rm{ac}}.
\end{equation}
\end{mainthm}

For example, when $k=1$ this simplifies to
\[\exp\int \log \|x - R_\omega x\|^2\ d\mu(\omega) = \Delta \phi_{\rm{ac}}.\]

From the viewpoint of ergodic theory, Theorem~\ref{mainthm:past} is somewhat surprising, because random-order entropy for a finite-valued stationary processes can misbehave when $\G$ is not amenable.  It agrees with Rokhlin or sofic entropy (discussed below) for some special examples of processes~\cite{AusPod--percolative,Alp--randomorder}, but in general it is only an upper bound for those quantities, and it need not be invariant under isomorphism~\cite[Sec. 7]{Seward--randomorder}.  By contrast, when both Theorems~\ref{mainthm:past} above and~\ref{mainthm:det-form} below can be applied to a positive definite function $\phi:\G\to\rmM_k$, they always give the same value $\log\Delta \phi_{\rm{ac}}$ for their respective notions of entropy.  So some serious pathologies from ergodic theory do not appear for unitary representations.

We prove Theorem~\ref{mainthm:past} in Section~\ref{sec:past}.  We do this by making contact with Arveson's theory of `subdiagonal subalgebras' of finite von Neumann algebras~\cite{Arv67}.  These are defined by axioms abstracted from the inclusion $H^\infty(\bbT) \subset L^\infty(\bbT)$, which becomes a commutative example; see also~\cite[Sec. 8]{PisXuNCLP} or~\cite{BlecLab07b} for surveys in the context of noncommutative Lebesgue spaces.  A subdiagonal subalgebra provides a finite von Neumann algebra with an abstract notion of `past'.  Among his other examples, Arveson showed how to construct a subdiagonal subalgebra from any countable group with an invariant total order, and our starting point is a generalization of this construction using an invariant random order and a crossed product.  Then we derive Theorem~\ref{mainthm:past} via a more abstract version of the same result (Theorem~\ref{thm:subdiag-Szego}), which we deduce from the Arveson--Labuschagne generalization of Jensen's inequality from~\cite{Arv67} and ~\cite{Lab05}.

\subsection{Almost periodic entropy}

Since its introduction in~\cite{Bowen10}, Bowen's notion of sofic entropy has taken a central place in the ergodic theory of actions of non-amenable groups. The survey~\cite{Bowen--survey} offers a thorough introduction.  This is the final notion of entropy that we pursue across the analogy between ergodic theory and representation theory in this paper. 

In this effort, we quickly find it helpful to allow the generality of unital representations of a separable, unital C*-algebra $\A$, rather than just unitary representations of a group.  The case of a countable group $\G$ is recovered when $\A = C^\ast \G$.  Even in that special case, certain auxiliary constructions lead us to consider other C*-algebras as well.

Given a representation $\pi$ of $\A$ and vectors $v_1$, \dots, $v_k \in H_\pi$, define their \textbf{type} to be the $\rmM_k$-valued completely positive map
\[\Phi^\pi_{v_1,\dots,v_k}(a) := [\langle \pi(a)v_j,v_i\rangle]_{i,j =1}^k \qquad (a \in \A).\]
If $O$ is any set of $\rmM_k$-valued completely positive maps for a fixed value of $k$, and $\pi$ is any representation, then we define
\[\X(\pi,O) := \{(v_1,\dots,v_k) \in H_\pi^k:\ \Phi^\pi_{v_1,\dots,v_k} \in O\}.\]
Imagining that $O$ is a small neighbourhood of a particular map $\phi$, this is the analog of a set of `good models' for a given shift-invariant measure in sofic entropy theory (see~\cite{Bowen--survey})

Consider also a sequence $\bspi = (\pi_n)_{n\ge 1}$ of representations of $\A$ whose dimensions $d_n$ are finite but diverge.  We refer to it as an \textbf{almost periodic sequence} for $\A$.  We define the \textbf{almost periodic entropy} of an $\rmM_k$-valued completely positive map $\phi$ along $\bspi$ to be
\[\rmh_{\bspi}(\phi) := \inf_O \limsup_{i\to \infty}\frac{1}{d_i}\log \frac{\vol_{2kd_i}\X(\pi_i,O)}{v(d_i)^k},\]
where $O$ ranges over neighbourhoods of $\phi$, $\vol_{2kd_i}$ refers to Lebesgue measure in $2kd_i$ real dimensions, and $v(d_i)$ is the volume of the unit ball in $\bbC^{d_i}$.  This is inspired by the definition of sofic entropy for a finite-valued stationary process over a sofic group.

We introduce this new notion of entropy and develop its basic properties in Section~\ref{sec:AP}.  Many of these resemble properties of sofic entropy, but some differences emerge.  For example, almost periodic entropy is not an invariant of unitary equivalence of representations, but it does satisfy a transformation formula if one changes cyclic vector within a fixed representation.

Our main result in this section is a formula for $\rmh_{\bspi}(\phi)$ as a Fuglede--Kadison determinant.  It holds whenever $\phi$ is `asymptotically associated' to $\bspi$ and the pulled-back traces $d_i^{-1}\Tr_{d_i}\circ \pi_i$ converge to a limiting tracial state $\tau$ of $\A$.  `Asymptotic association' means that, for every neighbourhood $O$ of $\phi$, the set $\X(\pi_i,O)$ is nonempty for infinitely many $i$; if this fails then $\rmh_{\bspi}(\phi)$ is simply forced to be $-\infty$.

Let $\Delta$ be the Fuglede--Kadison determinant defined from $\tau$, and let $\phi_\rm{ac} + \phi_\rm{sing}$ be the Lebesgue decomposition of $\phi$ relative to $\tau$ (see Subsection~\ref{subs:cp-maps-compar}).

\begin{mainthm}\label{mainthm:det-form}
Suppose that $d_n^{-1}\Tr_{d_n}\circ \pi_n \to \tau$ and that $\phi$ is asymptotically associated to $\bspi$. Then
\[\rmh_{\bspi}(\phi) = \log \Delta \phi_{\rm{ac}}.\]
\end{mainthm}

Let us emphasize two features of Theorem~\ref{mainthm:det-form} that are substantially different from Theorems~\ref{mainthm:amenable} and~\ref{mainthm:past}:
\begin{itemize}
\item Any tracial positive functional $\tau$ on $\A$ may appear in Theorem~\ref{mainthm:det-form}, provided it can arise as a suitable limit of normalized finite-dimensional traces.  By contrast, Theorems~\ref{mainthm:amenable} and~\ref{mainthm:past} refer specifically to the regular character on a group $\G$ and its associated tracial state on $C^\ast \G$.
\item Theorem~\ref{mainthm:det-form} does not make any assumption on $\A$ itself that corresponds to soficity of a group.  However, such an assumption is implicit in the hypothesis that $\tau$ is a limit of normalized finite-dimensional traces.  By applying this theorem to $C^\ast \G$ when $\G$ is a free group and $\tau$ is lifted from a quotient group of $\G$, one recovers a theorem for precisely R\u{a}dulescu's class of hyperlinear groups from~\cite{Rad08}: see Subsection~\ref{subs:AP}.
\end{itemize}

In comparison with Szeg\H{o}'s theorem, Theorem~\ref{mainthm:det-form} has the interesting new feature that $\rmh_{\bspi}(\phi)$ may equal $-\infty$ if $\phi$ is not asymptotically associated to $\bspi$, even though $\log \Delta \phi_{\rm{ac}}$ may still be finite in that case.  If $\A = C^\ast \G$ and $\tau$ is the state given by the regular character of $\G$, then this is possible only if $\G$ is non-amenable, and reflects basic features of the representation theory of non-amenable groups.

Subsection~\ref{subs:C-cors} derives various consequences of Theorem~\ref{mainthm:det-form}, for example concerning different possible modes of convergence for the sequence $\bspi$ itself.

As far as I know, the nearest precursors to Theorem~\ref{mainthm:det-form} in the literature are formulas for the sofic entropy of certain special measure-preserving systems in~\cite{Lyons05,Lyons10} and especially~\cite{Hayes16,Hayes21}.  We compare our work with these in Subsection~\ref{subs:C-rmks}.

\subsection*{Acknowledgements}

This work depended on insightful conversations and correspondence with many people.  In this regard I am particularly grateful to Nir Avni, Uri Bader, David Blecher, Lewis Bowen, Peter Burton, Amir Dembo, Louis Labuschagne, Michael Magee, Magdalena Musat, Narutaka Ozawa, Sorin Popa, Mikael R\o rdam, Brandon Seward and Dimitri Shlyakhtenko.

For the purpose of open access, the author has applied a Creative Commons Attribution (CC-BY) licence to any Author Accepted Manuscript version arising from this submission.

\section{Operator algebras, C*-algebras and representations}\label{sec:op-alg}

This section recalls the background we need from linear algebra and the theory of C*-algebras.  I assume functional analysis at about the level of~\cite{ReeSimFA}. Beyond that, I state a number of standard results for later reference, and prove a few that are not easily found in the literature.

\subsection{Linear algebra}\label{sec:lin-alg}

We write $\bbC^{\oplus k}$ for the space of complex height-$k$ column vectors.  More generally, if $S$ is a set, possibly infinite, and $H$ is a Hilbert space, then we write$H^{\oplus S}$ for the Hilbert-space direct sum of an $S$-indexed family of copies of $H$, still regarded as a space of column vectors. This insistence on column vectors is slightly unusual in functional analysis, but for finite $k$ it enables us to use matrix-vector notation from linear algebra in places where it simplifies the exposition.

We write $\rmM_{n,k}$ for the space of $n$-by-$k$ matrices over the complex numbers, and identify these with linear maps from $\bbC^{\oplus k}$ to $\bbC^{\oplus n}$ using matrix-vector multiplication.  By writing such a matrix as $[v_1,\dots,v_k]$, where $v_1$, \dots, $v_k$ are its columns, we can identify it with a $k$-tuple of vectors in $\bbC^{\oplus n}$.  We generalize this notation further by allowing columns from any vector space $H$, so a linear map $V$ from $\bbC^{\oplus k}$ to $H$ may still be written in the form $[v_1,\dots,v_k]$.  We sometimes abuse notation by calling $V$ itself a `$k$-tuple of vectors in $H$'.  If $H$ is a Hilbert space, then the adjoint $V^\ast$ is the map from $H$ to $\bbC^{\oplus k}$ whose output coordinates are given by the inner products with the vectors $v_i$.

We abbreviate $\rmM_{k,k}$ to $\rmM_k$ and regard it as a $\ast$-algebra over $\bbC$ in the usual way.  We write $I_k$ for the $k$-by-$k$ identity matrix. We write $\Tr_k$ and $\Det$ for the usual trace and determinant on any such algebra, and we set $\tr_k := k^{-1}\Tr_k$.

We write $\rmM_{k+}$ for the closed cone of positive semi-definite elements of $\rmM_k$.  It defines the positive definite ordering on self-adjoint matrices. If $Q \in \rmM_{k+}$, then its determinant and trace are related by the inequality
\begin{equation}\label{eq:AMGM}
(\det Q)^{1/k}\le \tr_k Q.
\end{equation}
This is simply the inequality of arithmetic and geometric means applied to the eigenvalues of $Q$.

For a linear operator on an inner product space, or a matrix that can be regarded as such, the notation $\|\cdot\|$ means the operator norm.

If $P$ is an orthogonal projection in a Hilbert space $H$, then we use $P^\perp$ as a shorthand for $I-P$.

\subsection{C*-algebras, von Neumann algebras, and representations}\label{subs:basics}

Throughout this paper, $\A$ is a \emph{separable}, \emph{unital} C*-algebra and we study \emph{separable} representations, meaning that they act on separable complex Hilbert spaces.  We denote the unit of $\A$ by $1_\A$ or just $1$.  Our guiding examples are the group C*-algebras of countable groups: see Subsection~\ref{subs:group-alg} below.  We follow the common convention that C$^\ast$-algebras may exist in the abstract, but a von Neumann algebra is always a weak-operator closed $\ast$-subalgebra of $\frL(H)$ for some particular Hilbert space $H$.  In particular, we sometimes casually identify isomorphic C$^\ast$-algebras when this can cause no confusion, but two von Neumann algebras acting on different Hilbert spaces are not identified, even if they are isomorphic.

We usually denote a representation of $\A$ by a single letter such as $\pi$, and then write its Hilbert space as $H_\pi$ when necessary.  We use $\oplus$ to denote direct sums in the categories of Hilbert spaces or representations~\cite[Subsec. 2.2.3]{Dix--Cstar}, and we use $\otimes$ for tensor products of Hilbert spaces, operators on Hilbert spaces, or von Neumann algebras of such operators~\cite[Secs. I.2.3--4]{Dix--vN}.  We do not need the more involved theory of tensor products of abstract C*-algebras.  For a representation $\pi$ and any $k\in \bbN\cup\{\infty\}$, we write either $\pi^{\oplus k}$ or $\pi \otimes I_k$ for the direct sum of $k$ copies of $\pi$, and refer to it as the \textbf{$k$-fold inflation} of $\pi$.

If $\pi$ is a representation and $M$ is a closed $\pi$-invariant subspace of $H_\pi$, then we write $\pi^M$ for the associated subrepresentation of $\A$ on $M$.  If $\pi$ is a representation of $\A$, then a subset $S$ of $H_\pi$ is \textbf{cyclic} for $\pi$ if it is not contained in any proper closed invariant subspace of $H_\pi$, or equivalently if $\sum_{v \in S}\pi(\A)v $ is dense in $H_\pi$.

Given two representations $\pi$ and $\rho$, we write $\pi \simeq \rho$ if they are unitarily equivalent, $\pi \lesssim \rho$ if $\pi$ is contained in $\rho$, and $\pi\spoon\rho$ if they are disjoint; see~\cite[Secs. 2.2 and 5.2]{Dix--Cstar}, for example.

Now let $\kappa := \pi \oplus \rho$.  Regard $H_\pi$ and $H_\rho$ as subspaces of $H_\kappa$, and let $P$ be the orthogonal projection from $H_\kappa$ onto $H_\pi$.  The next result is~\cite[Prop. 5.2.4]{Dix--Cstar}.

\begin{lem}\label{lem:central-projection}
We have $\pi \spoon \rho$ if and only if $P$ lies in the centre of $\kappa(\A)''$. \qed
\end{lem}

A related result of Mackey breaks up two arbitrary representations into quasi-equivalent and disjoint pieces~\cite[Thm. 1.11]{Mac76}.  We need the following special case.

\begin{prop}\label{prop:Mac}
If $\pi$ and $\rho$ are representations, then $\pi$ has a unique closed invariant subspace $M$ such that $\pi^M \lesssim \rho^{\oplus \infty}$ and $\pi^{M^\perp} \spoon \rho$. \qed
\end{prop}

We call $\pi^M$ and $\pi^{M^\perp}$ the \textbf{$\rho$-normal} and \textbf{$\rho$-singular} parts of $\pi$, respectively.

\subsection{Types and completely positive maps}\label{subs:types}

Let $\pi$ be a representation of $\A$, let $v_1$, \dots, $v_k \in H_\pi$, and regard the tuple $V := [v_1,\dots,v_k]$ as a linear map from $\bbC^{\oplus k}$ to $H_\pi$.  To keep track of how these vectors move together under the action of $\pi$, we can consider the $\rmM_k$-valued map
\begin{equation}\label{eq:cp-assoc}
\Phi^\pi_V(a) := V^\ast \pi(a) V = [\langle \pi(a)v_j,v_i\rangle]_{i,j=1}^k \qquad (a \in \A).
\end{equation}
Notice that the order of the indices matches the convention for the Gram matrix of a tuple of vectors~\cite[Sec. 7.2]{HorJohMA}.  We sometimes adapt a term from information theory by calling $\Phi^\pi_V$ the \textbf{type} of the tuple $V$ in $\pi$ (compare~\cite[Sec. 11.1]{CovTho06}, for example).  We also sometimes write $\Phi^\pi_{v_1,\dots,v_k}$ instead of $\Phi^\pi_V$.

The map in~\eqref{eq:cp-assoc} is completely positive, and any completely positive map from $\A$ to $\rmM_k$ has this form for some $\pi$ and $V$ by Stinespring's theorem~\cite[Ex. 1.5.2 and Thm. 1.5.3]{BroOza08}.  If we require in addition that $V$ be cyclic for $\pi$, then the resulting pair $(\pi,V)$ is unique up to unitary equivalence.  In this case $\pi$ is the \textbf{minimal dilation} of $\phi$ and is denoted by $\pi_\phi$.  By the uniqueness of minimal dilations, $\phi$ is associated to $\pi$ if and only if $\pi_\phi \lesssim \pi$. 

When $k=1$, these facts reduce to the construction of the GNS representation and its distinguished cylic vector from a positive linear functional.  We write $\A^\ast_+$ for the space of such functionals on $\A$.  For general values of $k$, we write $\frL(\A,\rmM_k)$ for the space of all continuous linear maps from $\A$ to $\rmM_k$, and $\frL(\A,\rmM_k)_+$ for the subset of all completely positive ones.  If $\phi, \psi \in \frL(\A,\rmM_k)$, then we write $\phi \le \psi$ if the functional $\psi - \phi$ is completely positive.  This defines a partial order on $\frL(\A,\rmM_k)$ in which $\frL(\A,\rmM_k)_+$ is the non-negative cone.  We write $\S_k(\A)$ for the subset of all completely positive maps $\phi:\A \to \rmM_k$ that are \textbf{normalized}, meaning that $\tr_k \phi(1) = 1$.  In particular, $\S_1(\A)$ is the state space of $\A$~\cite[Chap. 2]{Dix--Cstar}.

The vector space $\frL(\A,\rmM_k)$ has a natural topology obtained by applying the weak$^\ast$ topology in each matrix entry. Henceforth we simply refer to this as `the weak$^\ast$ topology' of $\frL(\A,\rmM_k)$, and take it as the default topology on $\frL(\A,\rmM_k)$ or its subsets.  For any $k$, complete positivity is defined by a family of closed linear inequalities, so $\frL(\A,\rmM_k)_+$ is a weak$^\ast$-closed cone in $\frL(\A,\rmM_k)$.  The further subset $\S_k(\A)$ is compact by the Banach--Alaoglu theorem and also metrizable because $\A$ is separable.

\begin{lem}\label{lem:unif-cts}
	For any $\pi$ and $k$, the type map
	\[H_\pi^k \mapsto \frL(\A,\rmM_k):[v_1,\dots,v_k] \mapsto \Phi^\pi_{v_1,\dots,v_k}\]
	is continuous.
\end{lem}

\begin{proof}
	This is elementary for the inner product map $H_\pi\times H_\pi\to \bbC$, and then follows for types by arguing pointwise for $i$, $j \in \{1,\dots,k\}$, and ${a \in \A}$.
\end{proof}

Now consider again a representation $\pi$ and a tuple $v_1$, \dots, $v_k$ in $H_\pi$.  If $a = [a_{ij}]$ is an $\ell$-by-$k$ matrix of elements of $\A$, then we can define a new $\ell$-tuple in $H_\pi$ by the formula
\begin{equation}\label{eq:matrix-vector}
\left[\begin{array}{c}y_1 \\ \vdots \\ y_\ell\end{array}\right] := [\pi(a_{ij})]\cdot\left[\begin{array}{c}v_1\\ \vdots \\ v_k\end{array}\right],
\end{equation}
understood by following the rules of matrix-vector multiplication.  For example, if $a = [q_{ij}\cdot 1]$ for some scalar matrix $Q = [q_{ij}]$, then we can identify $[\pi(a_{ij})]$ with $I_{H_\pi}\otimes Q$, and~\eqref{eq:matrix-vector} becomes
\begin{equation}\label{eq:y-Q-v}
[y_1,\dots,y_\ell]^\rm{T} := (I_{H_\pi}\otimes Q)[v_1,\dots,v_k]^\rm{T}.
\end{equation}

If the tuples $v_1$, \dots, $v_k$ and $y_1$, \dots, $y_\ell$ are related as in~\eqref{eq:matrix-vector}, and $\phi$ and $\psi$ are their respective types, then $\psi$ may be written using the tuple $v_1$, \dots, $v_k$ like this:
\[\psi_{ii'}(b) = \langle \pi(b)y_{i'},y_i\rangle = \sum_{j,j'=1}^k \langle \pi(ba_{i'j'})v_{j'},\pi(a_{ij})v_j\rangle \qquad (b \in \A,\ 1\le i,i'\le k).\]
Writing this right-hand side in terms of $\phi$ itself yields the following.

\begin{lem}\label{lem:new-type}
In the situation above, we have
\begin{equation}\label{eq:new-type}
\psi_{ii'}(b) = \sum_{j,j'=1}^k \phi_{jj'}((a_{ij})^\ast ba_{i'j'}) \qquad (b \in \A,\ 1\le i,i'\le k).
\end{equation}
As a result, with $[a_{ij}]$ held fixed, $\psi$ is continuous as a function of $\phi$ for the weak$^\ast$ topologies.

In particular, in the special case of~\eqref{eq:y-Q-v}, we have
	\begin{equation}\label{eq:psi-Q-phi}
\psi(b) := (Q^\rm{T})^\ast\phi(b)Q^\rm{T} \qquad (b \in \A).
	\end{equation}
\qed
\end{lem}

\subsection{Tracial functionals and determinants}\label{subs:tracial}

A special role is played by positive functionals $\tau$ on $\A$ that are \textbf{tracial}, meaning they satisfy the \textbf{trace identity}: $\tau(ab) = \tau(ba)$ for all $a,b \in \A$. Fix such a functional $\tau$, let its GNS representation be $\l$ with associating vector $\xi$, and let $H := H_\l$.  Then we also refer to the vector $\xi$ as \textbf{tracial}.

A tracial functional gives rise to a rich additional structure on the von Neumann algebra $\N := \l(\A)''$.  This is introduced in~\cite[Chap. I.5]{Dix--vN}, which uses the language of Hilbert algebras, and~\cite[Chap. I.6]{Dix--vN}, which explains the equivalence with tracial functionals.  A more modern account is~\cite[App. F]{BroOza08}.  The vector $\xi$ defines the normal positive functional $\langle (\cdot)\xi,\xi\rangle$ on $\frL(H)$, and the restriction of this functional to either $\N$ or $\N'$ still satisfies the trace identity.  We often write $\t{\tau}$ for either of these restrictions.  When $\t{\tau}$ is the restriction to $\N$, it satisfies $\tau = \t{\tau}\circ \l$.

In addition, $\xi$ turns out to be both cyclic and separating for both $\N$ and $\N'$~\cite[Cor. I.6.1 and Prop. I.1.5]{Dix--vN}.  Using $\xi$, the dense subspace $\N\xi$ has a well-defined involution given by
\begin{equation}\label{eq:canonical-involution}
	A\xi \mapsto A^\ast \xi \qquad (A \in \N).
\end{equation}
This map anti-linear, it fixes $\xi$, and it converts inner products to their conjugates as a consequence of the trace identity.  It is therefore an isometry, and extends by continuity to an involution of the whole of $H$ that has the same properties.  This is the \textbf{canonical involution} $J$ associated to $\N$ and $\xi$.  Finally, the map $A\mapsto JAJ$ is an involutive $\ast$-anti-automorphism of $\frL(H)$, it acts by complex conjugation on the functional $\langle(\cdot)\xi,\xi\rangle$, and it exchanges the subalgebras $\N$ and $\N'$ of $\frL(H)$.

In the context above, some constructions require operators that are closed, densely-defined, and affiliated to $\N$ or $\N'$, but possibly unbounded.  See~\cite[Chap. VIII]{ReeSimFA} for background on unbounded operators, and~\cite[Exers. I.1.10 and III.1.13]{Dix--vN} or~\cite[App. F]{BroOza08} for the definition and some basic properties of affiliated operators.  Such operators are the basic elements of `noncommutative integration theory': see~\cite{Seg53,Seg53b} for an early account.  They could also be identified with elements of abstract noncommutative Lebesgue spaces~\cite{PisXuNCLP}, but our uses below depend on them being operators on Hilbert spaces. 

Our specific need is for the following subclass. Let $T$ be affiliated to $\N$, and let $E$ be the spectral resolution of $|T|$ on $[0,\infty)$.  Following~\cite[Subsec. 3.4]{Seg53} and~\cite[App. F]{BroOza08}, we call $T$ \textbf{square-integrable} with respect to $\t{\tau}$ if
\begin{equation}\label{eq:sq-int}
\int_{[0,\infty)}t^2\ \t{\tau} E(dt) < \infty.
\end{equation}
If $\t{\tau}$ is the restriction of $\langle (\cdot)\xi,\xi\rangle$ to $\N$, then square-integrability is equivalent to $\xi \in \dom T$, and then the left-hand side of~\eqref{eq:sq-int} is equal to $\|T\xi\|^2$.

Any normal tracial positive functional $\t{\tau}$ on a von Neumann algebra $\N$ defines an associated \textbf{Fuglede--Kadison determinant}: see~\cite[Sec. I.6.11]{Dix--vN}.  We denote it by $\Delta_{\t{\tau}}$, or just $\Delta$ if the choice of $\t{\tau}$ is clear.

If $T$ is a square-integrable operator affiliated to $\N$, then we may extend the definitions of trace and determinant by setting
\begin{equation}\label{eq:FK-det}
\t{\tau}(T) := \int_{[0,\infty)}t\ \t{\tau}E(dt) \qquad \hbox{and}\qquad \Delta T := \exp \int_{[0,\infty)} \log t\ \t{\tau} E(dt).
\end{equation}
Both integrands are dominated by $t^2+1$, so $\t{\tau}(T)$ is finite and $\Delta T$ is either positive or equal to $\exp(-\infty) = 0$.  These definitions can be extended to even larger classes of operators, but we do not need those here; see~\cite[Sec. 2]{HaaSch07} for the determinant in the full generality of `log-integrable' operators.  If $\t{\tau}$ is a state, then $\t{\tau} E$ is a probability measure, and we can apply Jensen's inequality~\cite[Thm. 3.3]{Rudin--realcplx} to the two integrals in~\eqref{eq:FK-det} to obtain an extension of the determinant-trace inequality~\eqref{eq:AMGM}:
\begin{equation}\label{eq:AMGM2}
\Delta T\le \t{\tau}(T).
\end{equation}

Now suppose that $T$ is non-negative, square-integrable, and affiliated to $\N$, and let $E$ be its spectral resolution on $[0,\infty)$.  For any $\delta \in (0,1)$, let 
	\[T_\delta := (T\vee \delta)\wedge \delta^{-1} :=  \int_{[0,\infty)} (t\vee \delta)\wedge \delta^{-1}\ E(dt),\]
	where `$\vee$' stands for `$\max$' and `$\wedge$' stands for `$\min$'. Then $T_\delta$ is an element of $\N$ satisfying $\delta \le T_\delta \le \delta^{-1}$.  Let $E_+ := E(0,\infty) = 1 - E\{0\}$.

\begin{lem}\label{lem:Kap+}
 As $\delta \downarrow 0$, these operators satisfy
	\begin{enumerate}
	\item[i.] $T_\delta x \to Tx$ and $T_\delta^{-1}T x \to E_+x$ for every $x \in \dom T$;
	\item[ii.] $\Delta T_\delta \to \Delta T$.
	\end{enumerate}
\end{lem}

\begin{proof}
Let $x \in \dom T$. Then the spectral theorem gives
\begin{align*}
	\|Tx - T_\delta x\|^2 &= \int_{[0,\infty)} |t - (t\vee \delta)\wedge \delta^{-1}|^2\ \langle E(dt)x,x\rangle,\\
	\|E_+x - T_\delta^{-1}Tx\|^2 &= \int_{(0,\infty)} \Big|1 - \frac{t}{(t\vee \delta)\wedge \delta^{-1}}\Big|^2\ \langle E(dt)x,x\rangle\\
	\hbox{and} \qquad \log \Delta T_\delta &= \int_{[0,\infty)} \log((\delta \vee t)\wedge \delta^{-1})\ \t{\tau} E(dt)
\end{align*}
	Both of the expressions $\langle E(\cdot)x,x\rangle$ and $\t{\tau}E$ are finite Borel measures on $\bbR$, and the function $t^2+1$ is integrable with respect to both of them by our assumptions on $T$.  Therefore the dominated convergence theorem applies to the first two integrals above as $\delta \downarrow 0$, showing that they both converge to $0$.  This proves conclusion (i).  The dominated convergence theorem also applies to the positive part of the third integral, while the negative part is subject to the monotone convergence theorem.  It follows that the third integral converges to $\log \Delta T$ (even if this equals $-\infty$), and then exponentiating gives conclusion (ii).
\end{proof}

If $\tau$ is a tracial positive functional on a C*-algebra $\A$ and we construct $\l$, $\xi$ and $\t{\tau}$ on $\l(\A)''$ from it as above, and if $a$ is a non-negative element of $\A$, then we write $\Delta a$ for $\Delta \l(a)$.  If $a$ is also invertible, then $\Delta a$ is equal to $\exp(\tau(\log a))$.

\subsection{Group algebras and positive definite functions}\label{subs:group-alg}

Let $\G$ be a countable discrete group with identity element $e$. We write $\bbC[\G]$ for the complex group algebra of $\G$, and regard it concretely as the space of finitely supported functions from $\G$ to $\bbC$; see~\cite[Chap. 13]{Dix--Cstar}, which also allows general locally compact groups.  Given $g \in \G$, we write $\delta_g$ for its canonical image in $\bbC[\G]$, so the unit of $\bbC[\G]$ is $\delta_e$.

Similarly, we identity $\rmM_k(\bbC[\G])$ with the vector space $\rmM_k[\G]$ of finitely supported maps from $\G$ to $\rmM_k$.  Given $\phi,\psi:\G\to \rmM_k$, at least one of them finitely supported, we extend the usual definition of convolution by writing
\begin{equation}\label{eq:convo}
(\phi\ast \psi)(g) := \sum_{h,k:\ hk=g}\phi(h)\psi(k) = \sum_h \phi(h)\psi(h^{-1}g) \qquad (g \in \G).
\end{equation}
The individual summands here are matrix products.  If both $\phi$ and $\psi$ are finitely supported then so is $\phi\ast \psi$, and then~\eqref{eq:convo} defines the structure of $\rmM_k[\G]$ as a group algebra with matrix coefficients.

The group C*-algebra $C^\ast \G$ is the maximal C*-completion of $\bbC[\G]$.  Representations of $C^\ast \G$ are in one-to-one correspondence with unitary representations of $\G$ itself, and we generally use the same notation for a representation of $C^\ast\G$ and for its restriction to $\G$; see~\cite[Sec. 13.9]{Dix--Cstar}.  Throughout the rest of this paper, a `representation' of $\G$ always means a unitary representation. For each $g \in \G$, we continue to write $\delta_g$ for its image in $C^\ast \G$.  We can identify $\rmM_k(C^\ast \G)$ with the corresponding completion of $\rmM_k[\G]$ (see~\cite[Lem. 10]{Oza13}).

Now consider a completely positive map $\phi:C^\ast \G\to \rmM_k$. Identifying each group element $g$ with its image $\delta_g$, we can restrict $\phi$ to an $\rmM_k$-valued map on $\G$ itself, which we usually continue to denote by $\phi$.  The maps on $\G$ that arise this way are positive definite: see~\cite[Sec. 13.4]{Dix--Cstar} for the case $k=1$ or~\cite[App. D]{BroOza08} for the general case.  On the other hand, another variant of the GNS construction shows that any $\rmM_k$-valued positive definite function on $\G$ is associated to some unitary representation, and therefore extends to a completely positive map on $C^\ast\G$.  The proof in the matrix-valued case can be modeled directly on the scalar case; this idea essentially goes back Naimark's work~\cite{Nai43} on Abelian groups.  So Naimark's result is an older cousin of Stinespring's theorem, although not quite a special case of it; see~\cite[Thm. 12 and Cor. 13]{Oza13} for a unified treatment. 

Under this bijection between the set of $\rmM_k$-valued positive definite maps on $\G$ and the space $\frL(C^\ast \G,\rmM_k)_+$, the weak$^\ast$ topology on $\frL(C^\ast \G,\rmM_k)_+$ corresponds to the usual weak$^\ast$ topology restricted from $\ell^\infty(\G;\rmM_k)$, and when restricted further to any uniformly bounded subset it coincides with the topology of pointwise convergence.

Let $\phi:\G\to\rmM_k$ be positive definite, let $a \in \rmM_k[\G]$, and regard $a$ as an element of $\rmM_k(C^\ast \G)$.  Having defined convolution in the generality of~\eqref{eq:convo}, we may use it to express the pairing from~\eqref{eq:pairing} like this:
\begin{equation}\label{eq:matrix-pos-def-convo}
\langle \phi,a\rangle = \frac{1}{k}\sum_{ij}\sum_g a_{ij}(g)\phi_{ij}(g) = \sum_g\tr_k (\phi^\rm{t}(g^{-1}) a(g)) = \tr_k((\phi^\rm{t}\ast a)(e)),
\end{equation}
where we define
\begin{equation}\label{eq:phi-hatphi}
\phi^\rm{t}(g) := \phi(g^{-1})^\rm{T} \qquad (g \in \G)
\end{equation}
(that is, we apply inversion in $\G$ and transpose to elements of $\rmM_k$).

Finally, a state $\tau$ on $C^\ast \G$ is tracial if and only if its restriction to $\G$ is central.  In this case that restriction is called a \textbf{character} of $\G$.  For example, the function $1_{\{e\}}$ is the \textbf{regular character}.  It is associated to the left regular representation on $\ell^2(\G)$ by the function $\delta_e$.  More generally, if $H$ is a subgroup of $\G$, then the function $1_H$ is positive definite.  It is associated to the quasi-regular representation of $\G$ on $\ell^2(\G/H)$ by the function $\delta_{eH}$, and it is a character if and only if $H$ is normal in $G$.

\subsection{Algebras of matrices and the pairing isomorphism}

If $\A$ is a C$^\ast$-algebra and $k$ is a positive integer, then we write $\rmM_k(\A)$ for the algebra of $k$-by-$k$ matrices with entries from $\A$. The algebra operations combine those of $\A$ with the usual rules for matrices, and we define an involution on $\rmM_k(\A)$ by transposing and applying the involution of $\A$ entry-wise.  Then $\rmM_k(\A)$ has a natural identification with the algebraic tensor product of $\A$ and $\rmM_k$.

If $\pi$ is a representation of $\A$ and $k$ is a positive integer, then we define a representation $\pi^{(k)}$ of $\rmM_k(\A)$ on $H_\pi^{\oplus k}$ by setting $\pi^{(k)}([a_{ij}]) := [\pi(a_{ij})]$ and following the rules of matrix-vector multiplication as in~\eqref{eq:matrix-vector}.  From another point of view, we can identify $H_\pi^{\oplus k}$ with $H_\pi\otimes \bbC^{\oplus k}$, and then $\pi^{(k)}$ is the Kronecker product of $\pi$ with the canonical representation of $\rmM_k$ on $\bbC^{\oplus k}$.  Every representation of $\rmM_k(\A)$ is equivalent to one of this form: this can be seen by using the canonical copy of $\rmM_k$ inside $\rmM_k(\A)$ to break up a general representation into $k$ orthogonal subspaces with partial isometries between them (compare~\cite[Subs. 9.2.2]{Dix--Cstar}).  We make $\rmM_k(\A)$ into a C*-algebra in a canonical way by pulling back the operator norm through $\pi^{(k)}$ for any faithful representation $\pi$.

Elements of $\pi^{(k)}(\rmM_k(\A))''$ are naturally represented by elements of $\rmM_k(\pi(\A)'')$, and their commutant is given by
\begin{equation}\label{eq:pi-k-commutant}
\pi^{(k)}(\rmM_k(\A))' = \{T^{\oplus k}:\ T \in \pi(\A)'\}
\end{equation}
(see~\cite[Prop. I.2.4(iii)]{Dix--vN}).

We define the \textbf{pairing} of elements $a \in \rmM_k(\A)$ and $\phi \in \frL(\A,\rmM_k)$ by
\begin{equation}\label{eq:pairing}
	\langle \phi,a\rangle := \frac{1}{k}\sum_{ij}\phi_{ij}(a_{ij}).
\end{equation}
The map $\phi\mapsto \langle \phi,\cdot\rangle$ is the \textbf{pairing isomorphism}.  It identifies $\frL(\A,\rmM_k)$ with $\rmM_k(\A)^\ast$.  We henceforth regard either space as carrying the weak$^\ast$ topology.

The pairing isomorphism restricts to a bijection between the closed cones $\frL(\A,\rmM_k)_+$ and $\rmM_k(\A)_+^\ast$: see~\cite[Prop. 1.5.14]{BroOza08}, although their pairing differs from ours by a factor of $k$.  Applied to the type of a tuple of vectors, the pairing isomorphism has the following effect.

\begin{lem}\label{lem:dilation-matrix-dilation}
	If $\phi$ is associated to $\pi$ by the cyclic tuple $v_1,\dots,v_k\in H_\pi$, then $\langle \phi,\cdot\rangle$ is associated to $\pi^{(k)}$ by the cyclic vector $k^{-1/2}[v_1,\dots,v_k]^\rm{T}$. In particular, $\pi_{\langle \phi,\cdot\rangle}$ is equivalent to $\pi_\phi^{(k)}$.
	\qed
\end{lem}

In this lemma, the factor of $k^{-1/2}$ has the effect that an orthonormal tuple gives rise to a unit vector.

Because of positivity, the restriction of the dual norm satisfies
\begin{equation}\label{eq:norm-cts}
\|\langle \phi,\cdot\rangle\|_{\rmM_k(\A)^\ast} = \langle \phi,1\otimes I_k\rangle = \tr_k\phi(1) \qquad (\phi \in \frL(\A,\rmM_k)_+)
\end{equation}
(see~\cite[Prop. 2.1.4]{Dix--Cstar}).  It follows that the pairing isomorphism identifies $\S_k(\A)$ with $\S_1(\rmM_k(\A))$.  The identity~\eqref{eq:norm-cts} has another consequence in the next lemma, for which I have not found a reference.

\begin{lem}\label{lem:lcsc}
The restriction of the weak$^\ast$ topology to $\frL(\A,\rmM_k)_+$ is locally compact and second countable.
\end{lem}

\begin{proof}
By considering pairing functionals on $\rmM_k(\A)$ as in~\eqref{eq:pairing} instead of positive definite maps on $\A$, we may reduce to the case $k=1$.  Having done so, let
\[U_r := \{\phi \in \A^\ast_+:\ \phi(1) < r\} \qquad (r > 0).\]
Each of these sets is relatively weak$^\ast$-open in $\A^\ast_+$.  On the other hand, by~\eqref{eq:norm-cts}, the weak$^\ast$-closure $\ol{U_r}$ is equal to the intersection of $\A^\ast_+$ with a closed ball of radius $r$ in $\A^\ast$.  Therefore $\ol{U_r}$ is metrizable and also compact by the Banach--Alaoglu theorem, and so it is second countable.  This shows that $\A^\ast_+$ is covered by the sequence $U_1$, $U_2$, \dots  of open subsets, each of which is precompact and second countable in the weak$^\ast$ topology. 
\end{proof}

Via the pairing isomorphism, many facts about $\rmM_k$-valued completely positive maps on $\A$ can be reduced to the scalar-valued case for the algebra $\rmM_k(\A)$.  Two standard examples can be found as~\cite[Cor. 1.5.16 and Prop. 1.7.1]{BroOza08}.  We meet some more in the next subsection. 

\subsection{Comparing and decomposing completely positive maps}\label{subs:cp-maps-compar}

Two completely positive maps $\phi$ and $\psi$ are \textbf{disjoint} if $\pi_\phi \perp \pi_\psi$.

Now let $\rho$ be a representation of $\A$, let $\phi \in \frL(\A,\rmM_k)_+$, and let $\pi := \pi_\phi$ for brevity.  Then $\phi$ is \textbf{$\rho$-normal} or \textbf{$\rho$-singular} according as $\pi$ has this property.  By the general description of ultraweakly continuous functionals on a von Neumann algebra~\cite[Thm. I.3.1]{Dix--vN}, $\phi$ is $\rho$-normal if and only if it is equal to $\t{\phi}\circ \rho$ for some normal positive functional $\t{\phi}$ on $\rho(\A)''$, which is then unique.  We call $\t{\phi}$ the \textbf{normal extension} of $\phi$ to $\rho(\A)''$.

Suppose that $\phi$ is associated to $\pi$ by the cyclic $k$-tuple $V$ in $H_\pi$. Let $\pi^M$ be the $\rho$-normal part of $\pi$, as given by Proposition~\ref{prop:Mac}, and let $P$ be the orthogonal projection from $H_\pi$ onto $M$.  Because $P$ commutes with $\pi$, we have $\phi = \phi_{\rm{ac}} + \phi_{\rm{sing}}$, where $\phi_{\rm{ac}} := \Phi^\pi_{PV}$ and $\phi_{\rm{sing}} := \Phi^\pi_{P^\perp V}$.  These two summands are $\rho$-normal and $\rho$-singular respectively.  We refer to this as the \textbf{Lebesgue decomposition} of $\phi$ relative to $\rho$, because it reduces to the classic Lebesgue decomposition from measure theory if $\A = C(\Omega)$ for some compact metrizable space $\Omega$ and we represent positive functionals by Borel measures.  This also explains the choice of subscripts.  Both $\phi_{\rm{ac}}$ and $\phi_{\rm{sing}}$ are still associated to $\pi$ by construction. This decomposition into $\rho$-normal and $\rho$-singular parts is unique by the uniqueness of minimal dilations and the fact that any $\rho$-normal representation is disjoint from any $\rho$-singular representation.  Such decompositions appear in~\cite{Tak58}, where they are described via the enveloping algebra of $\A$. 

If $k > 1$, then by~\eqref{eq:pi-k-commutant} and Lemma~\ref{lem:dilation-matrix-dilation} the Lebesgue decomposition of $\phi$ gives rise to a suitable decomposition of $\langle \phi,\cdot\rangle$ through the pairing isomorphism, and so its uniquenes implies that
\begin{equation}\label{eq:pairing-decomp}
\langle\phi_{\rm{ac}},\cdot\rangle = \langle \phi,\cdot\rangle_{\rm{ac}} \qquad \hbox{and} \qquad \langle\phi_{\rm{sing}},\cdot\rangle = \langle \phi,\cdot\rangle_{\rm{sing}}.
\end{equation}

Accompanying the Lebesgue decomposition, we also have versions of the Radon--Nikodym theorem for completely positive maps.  There are several of these that allow for the non-commutativity of $\A$ in different ways.  The one we need compares a map $\phi \in \frL(\A,\rmM_k)_+$ to a tracial positive functional $\tau$, and originates in work of Dye~\cite{Dye52} for the case $k=1$.  Let $\l$ be the GNS representation of $\tau$ with canonical cyclic tracial vector $\xi$, and let $\M := \l(\A)''$.  Also, in $H_\l^{\oplus k}$, define
\begin{equation}\label{eq:xi-i}
	\xi_i = [0,\dots,0,\xi,0,\dots,0]^\rm{T} \qquad (i=1,2,\dots,k),
	\end{equation}
where only the $i^\rm{th}$ coordinate of $\xi_i$ is nonzero.

When $k=1$, the result has two parts:
\begin{itemize}
\item[a.] If $\phi$ is $\l$-normal, then it is actually associated to $\l$ itself (not just to $\l^{\oplus \infty}$) by some vector $x \in H_\l$: see~\cite[Thm. III.1.4]{Dix--vN}.
\item[b.] Any vector $x \in H_\l$ is equal to $T\xi$ for a unique square-integrable operator $T$ affiliated to $\l(\A)'$: see~\cite[Exers. III.1.13]{Dix--vN} or~\cite[Prop. F.11]{BroOza08}. 
\end{itemize}

When $k > 1$, we can apply parts (a) and (b) above to the functionals $\langle \phi,\cdot\rangle$ and $\tau \otimes \tr_k$ on $\rmM_k(\A) = \A\otimes \rmM_k$.  The GNS representation of $\tau\otimes \tr_k$ is equivalent to the tensor product of $\l$ with the left-multiplication action of $\rmM_k$ on itself, where $\rmM_k$ is a Hilbert space with the normalized Hilbert--Schmidt inner product.  Let us call this tensor product representation $\l^{(k\times k)}$.  Elements of the Hilbert space $H_\l\otimes \rmM_k$ may be written as $k$-by-$k$ arrays with entries in $H_\l$, and then $\l^{(k\times k)}$ acts on such arrays by following the rule for multiplying matrices.  By regarding a $k$-by-$k$ matrix as a $k$-tuple of column vectors, we have $\l^{(k\times k)} \simeq (\l^{(k)})^{\oplus k}$.

If $\phi$ is a $\l$-normal element of $\frL(\A,\rmM_k)_+$, then $\langle \phi,\cdot\rangle$ is $\l^{(k\times k)}$-normal, and so part (a) associates $\langle \phi,\cdot \rangle$ to $\l^{(k\times k)}$ by some $k$-by-$k$ array $[x_{ij}]$ with entries in $H_\l$.  Passing back through the pairing isomorphism, this says that $\phi$ is associated to $\l^{\oplus k}$ by the $k$-tuple of vectors $x_i := [x_{i1},\dots,x_{ik}]^\rm{T}$, $i=1,2,\dots,k$.  Then applying part (b) to each entry lets us write $x_{ij} = T_{ij}\xi$ for some square-integrable operators $T_{ij}$ affiliated to $\l(\A)'$.  Finally, by matrix-vector multiplication, this representation of each $x_{ij}$ may be written as $x_i = T\xi_i$, where $T$ is the square-integrable operator affiliated to $\l^{\oplus k}(\A)'$ that is represented by the transposed array $[T_{ji}]$ (see~\cite[Prop. I.2.4]{Dix--vN}, which extends straightforwardly to affiliated operators).

We have reached the representation $\phi = \Phi^{\l^{\oplus k}}_{T\xi_1,\dots,T\xi_k}$.  If $T_1$ is another square-integrable operator affiliated to $\l^{\oplus k}(\A)'$ which also represents $\phi$ this way, then the uniqueness of minimal dilations shows that $T_1 = RT$ for some partial isometry $R$ in $\l^{\oplus k}(\A)'$, and hence $|T_1| = |T|$.

We collect the conclusions above as follows.

\begin{prop}\label{prop:RadNik}
If $\phi$ is a $\l$-normal element of $\frL(\A,\rmM_k)_+$, then there is an operator $T$ affiliated to $\l^{\oplus k}(\A)'$ such that (i) $\xi_i \in \dom T$ for every $i$ and (ii) $\phi$ is associated to $\l^{\oplus k}$ by the tuple of vectors $T\xi_1$, \dots, $T\xi_k$.  Another affiliated operator $T_1$ also satisfies (i) and (ii) if and only if $|T_1| = |T|$.  In particular, the choice of $T$ is unique if we require it to be non-negative. \qed
\end{prop}

Sometimes we need to consider all the ways in which two matrix-valued completely positive maps could `sit together' inside a larger one.  To describe these, we borrow a term from Furstenberg's classic work~\cite{Fur67} in ergodic theory.  Let $k$ and $\ell$ be positive integers, and let
\[K \:= \{1,\dots,k\} \qquad \hbox{and} \qquad L := \{k+1,\dots,k+\ell\}.\]
For any $(k+\ell)$-by-$(k+\ell)$ matrix $M$, we write $M[K]$ for its $K$-by-$K$ submatrix, and similarly for $L$.

\begin{dfn}\label{dfn:joining}
Let $\phi:\A \to \rmM_k$ and $\psi:\A \to \rmM_\ell$ be completely positive.  A \textbf{joining} of them is a completely positive map $\theta:\A\to \rmM_{k+\ell}$ such that
\[\theta(a)[K] = \phi(a) \qquad \hbox{and} \qquad \theta(a)[L] = \psi(a) \qquad (a \in \A).\]
In particular, the \textbf{diagonal joining} is defined by
\[\rm{diag}(\phi,\psi)(a) := \left[\begin{array}{cc} \phi(a) & 0\\ 0 & \psi(a)\end{array}\right] \qquad (a \in \A).\]
\end{dfn}

This terminology is not standard in representation theory, but it is a convenient way to organize various arguments below.

Comparing with ergodic theory, the diagonal joining of two completely positive maps is the analog of the product of two invariant measures.

If $V = [v_1, \dots, v_k]$ and $W = [w_1,\dots,w_\ell]$ are two tuples in a representation $\pi$, then the combined type $\Phi^\pi_{[V,W]}$ is a joining of $\Phi^\pi_V$ and $\Phi^\pi_W$.  These two tuples generate orthogonal subrepresentations of $\pi$ if and only if
\[\Phi^\pi_{[V,W]} = \rm{diag}(\Phi^\pi_V,\Phi^\pi_W).\]
On the other hand, given any joining $\theta$ of $\Phi^\pi_V$ and $\Phi^\pi_W$, the minimal dilation $\pi_\theta$ contains canonical copies of both tuples.  We may therefore characterize disjointness as follows.

\begin{lem}\label{lem:disjoint-no-join-0}
If $\phi:\A\to \rmM_k$ and $\psi:\A \to \rmM_\ell$ are completely positive, then they are disjoint if and only if they have no joinings other than $\rm{diag}(\phi,\psi)$. \qed
\end{lem}

In fact, starting with Furstenberg's paper~\cite{Fur67}, the uniqueness of the product joining is taken as the definition of `disjointness' in ergodic theory, where the lack of orthogonal complements makes other definitions impractical or senseless.

\subsection{Determinants of completely positive maps}

Let $\l$ be a representation of $\A$ with a cyclic tracial vector $\xi$, let $\tau$ be the associated tracial functional on $\A$, and let $\t{\tau}$ be its normal extension to $\l(\A)''$.  Let $\t{\tau}$ also denote the tracial positive functional defined on $\N := \l(\A)'$ by the same vector $\xi$. Write $\Delta$ for the associated Fuglede--Kadison determinant on square-integrable operators affiliated to either von Neumann algebra.

\begin{dfn}\label{dfn:FK-det-ac}
Let $\phi$ be a $\l$-normal element of $\frL(\A,\rmM_k)_+$, and let $T$ be an operator affiliated to $\l^{\oplus k}(\A)'$ that represents it as in Proposition~\ref{prop:RadNik}.  Then the \textbf{Fuglede--Kadison determinant} of $\phi$ with respect to $\tau$ is
\begin{equation}\label{eq:FK-det-ac}
\Delta_\tau \phi := (\Delta_{\t{\tau}\otimes \Tr_k} |T|)^2.
\end{equation}

For any $\phi \in \frL(\A,\rmM_k)_+$, we lighten notation by defining $\Delta_\tau \phi := \Delta_\tau \phi_{\rm{ac}}$.
\end{dfn}

Since any two possible choices of $T$ in Definition~\ref{dfn:FK-det-ac} differ by a partial isometry, this definition is unambiguous.  Intuitively, the operator $T$ is an equivariant analog of the representation of a non-negative matrix $Q$ as the square of another non-negative matrix $V$, and so~\eqref{eq:FK-det-ac} extends the formula $\det Q = (\det V)^2$.

We sometimes shorten $\Delta_\tau \phi$ to $\Delta \phi$ if the correct choice of $\tau$ is clear.

Definition~\ref{dfn:FK-det-ac} has many near relatives in the literature, particularly for functionals on von Neumann algebras rather than C*-algebras.  One line of these originates in Arveson's study of subdiagonal subalgebras in~\cite{Arv67}, which plays a key role in Section~\ref{sec:past} below.  Another is Araki's notion of the quantum relative entropy between two normal positive functionals on a von Neumann algebra~\cite{Ara7576,Ara7778}. We develop the theory we need from scratch in this subsection, but we take several steps in common with those earlier works. 

Because of the relation~\eqref{eq:pairing-decomp}, we have $\Delta_\tau \phi = \Delta_{\tau\otimes \Tr_k}(\langle \phi,\cdot\rangle)$, even when $\phi$ is not $\l$-normal.  Also, notice the un-normalized trace on the right-hand side of~\eqref{eq:FK-det-ac}. In terms of $\tr_k$ we have instead
\begin{equation}\label{eq:det-phi-det-pairing}
\Delta_\tau \phi = (\Delta_{\t{\tau}\otimes \tr_k} T)^{2k} = \big(\Delta_{\tau\otimes \tr_k}(\langle \phi,\cdot\rangle)\big)^k.
\end{equation}

We sometimes need a variant of $T$ affiliated to $\l(\A)''$ rather than $\l(\A)'$.  To obtain this when $k=1$, choose $T$ non-negative and set $S := JTJ$, where $J$ is the canonical involution from Subsection~\ref{subs:tracial}.  Then the properties of $J$ give $\xi \in \dom S$ and also the relations
\begin{equation}\label{eq:SJTJ1}
S\xi = JTJ\xi = JT\xi = T^\ast \xi = T\xi
\end{equation}
and
\begin{equation}\label{eq:SJTJ2}
\Delta |S| = \exp\int_{[0,\infty)}\log s\ \t{\tau} (JE(ds)J) = \exp\int_{[0,\infty)}\log t\ \t{\tau} E(dt) = \Delta |T|,
\end{equation}
where $E$ is the spectral resolution of $|T|$ and consequently $JE(\cdot )J$ is the spectral resolution of $|S|$.

If $\tau = \tr_k$ on $\A = \rmM_k$, then for positive functionals on $\A$ the definition~\eqref{eq:FK-det-ac} reduces to the usual determinant of a positive semi-definite $k$-by-$k$ matrix.  On the other hand, if $k=1$, $\A = C(\bbT)$, and $\tau$ is integration with respect to $m_\bbT$, then~\eqref{eq:FK-det-ac} equals a logarithmic integral as in Szeg\H{o}'s theorem.  The next example combines and generalizes these two.

\begin{ex}\label{ex:mat-val-fn}
Let $\Omega$ be a compact metrizable space and let $\A := C(\Omega,\rmM_k)$.  Fix a Borel probability measure $\mu$, and let $\tau$ be the tracial state on $\A$ given by integrating $\tr_k$ with respect to $\mu$.  Any continuous linear functional $\phi$ on $\A$ has the form
\[\phi(a) = \int \tr_k\big(a(\omega)\cdot \nu(d\omega)\big) \qquad (a \in \A)\]
for some $\rmM_k$-valued Borel measure $\nu$ on $\Omega$, and $\phi$ is positive if and only if $\nu$ takes values in $\rmM_{k+}$. The functional $\phi_{\rm{ac}}$ is then represented in the same way by $\nu_{\rm{ac}}$, the absolutely continuous part of $\nu$ with respect to $\mu$ in the usual sense of measure theory.  The function $d\nu_{\rm{ac}}/d\mu$ lies in $L^1(\mu;\rmM_{k+})$, so we can define $h \in L^2(\mu;\rmM_{k+})$ by letting $h(\omega)$ be the non-negative square root of $(d\nu_{\rm{ac}}/d\mu)(\omega)$.  Finally, the non-negative operator $T$ from Proposition~\ref{prop:RadNik} is given by pointwise multiplication by $h$ acting on a dense subspace of $L^2(\mu;\bbC^{\oplus k})$, and we evaluate
\[\Delta_\tau \phi = (\Delta_{\t{\tau}}T)^2 = \exp \frac{1}{k}\int \log \det \frac{d\nu_{\rm{ac}}}{d\mu}(\omega)\ d\mu(\omega).\]
\qed
\end{ex}

Returning to the case of a functional $\phi$ on a general C*-algebra $\A$, if $k=1$ then we can approximate $\phi_{\rm{ac}}$ using functionals of the form $\tau(a^\ast(\cdot)a)$ for $a \in \A$.  The next lemma provides such an approximation that simultaneously works `in reverse' and also approximates the determinants.

\begin{lem}\label{lem:Kap+2}
Let $k=1$, let $\phi_{\rm{ac}}$ be associated to $\l$ by $T\xi$ as in Proposition~\ref{prop:RadNik}, let $E_+$ be the orthogonal projection onto $(\ker T)^\perp$, and let $y := E_+\xi$.  Let $A$ be any dense $\ast$-subalgebra of $\A$. Then there is a sequence $(b_n)_{n\ge 1}$ of positive invertible elements of $A$ such that all of the following hold:
\begin{itemize}
\item[i.] $\tau(b_n(\cdot)b_n) \to \phi_{\rm{ac}}$;
\item[ii.] $\phi_{\rm{ac}}(b_n^{-1}(\cdot) b_n^{-1}) \to \Phi^\l_y$;
\item[iii.] $\phi_{\rm{sing}}(b_n^{-1}(\cdot) b_n^{-1}) \to 0$;
\item[iv.] $(\Delta_\tau b_n)^2 \to \Delta_\tau \phi$.
\end{itemize}
\end{lem}

\begin{proof}
\emph{Step 1.}\quad Adjust $T$ so that it is non-negative if necessary, and then let $S:= JTJ$ where $J$ is the canonical involution on $H_\l$, so this satisfies~\eqref{eq:SJTJ1} and~\eqref{eq:SJTJ2}.  In addition, let $\phi_{\rm{sing}}$ be associated to its GNS representation $\pi_{\rm{sing}}$ by the vector $u$.  Overall, $\phi$ is associated to $\pi := \pi_{\rm{sing}} \oplus \l$ by the vector $(u,S\xi)$.

For each $\delta \in (0,1)$, let $S_\delta:= (S\vee \delta)\wedge \delta^{-1}$, and consider the operator
\[R_\delta := \delta^{-1} \ \oplus \ S_\delta\]
in $\frL(H_\pi)$. The projection from $H_\pi$ to $H_\l$ lies in $\pi(\A)' \cap \pi(\A)''$ by Lemma~\ref{lem:central-projection}, and $S_\delta$ commutes with $\l(\A)'$.  Therefore the whole operator $R_\delta$ commutes with $\pi(\A)'$, and so $R_\delta$ lies in $\pi(\A)''$.  We also have $\delta \le R_\delta \le \delta^{-1}$ by construction.

\vspace{7pt}

\emph{Step 2.}\quad As $\delta \downarrow 0$, we have
\[R_\delta^{-1}(u,0) = (\delta u,0) \to 0,\]
and hence $\Phi^\pi_{R_\delta^{-1}(u,0)} \to 0$ by Lemma~\ref{lem:unif-cts}. Similarly, Lemma~\ref{lem:Kap+} gives 
\[R_\delta(0,\xi) = (0,S_\delta\xi) \to (0,S\xi) \qquad \hbox{and hence} \qquad \Phi^\pi_{R_\delta(0,\xi)} \to \Phi^\l_{S\xi} = \phi_{\rm{ac}},\]
\[R_\delta^{-1}(0,S\xi) = (0,S_\delta^{-1}S\xi)\to (0,y) \qquad \hbox{and hence} \qquad \Phi^\pi_{R_\delta^{-1}(0,S\xi)} \to \Phi^\l_y,\]
and also $\Delta_{\t{\tau}} S_\delta \to \Delta_{\t{\tau}} T$.

\vspace{7pt}

\emph{Step 3.}\quad On the other hand, for any fixed $\delta$, we can apply the Kaplansky density theorem~\cite[Sec. I.3.5]{Dix--vN} to the intersection of $\pi(A)$ with the set of all self-adjoint elements $R$ of $\pi(\A)''$ that satisfy $\delta \le R \le \delta^{-1}$.  That theorem gives a sequence $(b_{\delta,n})_{n\ge 1}$ of positive elements of $A$ such that $\delta \le b_{\delta,n} \le \delta^{-1}$ and $\pi(b_{\delta,n}) \to R_\delta$ in the strong operator topology as $n\to\infty$.  Since both inversion and $\log$ can be uniformly approximated by polynomials on the interval $[\delta,\delta^{-1}]$, it follows by functional calculus that we also have $\pi(b_{\delta,n}^{-1}) \to R_\delta^{-1}$ and $\pi (\log b_{\delta,n})\to \log R_\delta$ in the strong operator topology.  From these approximations, we now also obtain
\[\phi_{\rm{sing}}(b_{\delta,n}^{-1}(\cdot)b_{\delta,n}^{-1}) = \Phi^\pi_{\pi(b_{\delta,n}^{-1})(u,0)} \to \Phi^\pi_{R_\delta^{-1}(u,0)},\]
\[\tau(b_{\delta,n}(\cdot)b_{\delta,n})  = \Phi^\pi_{\pi(b_{\delta,n})(0,\xi)} \to \Phi^\pi_{R_\delta(0,\xi)},\]
\[\phi_{\rm{ac}}(b_{\delta,n}^{-1}(\cdot) b_{\delta,n}^{-1}) = \Phi^\pi_{\pi(b_{\delta,n}^{-1})(0,S\xi)} \to \Phi^\pi_{R_\delta^{-1}(0,S\xi)},\]
and $\Delta_\tau b_{\delta,n} \to \Delta_{\t{\tau}} S_\delta$, all as $n\to\infty$ with $\delta$ fixed.

\vspace{7pt}

\emph{Step 4.}\quad Finally, in view of Lemma~\ref{lem:lcsc}, the desired sequence $(b_n)_{n\ge 1}$ can be obtained from Steps 2 and 3 by a diagonal argument.
\end{proof}

One consequence of Lemma~\ref{lem:Kap+2} is the following variational principle.  It is sometimes more convenient than working with Definition~\ref{dfn:FK-det-ac} directly.

\begin{prop}\label{prop:det-var-princ}
Let $A$ be any dense $\ast$-subalgebra of $\rmM_k(\A)$. Let $\tau$ be a tracial state on $\A$, and $\Delta$ its associated Fuglede--Kadison determinant.  Finally, let $\phi$ be an $\rmM_k$-valued completely positive map on $\A$.  Then
\begin{equation}\label{eq:det-var-princ}
(\Delta \phi)^{1/k} = \inf\big\{\langle \phi,a\rangle:\ a \in A\ \hbox{positive and invertible and}\ \Delta_{\tau\otimes \Tr_k}(a) \ge 1\big\}.
\end{equation}
\end{prop}

\begin{proof}
Assume first that $k=1$, and adopt the notation from Lemma~\ref{lem:Kap+2}.  We prove~\eqref{eq:det-var-princ} as a pair of inequalities.

If $a \in A$ is positive, invertible, and satisfies $\Delta a \ge 1$, then
\[(\Delta S)^2 \le \Delta S\cdot \Delta a \cdot \Delta S = \Delta (S\l(a)S) \le \t{\tau}(S\l(a)S) = \phi(a).\]
The second equality holds by the multiplicativity of $\Delta$ (see~\cite[Prop. 2.5]{HaaSch07}), and the second inequality holds by~\eqref{eq:AMGM2} (since we assume that $\tau$ is normalized).  This proves the inequality ``$\le$'' when $k=1$.

On the other hand, let $(b_n)_{n \ge 1}$ be the sequence in $A$ given by Lemma~\ref{lem:Kap+2}, and let $a_n := (\Delta b_n)^2b_n^{-2}$ for each $n$.  Then $\Delta a_n \ge 1$ for each $n$ by construction, while parts (ii), (iii), and (iv) of Lemma~\ref{lem:Kap+2} give
\[\phi(a_n) = (\Delta b_n)^2\cdot(\phi_{\rm{ac}}(b_n^{-2}) + \phi_{\rm{sing}}(b_n^{-2})) \to \Delta \phi \cdot \Phi^\l_y(1).\]
Since $y$ is the projection of $\xi$ onto an invariant subspace of $\l$, and $\tau$ is normalized, we have $\Phi^\l_y(1) \le \tau(1) = 1$.  This proves the inequality ``$\ge$'' when $k=1$.

Finally, if $k > 1$, then we reduce to the scalar-valued case by considering the pairing functional $\langle \phi,\cdot\rangle$ and using~\eqref{eq:pairing-decomp} and~\eqref{eq:det-phi-det-pairing}.  The $k^{\rm{th}}$ root appears on the left-hand side of~\eqref{eq:det-var-princ} because we define $\Delta\phi$ in~\eqref{eq:FK-det-ac} using the un-normalized functional $\tau \otimes \Tr_k$, so we need to convert to using $\tau\otimes \tr_k$ via~\eqref{eq:det-phi-det-pairing}.
\end{proof}

According to Definition~\ref{dfn:FK-det-ac}, the left-hand side of~\eqref{eq:det-var-princ} does not depend on $\phi_{\rm{sing}}$ at all, so neither does the right-hand side.  Referring to the proof of Lemma~\ref{lem:Kap+2}, this is because we can use the first direct summand in the operator $R_\delta$ to suppress the singular part of $\pi$ as much as we like, and then approximate $R_\delta$ arbitrarily well by elements of $\pi(A)$ using the Kaplansky density theorem.

For the trace and determinant on $\rmM_k$, Proposition~\ref{prop:det-var-princ} is a standard inequality of matrix analysis~\cite[Exer. 7.8.P4]{HorJohMA}.  On the other hand, if $k=1$, $\A = C(\Omega)$ for some compact metrizable space $\Omega$, and $\tau$ is integration with respect to a Borel probability meaure $\mu$, then Proposition~\ref{prop:det-var-princ} becomes the classical variational principle for the `reversed' relative entropy (also called Kullback--Leibler divergence) $S(\mu\mid \nu)$~\cite[Lem. 2.3.3]{SimOPUCI}.

Arveson took a variational formula similar to~\eqref{eq:det-var-princ} as his definition of determinants for positive functionals on von Neumann algebras in~\cite[Def. 4.3.7]{Arv67}.  If $k=1$ and $\phi$ is $\l$-normal, then Proposition~\ref{prop:det-var-princ} is essentially~\cite[Prop. 2.1]{BlecLab06}.

We collect several further properties into the next proposition.  Some of them generalize~\cite[Cor. 4.3.3]{Arv67}.

\begin{prop}\label{prop:det-props}
Let $\tau$, $\l$, and $\Delta$ be as above, and let $\phi \in \frL(\A,\rmM_k)_+$ and $\psi \in \frL(\A,\rmM_\ell)_+$.
\begin{itemize}
\item[a.] If $k=\ell$ and $\phi \ge \psi$ in the positive definite order, then $\Delta\phi \ge \Delta\psi$;
\item[b.] If $k = \ell$ and $t \ge 0$, then $\Delta (t\phi) = t^k\Delta \phi$ and $\Delta (\phi + \psi) \ge \Delta \phi + \Delta \psi$;
\item[c.] The function $\Delta$ is upper semicontinuous on $\frL(\A,\rmM_k)_+$ for each $k$;
\item[d.] Any joining $\theta$ of $\phi$ and $\psi$ satisfies $\Delta \theta \le \Delta \phi\cdot \Delta \psi$, with equality if $\theta_{\rm{ac}} = \rm{diag}(\phi_{\rm{ac}},\psi_{\rm{ac}})$ (and so, in particular, if $\theta = \rm{diag}(\phi,\psi)$).
\item[e.] If $k = \ell$ and $\langle \psi,\cdot\rangle = \langle \phi,a^\ast(\cdot)a\rangle$ for some $a \in \rmM_k(\A)$ (see~\eqref{eq:new-type}), then
\begin{equation}\label{eq:FK-transformation}
\Delta_\tau \psi = (\Delta_{\tau \otimes \Tr_k} |a|)^2\Delta_\tau \phi.
\end{equation}
\end{itemize}
\end{prop}

\begin{proof}
Proposition~\ref{prop:det-var-princ} expresses $\Delta \phi$ as an infimum of weak$^\ast$ continuous positive linear functionals.  From here, parts (a), (b) and (c) follow by standard arguments of infinite-dimensional convex analysis: compare~\cite[Thm. I.6.3]{SimSMLG}, for example.

\vspace{7pt}

\emph{Part (d).}\quad First suppose that $T_1$ and $T_2$ are non-negative square-integrable operators affiliated to $\l^{\oplus k}(\A)'$ and $\l^{\oplus \ell}(\A)'$, respectively, and let $T := T_1 \oplus T_2$.  If $E_1$ and $E_2$ are the respective spectral resolutions of $T_1$ and $T_2$, then the spectral resolution of $T$ is given by $E(\cdot) := E_1(\cdot) \oplus E_2(\cdot)$. This gives the calculation of the Fuglede--Kadison determinant of $T$:
\begin{align*}
\log\Delta_{\t{\tau}\otimes \Tr_{k+\ell}}T &= \int_{[0,\infty)}\log t\ (\t{\tau}\otimes \Tr_{k+\ell})(E(dt))\\
&= \int_{[0,\infty)}\log t\ \big((\t{\tau}\otimes \Tr_k)(E_1(dt)) + (\t{\tau}\otimes \Tr_\ell)(E_2(dt))\big)\\
&= \log\Delta_{\t{\tau}\otimes \Tr_k}T_1 + \log\Delta_{\t{\tau}\otimes \Tr_\ell}T_2.
\end{align*}
Exponentiating, and then applying this identity to the operators that represent $\phi_{\rm{ac}}$ and $\psi_{\rm{ac}}$ according to Proposition~\ref{prop:RadNik}, we arrive at the desired equality when $\theta_{\rm{ac}}$ equals $\rm{diag}(\phi_{\rm{ac}},\psi_{\rm{ac}})$.

Now consider an arbitrary joining $\theta$.  Pick any positive and invertible elements $a_1 \in \rmM_k(\A)$ and $a_2 \in \rmM_\ell(\A)$ that satisfy $\Delta_{\tau\otimes \Tr_k}(a_1) \ge 1$ and $\Delta_{\tau\otimes \Tr_\ell}(a_2) \ge 1$.  Pick also a positive real value $r$, and let $a := \rm{diag}(r^\ell a_1,r^{-k}a_2)$, so this lies in $\rmM_{k+\ell}(\A)$.   By the calculation above and part (b), we have
\[\Delta_{\tau\otimes \Tr_{k+\ell}}(a) = \Delta_{\tau\otimes \Tr_k}(r^\ell a_1)\cdot \Delta_{\tau\otimes \Tr_\ell}(r^{-k}a_2) = \frac{r^{k\ell}}{r^{k\ell}}\Delta_{\tau\otimes \Tr_k}(a_1)\cdot \Delta_{\tau\otimes \Tr_\ell}(a_2) \ge 1.\]
On the other hand,
\[\langle \theta,a\rangle = \frac{kr^\ell}{k+\ell}\langle \phi,a_1\rangle + \frac{\ell r^{-k}}{k+\ell}\langle\psi,a_2\rangle.\]
The infimum of this right-hand side over $r$ is equal to $(\langle \phi,a_1\rangle\cdot \langle\psi,a_2\rangle)^{1/(k+\ell)}$.  By Proposition~\ref{prop:det-var-princ}, this implies the desired upper bound on $\Delta \theta$.

\vspace{7pt}

\emph{Part (e).}\quad We prove this when $k=1$.  The matrix-valued case follows from this by considering pairing functionals on $\rmM_k(\A)$.

If $\phi \in \A_+^\ast$ is $\l$-normal and is represented by $T$ as in Proposition~\ref{prop:RadNik}, and if $a \in \A$, then $\psi$ is represented in the same way by $T B$, where $B = J\l(a^\ast)J$ and $J$ is the canonical involution.  Now the result follows from the multiplicativity of $\Delta$ (see~\cite[Prop. 2.5]{HaaSch07}).

Finally, for an arbitrary functional $\phi$, the operator $\pi_\phi(a)$ preserves the $\l$-normal and $\l$-singular parts of $\pi_\phi$, and so $\psi_{\rm{ac}}$ and $\psi_{\rm{sing}}$ are equal to $\phi_{\rm{ac}}(a^\ast(\cdot)a)$ and $\phi_{\rm{sing}}(a^\ast(\cdot)a)$, respectively.
\end{proof}

Since the function $\log$ is monotone and continuous on $(0,\infty)$, conclusions (a) and (c) from Proposition~\ref{prop:det-props} carry over to the expression $\log \Delta \phi$ as well.  As written above, conclusion (b) does not hold for this expression, but using also the concavity of $\log$ we can still conclude that $\log\Delta$ is concave:
\begin{equation}\label{eq:log-det-concave}
\log \Delta (t\phi + (1-t)\psi)  \ge \log(t\Delta \phi + (1-t)\Delta \psi) \ge t\log \Delta \phi + (1-t)\log \Delta \psi.
\end{equation}

The next lemma generalizes~\cite[Cor. 4.3.4]{Arv67}. Let $\Psi:\A \to \B$ be a unital and completely positive map between C*-algebras, let $\tau_\A$ and $\tau_\B$ be tracial states on $\A$ and $\B$ respectively, and assume that $\tau_\A := \tau_\B\circ \Psi$.

\begin{lem}\label{lem:det-ucp}
If $\phi \in \frL(\B,\rmM_k)_+$, then $\Delta_\A(\phi \circ \Psi) \ge \Delta_\B \phi$.  If $\B \subset \A$ and $\Psi|\B$ is the identity, then this is an equality.
\end{lem}

\begin{proof}
First suppose that $k=1$.  Let $a$ be a positive and invertible element of $\A$ with $\Delta_\A a \ge 1$.  The function $\log$ is matrix-concave on $(0,\infty)$~\cite[Prob. 6.6.18]{HorJohTMA}, and this implies that $\log \Psi(a) \ge \Psi(\log a)$~\cite[Thm. 2.1]{Choi74}.  Applying $\tau_\B$, this shows that $\Delta_\B \Psi(a) \ge 1$, and so applying Proposition~\ref{prop:det-var-princ} to $\Delta_\B$ gives $\phi(\Psi(a)) \ge \Delta_\B\phi$.  Taking the infimum over $a$, another appeal to Proposition~\ref{prop:det-var-princ} turns this into $\Delta_\A (\phi\circ \Psi) \ge \Delta_\B\phi$.

On the other hand, if $\B \subset \A$ and $\Psi|\B$ is the identity, then any positive and invertible element $b$ of $\B$ with $\Delta_\B b \ge 1$ is also an element of $\A$ which satisfies $\Delta_\A b \ge 1$ and $(\phi\circ \Psi)(b) = \phi(b)$.  So this time Proposition~\ref{prop:det-var-princ} gives $\Delta_\A (\phi\circ \Psi) \le \Delta_\B\phi$ by taking the infimum over such $b$.

Finally, if $k > 1$, then we reduce to the scalar-valued case by considering the algebras $\rmM_k(\A)$ and $\rmM_k(\B)$ with the tracial states $\tau_\A\otimes \tr_k$ and $\tau_\B\otimes \tr_k$ and the map $\Psi^{(k)}$, which is still completely positive.
\end{proof}

By a classic result of Tomiyama, if $\B\subset \A$, and $\Psi:\A\to \B$ is a linear map of norm one such that $\Psi|\B$ is the identity, then $\Psi$ is completely positive, so Lemma~\ref{lem:det-ucp} can be applied.  Such a map onto a C*-subalgebra is often called a \textbf{conditional expectation} onto that subalgebra: see~\cite[Sec. 1.5]{BroOza08}.  We meet some examples in Section~\ref{sec:past}. 

\section{A Szeg\H{o}-like theorem over amenable groups}\label{sec:amenable}

Throughout this section, $\G$ is a countable group and $\l$ is its left regular representation with the usual cyclic vector $\xi$.  This vector associates the regular character to $\l$.  Unless stated otherwise, we write $\tau$ for the resulting tracial state on $C^\ast \G$, and $\t{\tau}$ for its normal extension to $\l(\G)''$ or the corresponding normal tracial functional on $\l(\G)'$.  In all three cases the associated Fuglede--Kadison determinant is denoted by $\Delta$.

\subsection{Lower bound}\label{subs:amen-LB}

In this subsection we prove the inequality ``$\ge$'' in Theorem~\ref{mainthm:amenable}.  This direction does not require the amenability of $\G$: see Corollary~\ref{cor:amen-LB} below.

The proof of this inequality can be reduced quickly to the case when $\phi$ is $\l$-normal.  For that case, the work is done by a more abstract inequality for von Neumann algebras, given in the next proposition.

\begin{prop}\label{prop:amen-LB}
Let $\M$ be a von Neumann subalgebra of $\frL(H)$.  Let $V = [x_1,\dots,x_k]$ be an orthonormal tuple in $H$ such that the functional
\[\t{\tau}(A) := \frac{1}{k}\sum_{i=1}^k\langle Ax_i,x_i\rangle \qquad (A \in \M)\]
is tracial, and let $\Delta$ be the Fuglede--Kadison determinant associated to $\t{\tau}$.  Finally, let $T$ be a non-negative operator affiliated to $\M$ whose domain contains $x_1$, \dots, $x_k$.  Then
\[\det((TV)^\ast (TV)) \ge (\Delta T)^{2k}.\]
\end{prop}

If $k=1$, then $x_1$ is a tracial vector for $\t{\tau}$, and the inequality above is simply~\eqref{eq:AMGM2} for $T^2$.  Put roughly, we prove the general case by choosing carefully a single vector in the tensor product $H^{\otimes k}$ that reduces the desired inequality to this special case.  More precisely, we first prove Proposition~\ref{prop:amen-LB} when $T$ is bounded and invertible, and then extend to the general case using Lemma~\ref{lem:Kap+}.  (We could merge these steps by working with tensor products of unbounded, densely-defined operators, but the resulting technicalities seem to outweigh the advantages.)

For any bounded operator $A$ on $H$, let
\[A_i := I_H\otimes \cdots \otimes I_H\otimes A\otimes I_H\otimes \cdots \otimes I_H \in \frL(H^{\otimes k}),\]
where $A$ is in the $i^\rm{th}$ position.  The operators $A_1$, \dots, $A_k$ are all still bounded; they commute; and if $A$ is self-adjoint then so is every $A_i$.  The tensor product $T^{\otimes k} \in \frL(H^{\otimes k})$ is equal to the product $A_1A_2\cdots A_k$.

In the coming proof, we apply such tensor products to vectors of the form
\begin{equation}\label{eq:wedge}
x_1\wedge \cdots \wedge x_k := \frac{1}{\sqrt{k!}}\sum_{\pi}\rm{sgn}(\pi)x_{\pi(1)}\otimes \cdots \otimes x_{\pi(k)} \qquad (x_1,\dots,x_k \in H),
\end{equation}
where the sum runs over all permutations of $\{1,2,\dots,k\}$. This vector is called the \textbf{alternating product} of $x_1$, \dots, $x_k$.  Since $H$ is a Hilbert space, the closed span of all alternating product vectors can be identified with the alternating product space $H^{\wedge k}$: see, for instance,~\cite[Sec. 1.5]{SimTIA}.  If $x_1$, \dots, $x_k$ are orthonormal then $x_1\wedge \cdots \wedge x_k$ is a unit vector, and more generally alternating products satisfy
\begin{equation}\label{eq:alt-det-form}
\langle x_1\wedge \cdots \wedge x_k,y_1\wedge \cdots \wedge y_k\rangle = \det[\langle x_i,y_j\rangle]
\end{equation}
(see, for instance,~\cite[equation (1.10)]{SimTIA}).

\begin{proof}[Proof of Proposition~\ref{prop:amen-LB}]
\emph{Step 1.}\quad Assume first that $T$ lies in $\M$ and has a bounded inverse.  Since it is also positive definite, we can define the new self-adjoint operator $S := \log T$ by the functional calculus.  Then $S_i$ is equal to $\log T_i$, because this is a self-adjoint operator whose exponential equals $T_i$, and such an operator is unique.  By the functional calculus for the commuting self-adjoint operators $T_1$, \dots, $T_k$, it follows that
\begin{equation}\label{eq:logTk}
\log T^{\otimes k} = \log(T_1\cdots T_k) = S_1 + \cdots + S_k.
\end{equation}

Recall that $V = [x_1,\dots,x_k]$, and let $z := {x_1\wedge \cdots \wedge x_k}$.  This is a unit vector in $H^{\otimes k}$ because $x_1$, \dots, $x_k$ are orthonormal.  Substituting from~\eqref{eq:wedge}, it satisfies
\[\langle S_1z,z\rangle = \frac{1}{k!}\sum_{\s,\pi}\rm{sgn}(\s\pi) \langle Sx_{\s(1)},x_{\pi(1)}\rangle \langle x_{\s(2)},x_{\pi(2)}\rangle\cdots \langle x_{\s(k)},x_{\pi(k)}\rangle.\]
Since $x_1$, \dots, $x_k$ are orthogonal, the summand on the right vanishes unless $\s(2) = \pi(2)$, \dots, $\s(k) = \pi(k)$, and hence actually $\s = \pi$.  For these summands, we have $\rm{sgn}(\s\pi) = 1$, and every factor of the form $\langle x_{\s(i)},x_{\pi(i)}\rangle$ also equals $1$.  As a result, the equation above simplifies to
\[\langle S_1z,z\rangle = \frac{1}{k!}\sum_\pi \langle Sx_{\pi(1)},x_{\pi(1)}\rangle= \frac{1}{k}\sum_{i=1}^k \langle Sx_i,x_i\rangle= \t{\tau}(S)= \log \Delta T.\]
By symmetry, the analogous formula also holds for $S_2$, \dots, $S_k$.  Adding these together and substituting from~\eqref{eq:logTk}, we arrive at
\begin{equation}\label{eq:klogDeltaT}
k\cdot \log\Delta T = \langle S_1z,z\rangle + \cdots + \langle S_kz,z\rangle = \langle \log T^{\otimes k}z,z\rangle.
\end{equation}
On the other hand, if $E$ is the spectral resolution of $T^{\otimes k}$ on $[0,\infty)$, then
\begin{align*}
\langle \log T^{\otimes k}z,z\rangle &= \int_{[0,\infty)} \log t\ \langle E(dt)z,z\rangle \\
&\le \frac{1}{2}\log \int_{[0,\infty)}t^2\ \langle E(dt)z,z\rangle \\
&= \frac{1}{2}\log \langle T^{\otimes k}z,T^{\otimes k}z\rangle \\
&= \frac{1}{2}\log \det [\langle Tx_i,Tx_j\rangle],
\end{align*}
where we use Jensen's inequality~\cite[Thm. 3.3]{Rudin--realcplx} on the second line and~\eqref{eq:alt-det-form} on the last line.  Combining this calculation with~\eqref{eq:klogDeltaT} completes the proof.

\vspace{7pt}

\emph{Step 2.}\quad Now let $T$ be any unbounded non-negative operator affiliated to $T$ whose domain contains $x_1$, \dots, $x_k$.  Apply Step 1 to the operators $T_\delta$ from Lemma~\ref{lem:Kap+}:
\[(\Delta T_\delta)^{2k} \le \det [\langle T_\delta x_i,T_\delta x_j\rangle].\]
As $\delta \downarrow 0$, this inequality converges to the desired conclusion, by applying parts (ii) and (i) of Lemma~\ref{lem:Kap+} to the left- and right-hand sides, respectively.
\end{proof}

\begin{rmk}
In the notation above, let $\omega$ be the pure state on $\frL(H^{\otimes k})$ defined by the vector $z$. The calculations above show that $\omega(S_1) = \t{\tau}(S)$ for any $S \in \M$, or equivalently that
\[\omega|\M\otimes I_H\otimes \cdots \otimes I_H = \t{\tau}\otimes 1\otimes \cdots \otimes 1.\]
The same holds with $\M$ in any other position in the tensor product, by symmetry.  However, $\omega|\M^{\otimes k}$ is typically not equal to $\t{\tau}^{\otimes k}$: indeed, the formula for $\langle T^{\otimes k}z,z\rangle$ as a determinant would violate this.  This is why the application of Jensen's inequality must be written out in terms of $z$ and $E$, not simply as an instance of the infinitary determinant-trace inequality~\eqref{eq:AMGM2}. \fin
\end{rmk}

Now let $\phi:\G\to\rmM_k$ be a positive definite function.  As in Theorem~\ref{mainthm:amenable}, for any finite subset $F$ of $\G$, we consider the $F$-by-$F$ block matrix
\[\phi[F] := [\phi(g^{-1}h):\ g,h \in F].\]

\begin{cor}\label{cor:amen-LB}
If $F$ is finite and nonempty, then
\[\det \phi[F] \ge (\Delta \phi)^{|F|},\]
where $\Delta$ is the Fuglede--Kadison determinant associated to the regular character.
\end{cor}

\begin{proof}
The Lebesgue decomposition gives $\phi[F_n] \ge \phi_{\rm{ac}}[F_n]$ in the positive definite ordering for every $n$, and so their determinants are ordered the same way~\cite[Corollary 7.7.4(e)]{HorJohMA}.  We may therefore discard $\phi_{\rm{sing}}$ and assume that $\phi$ is $\l$-normal.

Let $H := H_\l$, let $\t{\tau}$ be normal tracial state constructed from $\l$ and $\xi$ on $\l(\G)'$, and let $\M := \l^{\oplus k}(\G)'$.  Let $\xi_1$, \dots, $\xi_k$ be the cyclic $k$-tuple for $\l^{\oplus k}$ as in~\eqref{eq:xi-i}. This $k$-tuple satisfies
\begin{equation}\label{eq:matrix-tau}
(\t{\tau}\otimes \tr_k)(A) = \frac{1}{k}\sum_{i=1}^k\langle A\xi_i,\xi_i\rangle \qquad (A\in \M).
\end{equation}

Since $\phi$ is $\l$-normal, Proposition~\ref{prop:RadNik} gives a non-negative operator $T$ affiliated to $\M$ such that $\xi_1,\dots,\xi_k \in \dom T$ and
\begin{equation}\label{eq:phi-T}
\phi(g) = [\langle \l^{\oplus k}(g)T\xi_j,T\xi_i\rangle]_{i,j} \qquad (g \in \G).
\end{equation}

Now define an orthonormal $k|F|$-tuple in $H$ by
\[V := [\l^{\oplus k}(g)\xi_i:\ i=1,\dots,k,\ g \in F].\]
Regarded as a unitary embedding from $\bbC^{\oplus k|F|}$ into $H$, this tuple satisfies
\[\frac{1}{k|F|}\Tr(V^\ast A V) = \frac{1}{|F|}\sum_{g \in F}\frac{1}{k}\sum_{i=1}^k \langle A\l^{\oplus k}(g)\xi_i,\l^{\oplus k}(g)\xi_i\rangle \qquad (A \in \M).\]
Since $A$ commutes with $\l^{\oplus k}$, this simplifies to the expression in~\eqref{eq:matrix-tau}.  On the other hand, the definition of $V$ and the fact that $T$ commutes with $\l^{\oplus k}$ give
\begin{align*}
\phi[F] &= [\phi(g^{-1}h):\ g,h \in F] \\
&= [\langle\l^{\oplus k}(h)T\xi_j,\l^{\oplus k}(g)T\xi_i\rangle:\ i,j=1,\dots,k,\ g,h \in F] \\
&= [\langle T\l^{\oplus k}(h)\xi_j,T\l^{\oplus k}(g)\xi_i\rangle:\ i,j=i,\dots,k,\ g,h \in F] \\
&=(TV)^\ast (TV).
\end{align*}
Because of this calculation and~\eqref{eq:matrix-tau}, we can now apply Proposition~\ref{prop:amen-LB} to obtain
\[\det \phi[F] = \det ((TV)^\ast (TV)) \ge (\Delta_{\t{\tau}\otimes \tr_k}T)^{2k|F|} = (\Delta \phi)^{|F|},\]
recalling~\eqref{eq:det-phi-det-pairing} for the final equality.
\end{proof}

\subsection{Upper bound and completed proof of Theorem~\ref{mainthm:amenable}}\label{subs:amen-UB}

Our proof of the inequality ``$\le$'' in Theorem~\ref{mainthm:amenable} uses the variational principle from Proposition~\ref{prop:det-var-princ}.  This saves us from having to handle $\phi_{\rm{sing}}$ explicitly: it has already been controlled inside the proof of that principle.

This direction does require the right F\o lner property of $(F_n)_{n \ge 1}$.  We apply it through the next lemma and its corollary.

\begin{lem}\label{lem:phiFnAn}
Let $\phi:\G\to\rmM_k$, let $a:\G\to\rmM_k$ be finitely supported, and let $(F_n)_{n \ge 1}$ be a right F\o lner sequence.  There are subsets $E_n$ of $F_n$ such that $|F_n\setminus E_n| = o(|F_n|)$ and
\[(\phi[F_n]\cdot a[F_n])(g,h) = (\phi\ast a)[F_n](g,h) \qquad \hbox{whenever}\ (g,h) \in F_n \times E_n.\]
\end{lem}

\begin{proof}
Let $S := \{h:\ a(h) \ne 0\}$, so this is finite by assumption, and now let
\[E_n := \{h \in F_n:\ hS^{-1} \subset F_n\} = F_n\cap\bigcap_{s \in S}(F_ns) \qquad (n=1,2,\dots).\]
This satisfies $|F_n\setminus E_n| = o(|F_n|)$ by the right F\o lner property of $(F_n)_{n\ge 1}$.

For any $g,h \in F_n$, the definition~\eqref{eq:convo} gives
\[(\phi\ast a)[F_n](g,h) = (\phi\ast a)(g^{-1}h) = \sum_k \phi(g^{-1}k)a(k^{-1}h).\]
In this sum, the factor $a(k^{-1}h)$ is nonzero only if $k^{-1}h \in S$, or equivalently $k \in hS^{-1}$.  If $h \in E_n$, then $hS^{-1} \subset F_n$, so for these $h$ the sum above agrees with
\[\sum_{k\in F_n} \phi(g^{-1}k)a(k^{-1}h) = (\phi[F_n]\cdot a[F_n])(g,h).\]
\end{proof}

\begin{cor}\label{cor:phiFnAn}
Let $(F_n)_{n\ge 1}$ be a right F\o lner sequence.  If $\phi:\G\to \rmM_k$ is bounded and $a:\G\to\rmM_k$ is finitely supported, then
\[\tr_{k|F_n|}\big(\phi[F_n]\cdot a[F_n]\big) \to \tr_k((\phi\ast a)(e)) \qquad \hbox{as}\ n\to\infty.\]
\end{cor}

\begin{proof}
By Lemma~\ref{lem:phiFnAn}, there are subsets $E_n$ of $F_n$ such that $|F_n\setminus E_n| = o(|F_n|)$ and such that the matrix
\[\phi[F_n]\cdot a[F_n] - (\phi\ast a)[F_n]\]
vanishes in all columns indexed by $E_n$.  Its remaining entries can be bounded using the operator norm $\|\cdot\|$ on $\rmM_k$ and the triangle inequality, with the result that
\begin{multline*}
\big|\tr_{k|F_n|}\big(\phi[F_n]\cdot a[F_n]\big) - \tr_{k|F_n|}\big((\phi\ast a)[F_n])\big| \\ \le 2\cdot \Big(\sum_g\|a(g)\|\Big)\cdot \Big(\sup_g\|\phi(g)\|\Big)\cdot \frac{|F_n\setminus E_n|}{|F_n|} \to 0.
\end{multline*}
Finally, every diagonal block of $(\phi\ast a)[F_n]$ is simply equal to $(\phi\ast a)(e)$, so
\[\tr_{k|F_n|}\big((\phi\ast a)[F_n]) = \tr_k((\phi\ast a)(e)).\]
\end{proof}

The other ingredient we need to prove Theorem~\ref{mainthm:amenable} is a special case of that theorem which already appears in the literature.

\begin{prop}\label{prop:Den}
Let $(F_n)_{n\ge 1}$ be a right F\o lner sequence.  Suppose that ${a:\G\to\rmM_k}$ is finitely supported, positive definite, and invertible in $\rmM_k(C^\ast \G)$.  Then
\[(\det a[F_n])^{1/k|F_n|} \to \Delta_{\tau \otimes \tr_k} a \qquad \hbox{as}\ n\to\infty.\]
\qed
\end{prop}

The earliest reference I know that includes Proposition~\ref{prop:Den} for general amenable groups is~\cite[Thm. 3.2]{Den06}.  The proof starts with estimates similar to Lemma~\ref{lem:phiFnAn}, but in which all functions on $\G$ are finitely supported.  Using these, one shows by induction on $d$ that
\[\tr_{k|F_n|}\big(a[F_n]^d\big) \to \tr_k(a^{\ast d}(e)) \qquad \hbox{as}\ n\to\infty,\]
and then by taking linear combinations that
\[\tr_{k|F_n|}(p(a[F_n])) \to \tr_k(p(a)(e)) \qquad \hbox{as}\ n\to\infty\]
for any polynomial $p$.  Finally, the convergence of determinants follows by approximating $\log$ uniformly by polynomials on a compact subinterval of $(0,\infty)$.

\begin{proof}[Proof of Theorem~\ref{mainthm:amenable}]
We prove~\eqref{eq:amenable} as a pair of inequalities.

\vspace{7pt}

\emph{Step 1.}\quad The inequality ``$\ge$'' holds for every $n$ individually by Corollary~\ref{cor:amen-LB}.

\vspace{7pt}

\emph{Step 2.}\quad Let $\phi:\G\to\rmM_k$ be positive definite, and let $a\in \rmM_k(\bbC[\G])$ be positive definite, invertible in $\rmM_k(C^\ast \G)$, and satisfy $\Delta_{\tau\otimes \tr_k} a \ge 1$.  Then Proposition~\ref{prop:Den} gives
\begin{equation}\label{eq:dets-conv}
(\det a[F_n])^{1/k|F_n|} \to \Delta_{\tau\otimes \tr_k} a \ge 1\qquad \hbox{as}\ n\to\infty.
\end{equation}

Define $\phi^\rm{t}$ from $\phi$ as in equation~\eqref{eq:phi-hatphi}.  This satisfies
\[\phi^\rm{t}[F_n] = [\phi^\rm{t}(g^{-1}h):\ g,h \in F_n] = [\phi(h^{-1}g)^\rm{T}:\ g,h \in F_n] = (\phi[F_n])^\rm{T},\]
where the right-hand side is the transpose of $\phi[F]$ as a $k|F|$-by-$k|F|$ matrix.  Consequently, $\phi[F_n]$ and $\phi^\rm{t}[F_n]$ have the same determinant.

Now the multiplicativity of determinants, the determinant-trace inequality~\eqref{eq:AMGM}, and the trace property give
\begin{align*}
(\det \phi[F_n])^{1/k|F_n|}\cdot (\det a[F_n])^{1/k|F_n|} 
&= \big(\det (\sqrt{a[F_n]}\cdot \phi^\rm{t}[F_n]\cdot \sqrt{a[F_n]})\big)^{1/k|F_n|} \\
&\le \frac{1}{k|F_n|}\Tr (\sqrt{a[F_n]}\cdot \phi^\rm{t}[F_n]\cdot \sqrt{a[F_n]})\\
&= \frac{1}{k|F_n|}\Tr (\phi^\rm{t}[F_n]\cdot a[F_n]).
\end{align*}
By Corollary~\ref{cor:phiFnAn} and the calculation~\eqref{eq:matrix-pos-def-convo}, the last normalized trace converges to
\[\tr_k((\phi^\rm{t}\ast a)(e)) = \langle \phi,a\rangle\]
as $n\to\infty$. Combining this with~\eqref{eq:dets-conv}, we have shown that
\[\limsup_{n\to\infty} (\det \phi[F_n])^{1/k|F_n|} \le \langle \phi,a\rangle.\]
Taking the infimum over $a$, Proposition~\ref{prop:det-var-princ} completes the proof of ``$\le$''.
\end{proof}

\subsection{Further remarks}\label{subs:A-rmks}

\subsubsection*{\emph{Comparison with previous work}}

Our proof of Theorem~\ref{mainthm:amenable} has elements in common with various proofs of Szeg\H{o}'s theorem itself.  Here are two examples:
\begin{itemize}
\item The proof of Szeg\H{o}'s theorem presented in~\cite[Sec. 2.3]{SimOPUCI}, which is a version of Verblunsky's original proof, uses a classical predecessor of the variational principle from Proposition~\ref{prop:det-var-princ}.  But there it is used only to establish weak$^\ast$ upper semicontinuity of the right-hand side of~\eqref{eq:Szeg} as a function of $\mu$.  This is a preparation for that proof of Szeg\H{o}'s theorem, but not really an application inside the proof itself.  In other respects, that proof of Szeg\H{o}'s theorem seems essentially disjoint from our proof of Theorem~\ref{mainthm:amenable}.

\item In case $\mu \ll m$ in Szeg\H{o}'s theorem, the inequality ``$\ge$'' is often proved by an application of Jensen's inequality.  This step essentially corresponds to our use of Proposition~\ref{prop:amen-LB} above.
\end{itemize}

On the other hand, as discussed in the Introduction, most traditional proofs of Szeg\H{o}'s theorem make rather explicit use of a notion of the `past' of a vector under a unitary operator.  In the first place, this refers to the subspace $N$ in~\eqref{eq:Szeg-again}.  In the setting of Theorem~\ref{mainthm:amenable} we must do without this structure.  A meaning of `past' reappears when we turn to Theorem~\ref{mainthm:past} in the next section.

In this respect, we take inspiration from research in ergodic theory that studies entropy without conditioning on the past as in~\eqref{eq:entropy-past}.  This program is discussed more fully in~\cite{GlaThoWei00}. Regarding Fuglede--Kadison determinants and generalizations of Szeg\H{o}'s theorem itself, some other recent precedents for our work in ergodic theory also have this flavour.  We discuss these next.

For a polynomial $f$ in $d$ variables, a log-integral much like the right-hand side of~\eqref{eq:Szeg} defines its `Mahler measure'.  If $f$ has integer coefficients, then it can be used to construct an action of $\bbZ^d$ by automorphisms of a compact Abelian group.  This action necessarily preserves the Haar measure of that compact group, and the Kolmogorov--Sinai entropy of this measure-preserving action turns out to equal the logarithm of the Mahler measure of $f$~\cite{LindSchWar90}.

Starting from the observation that a Mahler measure is a Fuglede--Kadison determinant, Deninger generalized this construction to allow $f \in \rmM_k(\bbZ[\G])$ for any positive integer $k$ and countable group $\G$.  He conjectured that the Kolmogorov--Sinai entropy of the resulting system $\bs{X}_f$ should be the Fuglede--Kadison determinant of $f$ whenever $\G$ is amenable and $f$ is non-singular as an operator on $\ell^2(\G)^{\oplus k}$. He proved this under various extra hypotheses: see the description with further references in~\cite{Den09}.

Deninger proved~\cite[Thm. 3.2]{Den06} as a step towards those results.  This theorem contains our Proposition~\ref{prop:Den}.  As Deninger describes, similar arguments already appear in earlier works on $L^2$-invariants in algebraic topology such as~\cite{Luc94,DodMat98,Schi01}.

Li made further progress in~\cite{Li12}, where he proved Deninger's conjectured entropy formula whenever $\G$ is amenable and $f$ is invertible in $\rmM_k(\l(\G)'')$.  Among Li's technical ingredients,~\cite[Corollary 7.2]{Li12} strengthens Proposition~\ref{prop:Den} by allowing certain additional perturbations to each of the matrices $\phi[F_n]$.

Finally, Li and Thom proved Deninger's full conjecture in~\cite{LiTho14}, and generalized it further to a larger class of actions by automorphisms of compact Abelian groups.  This work needed another strengthening of Proposition~\ref{prop:Den} as an ingredient: see~\cite[Thm. 1.4]{LiTho14}, which allows positive elements of $\rmM_k(\l(\G)'')$ that are not necessarily invertible.  This is equivalent to the case of Theorem~\ref{mainthm:amenable} when $\phi$ is $\l$-normal and the operator $T$ from Proposition~\ref{prop:RadNik} is bounded.

Apart from these points of contact, the other details of our proof of Theorem~\ref{mainthm:amenable} are largely disjoint from those previous works.  The most obvious difference is our use of the variational principle from Proposition~\ref{prop:det-var-princ}.  In addition, not all aspects of Szeg\H{o}'s theorem make an appearance in those previous papers: for example, they all assume that $f$ already lies in $\rmM_k(\l(\G)'')$, so there is no singular part to take care of.

\subsubsection*{\emph{Difficulties of extension beyond amenable groups}}

It is natural to ask about generalizations of Theorem~\ref{mainthm:amenable} to non-amenable groups or other C*-algebras.  Various directions suggest themselves, but none of them seems entirely straightforward.  We quickly discuss four of these here.  For simplicity, we now restrict our attention to scalar-valued positive definite functions and positive functionals.

First, Corollary~\ref{cor:amen-LB} immediately gives the following: in the notation of Theorem~\ref{mainthm:amenable}, if $\G$ is any countable group and $\phi:\G\to \bbC$ is positive definite, then
\begin{equation}\label{eq:amenable-inf-version}
\Delta \phi_{\rm{ac}} \le \inf\big\{(\det \phi[F])^{1/|F|}:\ F \subset \G\ \hbox{finite and nonempty}\big\}.
\end{equation}
Theorem~\ref{mainthm:amenable} shows that this is an equality if $\G$ is amenable.  I suspect that this implication can be reversed, even if we allow only certain positive definite functions:

\begin{prob}\label{prob:characterize-amenable}
Suppose that equality holds in~\eqref{eq:amenable-inf-version} whenever $\phi = \tau(a^\ast (\cdot)a)$ for some $a \in \bbC[\G]$.  Must $\G$ be amenable?
\end{prob}

If the answer here is positive and $\G$ is not amenable, then it might be worth looking more closely at the `gap' in~\eqref{eq:amenable-inf-version}, perhaps for special elements of $\bbC[\G]$ such as the Laplacians corresponding to finite symmetric subsets of $\G$.

\begin{prob}
For particular choices of $a$ as in Problem~\ref{prob:characterize-amenable}, how does the gap in~\eqref{eq:amenable-inf-version} relate to other measures of non-amenability such as the isoperimetric profile of finite subsets of $\G$?
\end{prob}

For our second direction, let us consider the possibility of replacing finite subsets of $\G$ in Theorem~\ref{mainthm:amenable} with finite quotients.  This moves us to the class of residually finite groups.  By allowing quotients with small defects, one could extend this further to sofic groups~\cite{KwiPes13}, but we leave these aside here.

If $\G$ is residually finite, then it has a sequence of permutation representations $\s_n:\G \to \rm{Sym}(V_n)$ on finite sets such that the kernels $\ker \s_n$ have trivial intersection.  By replacing each $\s_n$ with the diagonal action of many copies of $\s_n$ if necessary, we may assume further that this sequence is asymptotically free, meaning that
\begin{equation}\label{eq:asymp-free}
|\{v\in V_n:\ \s_n(g)v = v\}| = o(|V_n|) \qquad \hbox{as}\ n\to\infty\ \hbox{for every}\ g \in \G\setminus e.
\end{equation}

Let $\pi_n$ be the unitary representation on $\bbC^{\oplus V_n}$ induced by $\s_n$.  Then~\eqref{eq:asymp-free} implies the convergence $\tr_{|V_n|} \pi_n(a) \to \tau(a)$ for any $a \in \bbC[\G]$, and then also for any $a \in C^\ast \G$ by approximation in norm.  From here one quickly reaches an analog of Proposition~\ref{prop:Den}: if $a \in C^\ast \G$ is positive and invertible, then
\begin{equation}\label{eq:asymp-free-det}
(\det \pi_n(a))^{1/|V_n|} \to \Delta a \qquad \hbox{as}\ n\to\infty.
\end{equation}
The proof is closely analogous to that of Proposition~\ref{prop:Den}.  See~\cite[Thm. 6.1]{DenSch07} (which makes the slightly stronger assumption that $a$ is invertible in the Banach algebra $\ell^1(\G)$) and \cite[Lem. 7.2 and Thm. 7.3]{KerLi11b}.

However, beyond~\eqref{eq:asymp-free-det} for positive and invertible elements $a$, a couple of serious difficulties quickly present themselves.
\begin{enumerate}
\item Suppose that $a$ is non-negative but not necessarily invertible.  The proof of~\eqref{eq:asymp-free-det} depends on approximating $\log$ by polynomials, and such approximations break down near the origin.  On the other hand, our proof of Proposition~\ref{prop:amen-LB} considers alternating products in the fixed Hilbert space of the regular representation, and has no obvious modification for the sequence of spaces $\bbC^{\oplus V_n}$.  Absent either of these arguments, we obtain only
\[\limsup_{n\to\infty}(\det \pi_n(a))^{1/|V_n|} \le \Delta a.
\]
It turns out that this inequality really can be strict: see~\cite[Ex. 13.69]{Luc02}.  This happens if $a$ is non-singular and satisfies $\Delta a > -\infty$, but the finitary matrices $\pi_n(a)$ have a few extremely small eigenvalues that drag their determinants down far below $\Delta a$.

The inequality above is still enough for some valuable applications.  An example is Elek and Szab\'o's proof of L\"uck's determinant conjecture for sofic groups in~\cite{EleSza05}; see also~\cite{BalSka16}.

\item The convergence in~\eqref{eq:asymp-free-det} can be phrased as a fact about the positive definite function $\phi = \tau(a^\ast(\cdot)a)$.  The square of the left-hand side of~\eqref{eq:asymp-free-det} is the determinant of a finite-dimensional matrix that approximates $\Delta \phi$.  But for more general positive definite functions, it is not even clear how to choose finite-dimensional matrices that could serve in such an approximation.  For a finite subset $F$ of $\G$, we can always form $\phi[F]$ simply by restricting, but no obvious analog of `restriction' applies to give matrices over the sets $V_n$.  This is already a problem even if $\phi$ is associated to the regular representation $\l$, but not by a vector of the form $\l(a)\xi$ for some $a \in \A$.
\end{enumerate}

Hayes discusses the first of these difficulties further near the end of~\cite[Intro.]{Hayes16}.  That paper concerns an analog of Deninger's entropy formula for certain dynamical systems of algebraic origin over sofic groups.  For Hayes, the change of focus to sofic entropy meant that he could sidestep the first difficulty above, and the second did not arise because sofic entropy is not defined directly as a limit of normalized Shannon entropy values.  We meet a similar situation in our study of almost periodic entropy and Theorem~\ref{mainthm:det-form} below.  See also the discussion in Subsection~\ref{subs:C-rmks}.

In a third direction, we could ask about generalizations of Theorem~\ref{mainthm:amenable} to other C*-algebras besides group algebras, but retaining an assumption to play the role of `amenability'.

A simple motivating result can be obtained if $\M$ is a finite von Neumann subalgebra of $\frL(H)$ and $\t{\tau}$ is a faithful normal tracial state on $\M$.  Let $\Delta$ be the Fuglede--Kadison determinant associated to $\t{\tau}$.  Let $k_1$, $k_2$, \dots be a divergent sequence of positive integers, let $V_n$ be an orthonormal $k_n$-tuple in $H$ for each $n$, and let $P_n := V_nV_n^\ast$ (the orthogonal projection from $H$ onto $\rm{img}\,V_n$).  Finally, assume the following two properties:
\begin{itemize}
\item $\tr_{k_n}(V_n^\ast AV_n) = \t{\tau}(A)$ for every $A \in \M$, and
\item $\|P_nA - AP_n\|_2 = o(\sqrt{k_n})$ as $n\to\infty$ for every $A \in \M$, where $\|\cdot\|_2$ is the Hilbert--Schmidt norm.
\end{itemize}
Then
\[(\det V_n^\ast A V_n)^{1/k_n}\to \Delta A \qquad \hbox{as}\ n\to\infty\ \hbox{for every non-negative}\ A \in \M.\]
Indeed, the inequality ``$\ge$'' holds for every $n$ by another appeal to Proposition~\ref{prop:amen-LB}, and its reverse holds in the limit by adapting the proof of Proposition~\ref{prop:Den}.

Sequences $(V_n)_{n\ge 1}$ satisfying the two properties above have long-standing importance in the study of hyperfiniteness~\cite[Chap. III.7]{Dix--vN}.  Indeed, among factors of type II$_1$, the existence of such a sequence identifies uniquely the hyperfinite one: see~\cite[Thm. 5.1]{Con76} and also the related discussion in~\cite[Sec. 1]{Arv77}.

Motivated by this result for von Neumann algebras, one could try to push Theorem~\ref{mainthm:amenable} towards a class of abstract C*-algebras that satisfy a substitute for amenability such as nuclearity~\cite[Sec. 2.3]{BroOza08} or quasidiagonality~\cite[Chap. 7]{BroOza08}.  However, I do not know precisely what statement one should try to prove.  It would need to be formulated carefully to prevent versions of difficulties (1) and (2) above from re-emerging. 

\begin{prob}
Formulate and prove a generalization of Theorem~\ref{mainthm:amenable} for positive functionals on an abstract class of C*-algebras that generalizes the class of group C*-algebras of countable amenable groups.
\end{prob}

For a concrete family of examples lying in this direction, one could start with C*-algebras of amenable groupoids: see~\cite[Sec. 5.6]{BroOza08} for an introduction and further references.

\section{Invariant random orders}\label{sec:past}

Throughout this section, $\G$ is a countable group, $\l$ is its left regular representation with the usual cyclic unit vector $\xi$, and $\chi$ denotes both the regular character on $\G$ and also its extension to a tracial state on $C^\ast \G$. (We use $\tau$ for a different functional in this section: see~\eqref{eq:trace-on-A0} below.)

Our proof of Theorem~\ref{mainthm:past} rests on a major result about Arveson's subdiagonal subalgebras.  When he introduced these subalgebras of finite von Neumann algebras in~\cite{Arv67}, Arveson already proposed a generalized Szeg\H{o}-type theorem for them, and showed its equivalence to two variants of Jensen's inequality: see the discussion around~\cite[properties 4.4.($\a$),($\b$),($\g$)]{Arv67}.  Arveson also proved all three assertions for several concrete examples in that paper, but not in general.  The general case was finally proved by Labuschagne in~\cite{Lab05}. We use this in the proof of Theorem~\ref{thm:subdiag-Szego}, which is itself another variant of Szeg\H{o}'s theorem, and which then implies Theorem~\ref{mainthm:past} fairly directly.  However, we first need some preparations to make contact with this machinery.


\subsection{Invariant random orders and C*-crossed products}\label{subs:inv-rndm}

We identify the set of all relations on $\G$ with $\{0,1\}^{\G\times \G}$, and give it the resulting product topology, which is compact and metrizable.  The subset $\Omega$ of all total orders is nonempty, closed, and invariant under the diagonal action of $\G$ on $\G\times \G$ by left translation.  For the sake of more familiar notation, we write ``$g <_\omega h$'' instead of just ``$g \omega h$'' when $\omega \in \Omega$.  We denote the action itself by $(g,\omega) \mapsto g\cdot \omega$, so it can be expressed like this:
\begin{equation}\label{eq:order-equiv}
gh <_{g\cdot \omega} gk \qquad \Leftrightarrow \qquad h <_\omega k \qquad (g,h,k \in \G).
\end{equation}

Let $\frD:= C(\Omega,\rmM_k)$, regarded as a C*-algebra.  Then $\G$ acts on $\frD$ by pre-composition, and we can form the resulting full C*-crossed product~\cite[Sec. 4.1]{BroOza08}.  In the sequel we need a specific construction of it.  Let $\A_0$ be the vector space of all continuous $\rmM_k$-valued functions on $\Omega\times \G$ with compact support, made into a *-algebra with these twisted versions of multiplication and involution:
\begin{equation}\label{eq:twisted}
(a\ast b)(\omega,g) := \sum_h a(\omega,h)b(h^{-1}\cdot \omega,h^{-1}g) \quad \hbox{and} \quad a^\ast(\omega,g) := a(g^{-1}\cdot \omega,g^{-1})^\ast.
\end{equation}
Then $\A_0$ is the linear span of elements that have the form
\[(d\otimes \delta_g)(\omega,h) := d(\omega)\delta_g(h) \qquad (\omega \in \Omega, h \in \G)\]
for some $d \in \frD$ and $g \in \G$.

Representations of $\A_0$ correspond to covariant representations of $\frD$ and $\G$. In particular, any representation $\a$ of $\frD$ can be combined with $\l$ to form a covariant representation on $H_\a\otimes \ell^2(\G)$ as in the paragraph prior to~\cite[Def. 4.1.4]{BroOza08}, and this family of examples is faithful on $\A_0$.  Because of these examples, and often many others as well, the maximal C*-seminorm on $\A_0$ is a norm.  The full C*-crossed product $\A$ is the resulting completion.  We henceforth identify $\frD$ with a C*-subalgebra of $\A$ through the map $d \mapsto d\otimes \delta_e$.  In the reverse direction, we write $a_g := a(\cdot,g)$ when $a \in \A_0$ and $g \in \G$.  Similarly, we identify $\G$ with the subset $\{1_\frD\otimes \delta_g:\ g \in \G\}$ of $\A_0$ .

Moreover, if $a \in \A_0$ and $\rho$ is the representation of $\A_0$ constructed from $\a$ and $\l$, then the compression of $\rho(a)$ to $H_\a\otimes \delta_e$ is unitarily equivalent to $\a(a_e)$.  Since we can choose $\a$ to be faithful, this shows that the map $a \mapsto a_e$ extends to a conditional expectation $\Psi_\frD:\A\to \frD$. 

An \textbf{invariant random order} is a $\G$-invariant Borel probability measure on $\Omega$.  For example, if $<$ is a left-invariant total order on $\G$, then $\delta_<$ is an invariant random order.  An arbitrary countable group $\G$ may not have a left-invariant total order, but it does always have an invariant random order via the following construction.

\begin{ex}\label{ex:Bern}
Let $\ol{\Omega} := [0,1]^\G$, and let $\ol{\mu}$ be the product measure under which each coordinate is uniformly distributed. Let $\G$ act by left-translation: if $\ol{\omega} = (\ol{\omega}_h)_{h \in \G}$, then $g\cdot \ol{\omega} =  (\ol{\omega}_{g^{-1}h})_{h \in \G}$.  Now define a $\G$-equivariant map $\ol{\omega}\mapsto <_{\ol{\omega}}$ by
\[g <_{\ol{\omega}} h \qquad \Leftrightarrow \qquad \ol{\omega}_g < \ol{\omega}_h.\]
If $\ol{\omega}$ is drawn at random from $\ol{\mu}$, then $<_{\ol{\omega}}$ is almost surely a total order because the event $\{\ol{\omega}_g = \ol{\omega}_h\}$ is negligible for any distinct $g$ and $h$.  Its distribution on $\Omega$ is an invariant random order called the \textbf{Bernoulli random order} on $\G$.  It is also characterized by the property that it orders any finite subset of $\G$ uniformly at random.  This example is discussed more fully in~\cite[Sec. 7]{Seward--randomorder}. \qed 
\end{ex}

Fix an invariant random order $\mu$ for the rest of this section.  We can use it to construct another family of representations of $\A_0$ as follows.  First, for any Hilbert space $H$, let $\rmM_k$ act on $H^{\oplus k}$ by following the rules of matrix-vector multiplication, as previously. Now let $\pi$ be a representation of $\G$, and define a representation $\Pi$ of $\A_0$ on $L^2(\mu;H_\pi^{\oplus k})$ like this:
\begin{multline*}
[\Pi(a)F](\omega) := \sum_h a(\omega,h)\pi^{\oplus k}(h)\big(F(h^{-1}\cdot \omega)\big) \\ (a \in \A_0,\ F \in L^2(\mu;H_\pi^{\oplus k}),\ \omega \in \Omega).
\end{multline*}
We continue to write $\Pi$ for the extension of this representation to $\A$.

Let $\kappa$ be the Koopman representation of $\G$ on $L^2(\mu)$~\cite[Sec. II.10]{KecGAEGA}.  Then the definition above gives $\Pi|\G \simeq \kappa \otimes \pi^{\oplus k}$ (see~\cite[Sec. 13.1]{Dix--Cstar} for tensor products of group representations).  On the other hand, $\Pi|\frD$ is the action of $\frD$ on $L^2(\mu;H_\pi^{\oplus k})$ by pointwise multiplication.

Notice that $(\Pi|\G)^{1_\Omega \otimes H_\pi^{\oplus k}} \simeq \pi^{\oplus k}$.  Since $\pi$ is arbitrary, this shows that our copy of $\G$ inside $\A$ generates an isomorphic copy of the whole of $C^\ast \G$, which we henceforth identify with $C^\ast \G$ itself.  In addition, if we define $\Psi_\G:\A_0 \to \rmM_k[\G]$ by
\begin{equation}\label{eq:CE2}
\Psi_\G(a)(g) := \int a_g(\omega)\ d\mu(\omega) \qquad (a \in \A_0,\ g \in \G),
\end{equation}
then the compression of $\Pi(a)$ to $1_\Omega \otimes H_\pi^{\oplus k} \cong H_\pi^{\oplus k}$ is given by $\pi^{(k)}(\Psi_\G(a))$.  Again since $\pi$ is arbitrary, $\Psi_\G$ extends to a conditional expectation from $\A$ to $\rmM_k(C^\ast \G)$.

Let $\L$ be the representation of $\A$ constructed as above when $\pi := \l$, and let ${\M := \L(\A)''}$ and $\N := \L(\frD)''$.  The algebra $\A$ has a tracial state uniquely determined by the equation
\begin{equation}\label{eq:trace-on-A0}
\tau(a) = \int \tr_k(a_e(\omega))\ d\mu(\omega) = \frac{1}{k}\sum_i \langle \L(a)(1_\Omega\otimes \xi_i),1_\Omega\otimes \xi_i\rangle \qquad (a \in \A_0),
\end{equation}
where $\xi_i$ is as in~\eqref{eq:xi-i}.  The first of these formulas shows that $\tau\circ \Psi_\G = \tau \circ \Psi_\frD = \tau$, and the second shows that $\tau$ is $\L$-normal.  Its extension $\t{\tau}$ is a faithful normal tracial state on $\M$.  Letting $\t{\Psi}$ be Umegaki's $\t{\tau}$-preserving conditional expectation from $\M$ onto $\N$~\cite[Lem. 1.5.11]{BroOza08}, another check against~\eqref{eq:trace-on-A0} shows that $\t{\Psi}\circ \L = \L\circ  \Psi_\frD$. 


\subsection{A subdiagonal subalgebra}

Let $\A_\ge$ be the subset of all $a \in \A_0$ that satisfy
\begin{equation}\label{eq:finite-sums-2}
a(\omega,g) = 0 \qquad \hbox{whenever} \qquad g >_\omega e.
\end{equation}
Let $\A_>$ be the further subset of those $a \in \A_\ge$ which also satisfy $a_e = 0$.  Then $\A_\ge \supset \frD$, but $\A_> \cap \frD = \{0\}$.  By comparing~\eqref{eq:order-equiv} with the involution in~\eqref{eq:twisted}, $\A_\ge^\ast$ consists of those $a \in \A_0$ that satisfy $a(\omega,g) = 0$ whenever $g <_\omega e$.

\begin{lem}\label{lem:subalg}
The set $\A_\ge$ is a unital subalgebra of $\A$.
\end{lem}

\begin{proof}
First, $\A_\ge$ is a linear space by construction, and it contains the unit $1_\frD \otimes \delta_e$.  To show that $\A_\ge$ is an algebra, let $a,b \in \A_\ge$, and consider the formula for $a\ast b$ in~\eqref{eq:twisted}.  Suppose that $g >_\omega e$.  Then there are two possibilities for each summand in that formula:
\begin{itemize}
\item If $h >_\omega e$, then $a(\omega,h) = 0$ is zero and the summand vanishes.
\item If $e \ge_\omega h$, then $g >_\omega h$ by transitivity, and hence $h^{-1}g >_{h^{-1}\cdot \omega} e$ by~\eqref{eq:order-equiv}.  This implies that $b(h^{-1}\cdot \omega,h^{-1}g) = 0$, and again the summand vanishes.
\end{itemize}
Therefore $a\ast b$ also belongs to $\A_\ge$.
\end{proof}

We can now make contact with a subdiagonal subalgebra.  The next lemma generalizes the results of~\cite[Subsec. 3.2]{Arv67} `relative to $\Omega$'.

\begin{lem}\label{lem:subdiag}
The image $\L(\A_\ge)$ is a subdiagonal subalgebra of $\M$ relative to $\t{\Psi}$, and $\L(\A_>) = \L(\A_\ge) \cap \ker \t{\Psi}$. 
\end{lem}

\begin{proof}
The image $\L(\A_\ge)$ is a unital subalgebra of $\M$ by Lemma~\ref{lem:subalg}, so we need only verify the first three of Arveson's axioms from~\cite[Def. 2.1.1]{Arv67}:
\begin{itemize}
\item[i.] Any element $a$ of $\A_0$ may be decomposed according to
\[a(\omega,g) = a(\omega,g)1_{\{g \le_\omega e\}} + a(\omega,g) 1_{\{g >_\omega e\}},\]
and both summands are still elements of $\A_0$.  Therefore $\A_\ge + \A_\ge^\ast$ equals the whole of $\A_0$.  This is norm dense in $\A$ by construction, and $\L(\A)$ is ultraweakly dense in $\M$~\cite[Thm. I.3.2]{Dix--vN}.
\item[ii.] If $a,b \in \A_{\ge}$, then the multiplication from~\eqref{eq:twisted} gives
\[(a\ast b)_e(\omega) = \sum_h a(\omega,h)b(h^{-1}\cdot \omega,h^{-1}).\]
Consider the summands on the right.  If $h >_\omega e$, then $a(\omega,h) = 0$.  On the other hand, if $h <_\omega e$, then~\eqref{eq:order-equiv} gives $e <_{h^{-1}\cdot \omega} h^{-1}$, and hence ${b(h^{-1}\cdot \omega,h^{-1}) = 0}$.  So all terms with $h\ne e$ vanish, and we are left with
\[\t{\Psi}(\L(a)\L(b)) = \L((a\ast b)_e) = \L(a_e)\L(b_e) = \t{\Psi}(\L(a))\t{\Psi}(\L(b)).\]
\item[iii.] We have $\t{\Psi}(\L(\A_\ge)) = \L(\Psi_\frD(\A_\ge)) = \L(\frD) \subset \L(\A_\ge) \cap \L(\A_\ge)^\ast$.
\end{itemize}
The final equality also follows from the fact that $\t{\Psi}\circ \L = \L\circ  \Psi_\frD$.
\end{proof}

\subsection{Wandering vectors and completion of the proof}

We prove Theorem~\ref{mainthm:past} via the more abstract reformulation in Theorem~\ref{thm:subdiag-Szego} below.  This is a version of Schur's determinantal formula~\cite[Subs. 0.8.5]{HorJohMA} for a positive functional on $\A$.  Because we have allowed $\frD$ to consist of $\rmM_k$-valued functions from the outset, we obtain a result for $\rmM_k$-valued completely positive maps as well via the pairing isomorphism.

We begin with the following classical notion.  In a representation $\Pi$ of $\A$, a vector $y$ in $H_\Pi$ is called \textbf{wandering} if $y$ is orthogonal to $\Pi(\A_>)y$ (compare~\cite[Sec. 4.1]{Arv67}, which uses the term `right-wandering').  In this case the following also hold:
\begin{align*}
\langle \Pi(d)y,\Pi(a)y\rangle &= \langle y,\Pi(d^\ast a)y\rangle = 0\qquad (d \in \frD,\ a \in \A_>);\\
\langle \Pi(d)y,\Pi(a^\ast)y\rangle &= \langle \Pi(ad)y,y\rangle = 0\qquad (d \in \frD,\ a \in \A_>);\\
\langle \Pi(a_1^\ast)y,\Pi(a_2)y\rangle &= \langle y,\Pi(a_1a_2)y\rangle = 0 \qquad (a_1,a_2 \in \A_>).
\end{align*}
Therefore the subspaces $\ol{\Pi(\A_>^\ast)y}$, $\ol{\Pi(\frD)y}$ and $\ol{\Pi(\A_>)y}$ are orthogonal.

Any $b \in \A_\ge + \A_\ge^\ast$ may be expressed as $\Psi_\frD(b) + a_1 + a_2^\ast$ with $a_1,a_2 \in \A_>$.  Using this decomposition and the fact that $\A_\ge + \A_\ge^\ast$ is norm dense in $\A$, it follows that a vector $y$ is wandering if and only if $\Phi_y^\Pi = \Phi_y^\Pi \circ \Psi_\frD$.

Now consider an arbitrary vector $x$ in $H_\Pi$.  We define its \textbf{wandering part} $y$ to be the component of $x$ orthogonal to the subspace $\ol{\Pi(\A_>)x}$.  This $y$ lies in $x + \ol{\Pi(\A_>)x}$, and therefore
\[\Pi(\A_>)y \subset \Pi(\A_>)(x + \ol{\Pi(\A_>)x}) \subset \ol{\Pi(\A_>)x}.\]
This subset is orthogonal to $y$ by construction, so $y$ is indeed wandering.  Finally, if $\phi = \Phi^\Pi_x$, then we call $\Phi^\Pi_y$ the \textbf{Schur complement} of $\phi$ relative to $\A_{\ge}$.   This name is unambiguous because of the uniqueness of the GNS representation of $\phi$.  This is an abstraction of the generalized Schur complement of a submatrix in a larger Gram matrix~\cite[Exers. 7.1.P28, 7.3.P8]{HorJohMA}.

\begin{thm}\label{thm:subdiag-Szego}
In the situation above, we have $\Delta_\tau \phi = \Delta_{\tau|\frD}(\psi|\frD)$.
\end{thm}

If $x$ is already wandering, then $\psi = \phi$, and Theorem~\ref{thm:subdiag-Szego} is an instance of Lemma~\ref{lem:det-ucp} for the conditional expectation $\Psi_\frD$.  We use this special case in the course of proving the full theorem.  The other main ingredient for the proof is the Arveson--Labuschagne Jensen inequality for maximal subdiagonal subalgebras from~\cite{Lab05}.  We use that inequality via the following consequence.

\begin{prop}\label{prop:Jensen}
Every $a \in \A_>$ satisfies $\Delta |1 + a| \ge 1$.
\end{prop}

\begin{proof}
By Lemma~\ref{lem:subdiag} and~\cite[Thm. 2.2.1]{Arv67}, $\L(\A_\ge)$ is contained in a unique maximal subdiagonal subalgebra of $\M$.  This provides the correct setting for the Arveson--Labuschagne version of Jensen's formula~\cite[Thm. 3]{Lab05}.  That formula gives us the inequality step here:
\[\Delta |1 + a| = \Delta_{\t{\tau}} |\L(1 + a)| \ge \Delta_{\t{\tau}} |\t{\Psi}(\L(1 + a))| = \Delta_{\t{\tau}}|\L(1 + \Psi_\frD(a))|  = 1.\]
\end{proof}

\begin{cor}\label{cor:shear}
Let $\Pi$ be a representation of $\A$, let $x \in H_\Pi$, and let $w$ lie in $x + \ol{\Pi(\A_>)x}$.  Then $\Delta \Phi^\Pi_w \ge \Delta \Phi^\Pi_x$.
\end{cor}

\begin{proof}
Choose a sequence $(a_n)_{n\ge 1}$ in $\A_>$ so that $w_n := (1 + \Pi(a_n))x \to w$ as $n\to\infty$.  Then $\Phi^\Pi_{w_n} = \Phi^\Pi_x((1+a_n)^\ast(\cdot)(1+a_n))$, and by Lemma~\ref{lem:unif-cts} these converge to $\Phi^\Pi_w$.  On the other hand, Propositions~\ref{prop:det-props}(e) and then~\ref{prop:Jensen} give
\[\Delta \Phi^\Pi_{w_n}= (\Delta|1 + a_n|)^2\cdot \Delta \Phi^\Pi_x  \ge \Delta \Phi^\Pi_x.\]
As $n\to\infty$ the result follows by upper semicontinuity (Proposition~\ref{prop:det-props}(c)).
\end{proof}

\begin{proof}[Proof of Theorem~\ref{thm:subdiag-Szego}]
Abbreviate $\Delta_\tau$ to $\Delta$ for the rest of this proof. Corollary~\ref{cor:shear} applies to the wandering part $y$ of $x$ to give $\Delta \psi \ge \Delta \phi$.

Now let $K_<:= \ol{\Pi(\A^\ast_>)y}$, $K_0 := \ol{\Pi(\frD)y}$ and $K_> := \ol{\Pi(\A_>)y}$.  Since $y$ is wandering, $\ol{\Pi(\A)y}$ is the orthogonal sum of $K_<$, $K_0$ and $K_>$.  If it happens that $x$ lies in $y + K_>$, then we can swap the roles of $x$ and $y$ in Corollary~\ref{cor:shear} to complete the proof.  This condition on $x$ may fail in general, but we can adapt the argument as follows.

Let $z$ be the orthogonal projection of $x$ onto $\ol{\Pi(\A)y}$.  This is the sum of the projections of $x$ onto $K_<$, $K_0$ and $K_>$, because those subspaces are orthogonal.  Now, the definition of $y$ gives
\begin{align*}
\langle \Pi(a^\ast)y,x\rangle &= \langle y,\Pi(a)x\rangle = 0 \qquad (a \in \A_>)\\
\hbox{and} \qquad \langle \Pi(d)y,(x-y)\rangle &= \langle y,\Pi(d^\ast)(x-y)\rangle = 0 \qquad (d \in \frD).
\end{align*}
So $x$ is orthogonal to $K_<$, and its projection onto $K_0$ is equal to $y$ itself.  It follows that $z \in y + K_>$, and so Corollary~\ref{cor:shear} gives $\Delta \Phi^\Pi_z \ge \Delta \psi \ge \Delta \phi$.  On the other hand, $z$ is the projection of $x$ onto a subrepresentation of $\Pi$, so $\Phi^\Pi_z \le \phi$ in the positive definite ordering (as we saw in the construction of the Lebesgue decomposition, for example).  Because of this, Proposition~\ref{prop:det-props}(a) gives $\Delta\Phi^\Pi_z \le \Delta \phi$.  So in fact we must have $\Delta \Phi^\Pi_z = \Delta \psi = \Delta \phi$.

Finally, since $y$ is wandering, Lemma~\ref{lem:det-ucp} gives $\Delta \psi = \Delta_{\tau|\frD}(\psi|\frD)$.
\end{proof}

It remains to deduce Theorem~\ref{mainthm:past} from Theorem~\ref{thm:subdiag-Szego}.  Let $\phi:\G\to \rmM_k$ be positive definite, and write $\phi$ also for its extension to a completely positive map on $C^\ast \G$.  Let $\phi$ be associated to $\pi$ by the cyclic tuple $x_1$, \dots, $x_k \in H_\pi$.

For each $\omega \in \Omega$, define
\begin{equation}\label{eq:local-past}
N_\omega:= \ol{\rm{span}}\{\pi(g)x_i:\ g <_\omega e,\ i=1,\dots,k\} \qquad (\omega \in \Omega).
\end{equation}
These are the measurably-varying subspaces that appear in the statement of Theorem~\ref{mainthm:past}.  They form a measurable field of closed subspace of $H_\pi$ (see~\cite[Chap. II.1]{Dix--vN}).  Let $R_\omega$ be the orthogonal projection from $H_\pi$ to $N_\omega$.

\begin{proof}[Proof of Theorem~\ref{mainthm:past}]
Let us start by re-writing~\eqref{eq:past} like this:
\begin{equation}\label{eq:past2}
\exp\frac{1}{k}\int \log \det [\langle R_\omega^\perp x_j,R_\omega^\perp x_i\rangle]\ d\mu(\omega) = (\Delta_\chi \phi)^{1/k}
\end{equation}
By~\eqref{eq:det-phi-det-pairing}, the right-hand side of~\eqref{eq:past2} is equal to $\Delta_{\chi\otimes \tr_k}(\langle \phi,\cdot\rangle)$, and now by~\eqref{eq:trace-on-A0} and the second part of Lemma~\ref{lem:det-ucp} this is equal to $\Delta_\tau \phi'$, where the positive functional $\phi' := \langle \phi,\cdot \rangle\circ \Psi_\G$ on $\A$ is associated to $\Pi$ by the vector
\[x := 1_\Omega \otimes k^{-1/2}[x_1,\dots,x_k]^\rm{T}.\]

Next, by Corollary~\ref{thm:subdiag-Szego}, $\Delta_\tau \phi'$ is equal to $\Delta_{\tau|\frD}(\psi|\frD)$, where $\psi$ is the Schur complement of $\phi'$.  It remains to evaluate $\psi|\frD$ and show that $\Delta_{\tau|\frD}(\psi|\frD)$ is given by the left-hand side of~\eqref{eq:past2}.

The wandering part $y = [y_1,\dots,y_k]^\rm{T}$ is the component of $x$ orthogonal to the closure of the subspace $\Pi(\A_>)x$. This subspace consists of the functions in $L^2(\mu;H^{\oplus k}_\pi)$ that have the form
\[\sum_ga(\omega,g)[\pi(g)x_1,\dots,\pi(g)x_k]^\rm{T} \qquad (\omega \in \Omega)\]
for some $a \in \A_>$.  Because of the property~\eqref{eq:finite-sums-2}, the sum above is a vector in $N^{\oplus k}_\omega$.  On the other hand, we can regard $a$ as a $k$-by-$k$ matrix of complex-valued functions on $\Omega\times \G$, and choose each of those functions independently subject to~\eqref{eq:finite-sums-2}.   Therefore a measurable selection argument followed by an approximation by continuous functions shows that
\[\ol{\Pi(\A_>)x} = \int^\oplus_\Omega N_\omega^{\oplus k}\ d\mu(\omega),\]
understanding this direct integral as a closed subspace of $L^2(\mu;H_\pi^{\oplus k})$ (see~\cite[Chap. II.1]{Dix--vN}). Consequently, $y_i(\omega) = k^{-1/2}R_\omega^\perp x_i$ as elements of $L^2(\mu;H_\pi)$, and so
\[\psi(d) = \langle \Pi(d)y,y\rangle = \frac{1}{k}\sum_{i,j}\int d_{ij}(\omega)\langle R_\omega^\perp x_j,R_\omega^\perp x_i\rangle\ d\mu(\omega) \qquad (d \in \frD).\]
Finally, the analysis from Example~\ref{ex:mat-val-fn} shows that $\Delta_{\tau|\frD}(\psi|\frD)$ is equal to the left-hand side of~\eqref{eq:past2}.
\end{proof}

\subsection{Further remarks}\label{subs:B-rmks}

Non-self-adjoint subalgebras of C*-algebras similar to our $\A_\ge$ have been studied by several authors.  In particular, Kawamura and Tomiyama introduced C*-subdiagonal subalgebras in~\cite[Sec. 3]{KawTom77}.  Our example $\A_\ge$ is a C*-subdiagonal subalgebra of $\A$ when $\G$ is amenable.  When $\G$ is non-amenable, this fails because the conditional expectation $\Psi_\frD$ is not faithful, and so the full and reduced C*-crossed products differ (compare~\cite[Cor. 5.6.17]{BroOza08}).   We do not need the faithfulness of $\Psi_\frD$ here, so have introduced our examples from scratch. Our examples also belong to a more general family of subalgebras of groupoid C*-algebras for totally ordered groupoids: see~\cite{MuhSol89}, which also includes background and further references on groupoid C*-algebras.  Indeed, besides technicalities around faithfulness, our Lemma~\ref{lem:subdiag} is mostly the same as the forward direction in~\cite[Thm. 4.2]{MuhSol89}.  It could be interesting to generalize our results to the broader setting suggested by that paper.

Fuglede--Kadison determinants of matrices over group von Neumann algebras appear in connections with various other parts of mathematics.  For example, connections to $L^2$-invariants in topology are described in~\cite[Sec. 3.2 and Chap. 13]{Luc02}.  See that reference and also~\cite{KirKreLuc25} for an overview of open problems in this direction, such as L\"{u}ck's approximation and determinant conjectures.  The latter is generalized to the `measure-theoretic determinant conjecture' in~\cite{LucSauWeg10}, and shown to have consequences for uniform measure equivalence of discrete groups. I expect that Theorem~\ref{mainthm:past} generalizes in a similar way, for instance by using a version of Example~\ref{ex:Bern} over a unimodular random network.

\begin{prob}
Does the construction of a subdiagonal subalgebra from an invariant random order shed any light on these open questions about $L^2$-invariants?
\end{prob}

Hayes' paper~\cite{Hayes21b} studies certain probability-preserving systems of algebraic origin for a finitely generated group $\G$ with a left invariant total order.  They are constructed from elements of $\bbZ[\G]$ that he calls `lopsided', which are defined in terms of the order.  His main result is that a nondegenerate lopsided group-ring element always gives a factor of a Bernoulli shift.

\begin{prob}
Can coupling to an invariant random order generalize Hayes' construction and result to other finitely generated groups?
\end{prob}

\section{Approximate association and almost periodic sequences}\label{sec:approx-assoc}

This section makes more preparations for our introduction of `almost periodic entropy', the new notion that appears in Theorem~\ref{mainthm:det-form}.

\subsection{Typical vectors and approximate association}\label{subs:typ-vec}

Consider again a general separable, unital C*-algebra $\A$. We can classify tuples of vectors in a representation $\pi$ of $\A$ according to their type.

\begin{dfn}\label{dfn:typical}
For any positive integer $k$ and subset $O$ of $\frL(\A,\rmM_k)$, let
\[ \X(\pi,O) := \big\{[v_1,\dots,v_k]^\rm{T} \in H_\pi^{\oplus k} :\ \Phi^\pi_{v_1,\dots,v_k} \in O \big\}. \]
The elements of $\X(\pi,O)$ are the \textbf{$O$-typical} tuples of the representation $\pi$.  In addition, let
\[\S_k(\pi) := \big\{\Phi^\pi_{v_1,\dots,v_k}:\ v_1,\dots,v_k \in H_\pi\ \hbox{and}\ \|v_1\|^2 + \cdots + \|v_k\|^2 = k\big\}.\]
This is the subset of elements of $\S_k(\A)$ that are associated to $\pi$.
\end{dfn}

The transpose in the definition of $\X(\pi,O)$ is not conceptually significant, but it simplifies some manipulations later.

If $O \subset \S_k(\A)$, observe that
\begin{equation}\label{eq:meet-and-nonempty}
O\cap \S_k(\pi) \ne \emptyset \quad \Leftrightarrow \quad \X(\pi,O) \ne \emptyset
\end{equation}
for any representation $\pi$.

We often use Definition~\ref{dfn:typical} when $O$ is a small neighbourhood of a given `target' completely positive map $\phi$.  In this case we may informally describe elements of $\X(\pi,O)$ as `approximately $\phi$-typical tuples'.  This resembles the use of terms such as `microstate' in free probability or `good model' in the study of sofic entropy in ergodic theory: compare~\cite[Sec. 2.3]{Bowen--survey}, for example.

Let $\phi \in \frL(\A,\rmM_k)_+$, let $O$ be a neighbourhood of $\phi$, and let $\t{O}$ be the corresponding neighbourhood of $\langle \phi,\cdot\rangle$ under the pairing isomorphism.  Then Lemma~\ref{lem:dilation-matrix-dilation} tells us that
\begin{equation}\label{eq:X-and-X}
	\X(\pi^{(k)},\t{O}) = k^{-1/2}\X(\pi,O).
\end{equation}

If $\pi$ is a representation and $\phi \in \frL(\A,\rmM_k)_+$, then $\phi$ is \textbf{approximately associated} to $\pi$ if it lies in the closure of the set of completely positive maps associated to $\pi$. Because of the second countability from Lemma~\ref{lem:lcsc}, this holds if and only if there is a sequence of $k$-tuples $[v_{n,1}, \dots, v_{n,k}]$ in $H_\pi$ such that
\begin{equation}\label{eq:approx-assoc-conv}
\Phi^\pi_{v_{n,1},\dots,v_{n,k}}(a) \to \phi(a) \qquad \hbox{for every}\ a\in \A\ \hbox{as}\ n\to\infty.
\end{equation}

If $\phi$ is approximately associated to $\pi$ and also normalized, and if the tuples $v_{n,1}$, \dots, $v_{n,k}$ witness the convergence in~\eqref{eq:approx-assoc-conv}, then we can normalize those tuples slightly to show that $\phi$ actually lies in the weak$^\ast$ closure $\ol{\S_k(\pi)}$. Let us call $\ol{\S_k(\pi)}$ the \textbf{$k$-summary} of $\pi$.

The literature makes more use of a slightly coarser notion: $\phi$ is \textbf{weakly associated} to $\pi$ if $\phi$ can be approximated by \emph{convex combinations} of maps associated to $\pi$.  When $k=1$, this is the basis for Godemont and Fell's relation of weak containment of representations~\cite[Sec. 3.4]{Dix--Cstar}.  One can define a similar but finer relation on representations as the inclusion of $k$-summaries for all $k$, without allowing convex combinations.  Some references call this `weak containment in the sense of Zimmer' because of its appearance in~\cite[Def. 7.3.5]{Zim84}.  However, in that reference Zimmer himself attributes the idea to Fell.  While I have not found this precise definition in Fell's papers, it is suggested rather naturally by his study of the quotient topology in~\cite{Fel60b,Fel62}.  Moreover, by results of Voiculescu~\cite{Voi76} (see also~\cite{Arv77}),  two separable representations have equal $k$-summaries for all $k$ if and only if they are approximately unitarily equivalent.  I have chosen the term `approximate association' to reflect this last connection. 

Approximate association can sometimes detect multiplicities in the GNS representation of a positive functional, whereas weak association cannot.

\begin{ex}\label{ex:sigma-not-convex}
Let $\pi$ be an irreducible representation whose dimension is finite but at least $2$.  Let $x$ and $y$ be linearly independent unit vectors in $H_\pi$, and let $\phi:= \Phi^\pi_x$ and $\psi := \Phi^\pi_y$.  These are both associated to $\pi$, and they are linearly independent by the uniqueness of the GNS construction and the irreducibility of $\pi$.  However, $(\phi + \psi)/2$ cannot be associated to $\pi$, because it is not pure, and so its GNS representation must be the whole of $\pi^{\oplus 2}$.  Moreover, these facts persist if we require only approximate association, because the finite dimensionality of $\pi$ implies that $\S_1(\pi)$ and $\S_1(\pi^{\oplus 2})$ are already compact sets, without taking closures.  So $\S_1(\pi)$ is not convex in this example.  On the other hand, $(\phi + \psi)/2$ is certainly weakly associated to $\pi$. \fin
\end{ex}

Our notion of almost periodic entropy (Definition~\ref{dfn:APent} below) involves approximations to a positive functional, but it does not allow for taking convex combinations.  For this reason, approximate association plays a larger role than weak association in the present paper.

We next compare typical tuples for two completely positive maps if they are related by association or approximate association.  First, fix an $\ell$-by-$k$ matrix $[a_{ij}]$ of elements of $\A$.  If $\pi$ is a representation and ${v_1, \dots, v_k} \in H_\pi$, recall that we can define a new $\ell$-tuple ${y_1, \dots, y_\ell}$ in $H_\pi$ using $[a_{ij}]$ and ${v_1, \dots, v_k}$ as in formula~\eqref{eq:matrix-vector}.  The resulting type of ${y_1, \dots, y_\ell}$ is then related to the type of ${v_1, \dots, v_k}$ by Lemma~\ref{lem:new-type}.  Because of the weak$^\ast$ continuity given by that lemma, we obtain the following.

\begin{lem}\label{lem:lin-maps}
	Let $\psi$ be obtained from $\phi$ and $[a_{ij}]$ as in Lemma~\ref{lem:new-type}. For any neighbourhood $O$ of $\psi$, there is a neighbourhood $U$ of $\phi$ such that
\[\big\{[\pi(a_{ij})]\cdot[v_1,\dots,v_k]^\rm{T}:\ [v_1,\dots,v_k]^\rm{T} \in \X(\pi,U)\big\} \subset \X(\pi,O)\]
for any representation $\pi$.

In particular, suppose that
\[\psi(b) = (Q^\rm{T})^\ast \phi(b) Q^\rm{T} \qquad (b \in \A)\]
for some $Q \in \rmM_{\ell,k}$, as in~\eqref{eq:psi-Q-phi}.  Then, for any neighbourhood $O$ of $\psi$, there is a neighbourhood $U$ of $\phi$ such that
\[(I_{H_\pi}\otimes Q)[\X(\pi,U)] \subset \X(\pi,O)\]
for any representation $\pi$ (identifying $H_\pi^{\oplus k}$ with $H_\pi\otimes \bbC^{\oplus k}$ as in~\eqref{eq:psi-Q-phi}). \qed
\end{lem}

\begin{cor}\label{cor:typ-trans}
	Let $\phi \in \frL(\A,\rmM_k)_+$ and $\psi \in \frL(\A,\rmM_\ell)_+$, and assume that $\psi$ is approximately associated to $\pi_\phi$.  Then for any neighbourhood $O$ of $\psi$ there is a neighbourhood $U$ of $\phi$ such that
	\[\X(\pi,U)\ne \emptyset \quad \Rightarrow \quad \X(\pi,O) \ne \emptyset.\]
	for any representation $\pi$.
\end{cor}

\begin{proof}
	Let $\phi$ be associated to $\pi_\phi$ by the cyclic tuple $x_1$, \dots, $x_k$.  By cyclicity and Lemma~\ref{lem:unif-cts}, there is some $\ell$-by-$k$ matrix $[a_{ij}]$ of elements of $\A$ such that the tuple defined by
	\[[y_1, \dots, y_\ell]^\rm{T} := [\pi(a_{ij})]\cdot [x_1,\dots,x_k]^\rm{T}\]
	satisfies $\psi' := \Phi^\pi_{y_1,\dots,y_\ell} \in O$.  Now apply Lemma~\ref{lem:lin-maps} to $\phi$, $\psi'$ and $[a_{ij}]$.
\end{proof}

\subsection{Pairs and sums of typical tuples}

Let $k$ and $\ell$ be positive integers, and let
\[K \:= \{1,\dots,k\} \qquad \hbox{and} \qquad L := \{k+1,\dots,k+\ell\}.\]
Let $\phi \in \frL(\A,\rmM_k)_+$ and $\psi \in \frL(\A,\rmM_\ell)_+$. The next lemma is a robust form of Lemma~\ref{lem:disjoint-no-join-0}.

\begin{lem}\label{lem:disjoint-no-join}
If $\phi$ and $\psi$ are disjoint, then for every neighbourhood $O$ of $\rm{diag}(\phi,\psi)$ there are neighbourhoods $U$ of $\phi
$ and $V$ of $\psi$ such that the following holds:
\begin{quote}
If $\theta \in \frL(\A,\rmM_{k+\ell})_+$ satisfies $\theta[K] \in U$ and $\theta[L] \in V$, then $\theta \in O$.
\end{quote}
These neighbourhoods satisfy
\[\X(\pi,O)\supset \X(\pi,U) \times \X(\pi,V)\]
for any representation $\pi$ (regarding the right-hand side as a subset of $H_\pi^{\oplus(k+\ell)}$).
\end{lem}

\begin{proof}
We prove the first conclusion by contraposition. Assume that $\rm{diag}(\phi,\psi)$ has a neighbourhood $O$ for which no pair of neighbourhoods $U$ and $V$ gives the desired implication.  Then, by the second countability from Lemma~\ref{lem:lcsc}, there is a sequence $(\theta_n)_{n\ge 1}$ in $\frL(\A,\rmM_k)_+\setminus O$ such that
\[\theta_n[K]\to \phi \qquad \hbox{and} \qquad \theta_n[L] \to \psi.\]
This sequence $(\theta_n)_{n\ge 1}$ must be uniformly bounded in the dual norm because of the identity~\eqref{eq:norm-cts} and the relation
\[\tr_{k+\ell}\theta_n(1_\A) = \frac{k}{k+\ell}\tr_k\big(\theta_n(1_\A)[K]\big) + \frac{\ell}{k+\ell}\tr_\ell\big(\theta_n(1_\A)[L]\big).\]
Therefore, by the Banach--Alaoglu theorem, $(\theta_n)_{n\ge 1}$ has a subsequential limit in the weak$^\ast$ topology.  This limit must be a joining of $\phi$ and $\psi$, but also it cannot lie in $O$ and so it must be different from $\rm{diag}(\phi,\psi)$.  Therefore $\phi$ and $\psi$ are not disjoint.

The second conclusion follows from the first one and the definition of type.
\end{proof}

The assumption of disjointness in Lemma~\ref{lem:disjoint-no-join} is not superfluous.  Indeed, if $\pi_\phi$ is finite-dimensional and irreducible, then $\X(\pi,U)$ is nonempty for any neighourhood $U$ of $\phi$, but $\rm{diag}(\phi,\phi)$ is not necessarily approximately associated to $\pi$, only to $\pi^{\oplus 2}$, by reasoning similar to Example~\ref{ex:sigma-not-convex}.

Now we consider typical vectors for the sum $\g := \phi + \psi$.  Here we restrict our attention to the case $k = \ell = 1$.  The case $k=\ell > 1$ would not involve any new ideas, but would require heavier notation.  Later in the paper we use the pairing isomorphism to avoid needing that case.

\begin{lem}\label{lem:split-into-disjoint}
Assume that $k=\ell=1$ and that $\phi$ and $\psi$ are disjoint.  For any neighbourhoods $U$ of $\phi$ and $V$ of $\psi$ there are a neighbourhood $W$ of $\g$ and an element $a$ of $\A$ such that
\[\{(\pi(a)x,\pi(1-a)x):\ x \in \X(\pi,W)\} \subset  \X(\pi,U) \times \X(\pi,V)\]
for any representation $\pi$.
\end{lem}

\begin{proof}
The representations $\pi_\phi$ and $\pi_\psi$ are disjoint by assumption, and they are both contained in $\pi_\g$ (see~\cite[Prop. 2.5.1]{Dix--Cstar}, for example).  Therefore, by the uniqueness of GNS representations, we may identify $\pi_\g$ with ${\pi_\phi\oplus \pi_\psi}$.  The resulting orthogonal projection $P$ from $H_\g$ to $H_\phi$ lies in the centre of $\pi_\g(\A)''$ by Lemma~\ref{lem:central-projection}.

Now suppose that $\g$ is associated to $\pi_\g$ by the cyclic vector $v$.  Then $\phi$ and $\psi$ are associated to $\pi_\g$ by $Pv$ and $v-Pv$, respectively. The Kaplansky density theorem~\cite[Sec. I.3.5]{Dix--vN} applied to $\pi_\g(\A)$ gives an element $a \in \A$ such that $0\le a \le 1$ and such that $\pi_\g(a)v$ lies as close as we wish to $Pv$.  In particular, we may choose $a$ so that
\[\Phi^{\pi_\g}_{\pi_\g(a)v} = \g(a^\ast(\cdot)a) \in U \qquad \hbox{and} \qquad \Phi^{\pi_\g}_{\pi_\g(1-a)v} = \g((1-a)^\ast(\cdot)(1-a)) \in V.\]
Now two applications of Lemma~\ref{lem:lin-maps} produce the required neighbourhood $W$.
\end{proof}

\begin{cor}\label{cor:sums}
Assume that $k=\ell=1$ and that $\phi$ and $\psi$ are disjoint.
\begin{enumerate}
\item[a.] For every neighbourhood $W$ of $\g$ there are neighbourhoods $U$ of $\phi$ and $V$ of $\psi$ such that
\[\X(\pi,W)\supset \X(\pi,U) + \X(\pi,V)\]
for any representation $\pi$.
\item[b.] For any neighbourhoods $U$ of $\phi$ and $V$ of $\psi$ there is a neighbourhood $W$ of $\g$ such that
\[\X(\pi,W) \subset \X(\pi,U) + \X(\pi,V)\]
for any representation $\pi$.
\end{enumerate}
\end{cor}

\begin{proof}
First let $W$ be a neighbourhood of $\g$. Given a representation $\pi$ and vectors $x$, $y \in H_\pi$, observe that
\[\Phi^\pi_{x+y}(a) = \Phi^\pi_x(a) + \Phi^\pi_y(a) + \langle \pi(a)x,y\rangle + \langle \pi(a)y,x\rangle \qquad (a \in \A).\]
We may therefore choose a neighbourhood $O$ of $\rm{diag}(\phi,\psi)$ such that
\[\{x+y:\ [x,y]^\rm{T} \in \X(\pi,O)\} \subset \X(\pi,W)\]
for any representation $\pi$.  Part (a) follows by concatenating this with the last inclusion from Lemma~\ref{lem:disjoint-no-join}.

On the other hand, given the neighbourhoods $U$ of $\phi$ and $V$ of $\psi$, choose $W$ and $a$ as in Lemma~\ref{lem:split-into-disjoint}.  Since
\[x = \pi(a)x + \pi(1-a)x\]
for any $\pi$ and $x \in H_\pi$, the conclusion of Lemma~\ref{lem:split-into-disjoint} represents any element of $\X(\pi,W)$ as an element of $\X(\pi,U) + \X(\pi,V)$.
\end{proof}

\subsection{Strong-quotient convergence of representations}\label{subs:strong-quotient}

For each $k$, the space $\S_k(\A)$ is compact by the Banach--Alaoglu theorem, and metrizable because $\A$ is separable.  Let $\cal{K}(\S_k(\A))$ be its hyperspace (the set of all its compact subsets), and endow this with the Vietoris topology~\cite[Subsec. I.4.F]{Kec95}.  This topology is again compact and metrizable, for instance using a Hausdorff metric constructed from a suitable metric on $\S_k(\A)$~\cite[Thm. 4.26]{Kec95}.  We can therefore describe it in terms of sequences.  Fix $k$, and let $(K_n)_{n\ge 1}$ be a sequence in $\cal{K}(\S_k(\A))$. Its \textbf{topological upper} and \textbf{lower limits} are the sets
\begin{multline*}
\rm{T}\limsup_n K_n := \{\phi \in \S_k(\A):\ \hbox{every neighbourhood of $\phi$}\\ \hbox{meets $K_n$ for infinitely many $n$}\}
\end{multline*}
and
\begin{multline*}
\rm{T}\liminf_n K_n := \{\phi \in \S_k(\A):\ \hbox{every neighbourhood of $\phi$}\\ \hbox{meets $K_n$ for all sufficiently large $n$}\},
\end{multline*}
respectively.  Then $(K_n)_{n\ge 1}$ Vietoris converges if and only if its topological upper and lower limits are equal, and in this case that common set is $\lim_n K_n$.

If $(\pi_n)_{n\ge 1}$ is a sequence of representations, then we say that it \textbf{strong-quotient converges} if the $k$-summaries $\ol{\S_k(\pi_n)}$ Vietoris converge to a limit in $\cal{K}(\S_k(\A))$ as $n\to\infty$ for every $k$.  More abstractly, for any family $R$ of either separable representations or equivalence classes of such representations, we define the \textbf{strong-quotient topology} on $R$ to be the topology generated by the sequence of maps $\pi \mapsto \ol{\S_k(\pi})$ into the spaces $\cal{K}(\S_k(\A))$.

As far as I know, this topology was first studied by Ab\'ert and Elek in~\cite{AbeEle11}, but without this terminology.  As inspiration they cite the notion of `local-global convergence' for sequencess of sparse finite graphs~\cite{BolRio11,HatamiLovSze14}.  Ab\'ert and Elek focused on the analogous mode of convergence for measure-preserving group actions, but their paper indicates the story for unitary representations as well.  I use the term `strong-quotient' because convergence in this topology strengthens both `strong convergence' of representations, as surveyed in~\cite{Magee--survey}, and also convergence in Fell's `quotient topology'~\cite[Sec. 3]{Fel60b}.

By Voiculescu's results from~\cite{Voi76}, the strong-quotient topology is actually pulled back from a topology on approximate unitary equivalence classes of separable representations, and these are parametrized by their images in $\prod_{k\ge 1}\cal{K}(\S_k(\A))$.  According to the C*-variant of~\cite[Thm. 1]{AbeEle11}, the subset of all such images of representations is closed, and hence compact.  We do not depend on this fact in the sequel, but if desired it provides an interpretation of sequential strong-quotient limits as representations themselves, not just elements of $\prod_{k\ge 1}\cal{K}(\S_k(\A))$.

The next definition is a variant of approximate association.

\begin{dfn}\label{dfn:asymp-assoc}
A map $\phi \in \frL(\A,\rmM_k)_+$ is \textbf{asymptotically associated} to a sequence $(\pi_n)_{n\ge 1}$ of separable representations if $\phi$ lies in $\rm{T}\limsup_n \ol{\S_k(\pi_n)}$.
\end{dfn}

The choice of $\rm{T}\limsup$ rather $\rm{T}\liminf$ here is a matter of convention.  If $(\pi_n)_{n\ge 1}$ strong-quotient converges then this choice makes no difference.

\begin{lem}\label{lem:strong-quot-for-pairing}
Let $(\pi_n)_{n\ge 1}$ be a sequence of separable representations, and let $\phi \in \frL(\A,\rmM_k)_+$.  Then $\phi$ is asymptotically associated to $(\pi_n)_{n\ge 1}$ if and only if $\langle \phi,\cdot\rangle$ is asymptotically associated to $(\pi_n^{(k)})_{n\ge 1}$. 
\end{lem}

\begin{proof}
For each neighbourhood $O$ of $\phi$, let $\t{O}$ be the corresponding neighbhourhood of $\langle \phi,\cdot\rangle$ under the pairing isomorphism.  As $O$ runs over all neighbourhoods of $\phi$, likewise $\t{O}$ runs over all neighbourhoods of $\langle \phi,\cdot\rangle$.  The result follows from this and the relations~\eqref{eq:meet-and-nonempty} and~\eqref{eq:X-and-X}.
\end{proof}

\begin{lem}\label{lem:diag-and-sum-still-assoc}
Let $\phi \in \frL(\A,\rmM_k)_+$ and $\psi \in \frL(\A,\rmM_\ell)_+$. Assume that $\bspi$ strong-quotient converges.  If $\phi$ and $\psi$ are disjoint and they are both asymptotically associated to $\bspi$, then so is $\rm{diag}(\phi,\psi)$, and if $k=\ell$ then so is $\phi+\psi$.
\end{lem}

\begin{proof}
Let $O$ be any neighbourhood of $\rm{diag}(\phi,\psi)$.  Lemma~\ref{lem:disjoint-no-join} gives neighbourhoods $U$ of $\phi$ and $V$ of $\psi$ such that
\[\X(\pi_n,O) \supset \X(\pi_n,U) \times \X(\pi_n,V)\]
for all $n$.  By asymptotic association and strong-quotient convergence, both factors in this Cartesian product are nonempty for all sufficiently large $n$, and so $\X(\pi_n,O)$ is also nonempty for these $n$.

If $k=\ell = 1$, then the sum $\phi+\psi$ is handled in the same way, using Corollary~\ref{cor:sums}(a) in place of Lemma~\ref{lem:disjoint-no-join}. Finally, the case when $k=\ell > 1$ follows by Lemma~\ref{lem:strong-quot-for-pairing}.
\end{proof}

\subsection{Almost periodic sequences}\label{subs:AP}

In the next section, the almost periodic entropy of a completely positive map $\phi$ is defined in terms of the volumes of the $O$-typical sets in certain representations as $O$ ranges over neighbourhoods of $\phi$.  The representations that we use belong to the following sequences.

\begin{dfn}\label{dfn:AP}
An \textbf{almost periodic} (`\textbf{AP}') \textbf{sequence} for $\A$ is a sequence of finite-dimensional representations of $\A$ whose dimensions diverge to $\infty$.
\end{dfn}

Fix an AP sequence $\bspi = (\pi_n)_{n\ge 1}$, and let $d_n$ be the dimension of $\pi_n$ for each $n$.  Later we sometimes consider strong-quotient convergence for such a sequence.  In this case, since each $d_n$ is finite, the sets $\S_k(\pi_n)$ are continuous images of finite-dimensional spheres by Lemma~\ref{lem:unif-cts}, and hence already closed.

For an AP sequence $\bspi$, another possible mode of convergence is convergence of the tracial states $\tr_{d_n}\circ \pi_n$ to some limit tracial state $\tau$ in the weak$^\ast$ topology of $\A_+^\ast$.  In the terminology of free probability theory, this asserts that, for any finite subset $F$ of $\A$, the tuples $(\pi_n(a):\ a \in F)$ form a sequence of `microstates' for $F$ according to the `non-commutative probability space' $(\A,\tau)$ (see~\cite{Voi02--survey}, for example).  This convergence of tracial states is one of the hypotheses of Theorem~\ref{mainthm:det-form}.  It determines which limiting tracial state should be used to define the Fuglede--Kadison determinant in the conclusion of that theorem.

The next lemma is a companion to Lemma~\ref{lem:strong-quot-for-pairing}.

\begin{lem}\label{lem:trace-conv-for-pairing}
Let $(\pi_n)_{n\ge 1}$ be an AP sequence for $\A$, and consider the AP sequence $(\pi^{(k)}_n)_{n\ge 1}$ for $\rmM_k(\A)$.  If $\tr_{d_n}\circ \pi_n\to \tau$, then $\tr_{kd_n}\circ \pi_n^{(k)}\to \tau \otimes \tr_k$.
\end{lem}

\begin{proof}
After identifying $\tr_{kd_n}$ with $\tr_{d_n}\otimes \tr_k$, this follows from the continuity of the operation $(\cdot)\otimes \tr_k$ on tracial states.  That continuity can be checked entry-wise against elements of $\rmM_k(\A)$.
\end{proof}

In general, neither of strong-quotient convergence nor convergence of tracial states implies the other.  But convergence of tracial states does give a lower bound on strong-quotient limits: see Corollary~\ref{cor:char-and-weak-global} below.

\section{Almost periodic entropy}\label{sec:AP}

In this section we define the almost periodic entropy of an $\rmM_k$-valued completely positive map on a separable, unital C*-algebra $\A$.  We then build up its properties towards the proof of Theorem~\ref{mainthm:det-form}, and follow that with a few consequences.

\subsection{Preliminary results from high-dimensional probability}

If $g:\bbN \to (0,\infty)$, then we write $o(g(n))$ as a placeholder for any function $f:\bbN\to \bbR$ that satisfies $f(n)/g(n) \to 0$.

For the rest of this section, let
\begin{equation}\label{eq:ball-vol}
v(d):=\frac{\pi^d}{d!} \qquad (d=1,2,\dots).
\end{equation}
Then $v(d)$ is equal to the volume of the unit ball in $\bbC^d$~\cite[Subsec. 1.4.9]{RudinFTUB}. By Stirling's approximation, this function satisfies
\begin{equation}\label{eq:ball-vol-approx}
k^{kd}\cdot v(kd) = e^{o(d)}\cdot v(d)^k
\end{equation}
for any fixed $k$ as $d\to\infty$.

We write $\vol_d$ for Lebesgue measure on $\bbR^d$.  For any positive integers $d$ and $k$, we also write $\vol_{2kd}$ for the measure on $\rmM_{k,d}$ obtained by identifying this space with $\bbR^{2kd}$.  We write $\rm{S}^{d-1}$ for the unit sphere in $\bbR^d$ or in any other vector space that has a standing identification with $\bbR^d$. We write $\s_{d-1}$ for the surface-area measure on $\rmS^{d-1}$ normalized to have total mass $1$, and we refer to an integral with respect to $\s_{d-1}$ as a `spherical average'.  A key feature of these measures in high dimensions is the phenomenon of measure concentration.  We need the following special case of this.

\begin{lem}\label{lem:conc}
There is an absolute positive constant $c$ such that, for any positive integer $d$ and linear transformation $A$ of $\bbC^d$, we have
\[\s_{2d-1}\{v:\ | \langle Av,v\rangle - \tr_d A | \ge t\} \le 4e^{-ct^2d/\|A\|^2} \qquad (t > 0).\]
\end{lem}

\begin{proof}
If $f(v) = \langle Av,v\rangle$ for $v \in \rmS^{2d-1}$, then $f$ is $2\|A\|$-Lipschitz, and its spherical average is $\tr_d A$ by the invariance of trace under conjugation.  Therefore the desired inequality follows from the concentration of general Lipschitz functions on high-dimensional spheres.  See~\cite{Led92} for a proof of this by interpolation along the heat semigroup, for example.
\end{proof}

For any positive integers $d$ and $k$ and any $O \subset \rmM_{k+}$, let
\begin{equation}\label{eq:TnO}
	T(d,O) := \{X \in \rmM_{d,k}:\ X^\ast X \in O\}.
\end{equation}
Put another way, this is the set of $k$-tuples in $\bbC^{\oplus d}$ whose Gram matrices lie in $O$.

\begin{cor}\label{cor:conc}
Let $k$ be a positive integer and let $O$ be any neighbourhood of $I_k$ in $\rmM_{k+}$.  Then there are positive constants $C$ and $c$ (depending on $k$ and $O$) such that
\[\s_{2dk-1}\big(k^{-1/2}T(d,O)\big) > 1 - Ce^{-cd} \qquad (d=1,2,\dots).\]
\end{cor}

\begin{proof}
By shrinking $O$ if necessary, we may assume it is equal to
\[\bigcap_{i=1}^k\{Q \in\rmM_{k+}:\ |q_{ii} - 1| < \eps\}\cap \bigcap_{1\le i < j \le k}\{Q \in \rmM_{k+}:\ |q_{ij}| < \eps\}\]
for some $\eps > 0$. In this case, we have
\begin{multline*}
k^{-1/2}T(n,O) = \bigcap_{i=1}^k\{[x_1,\dots,x_k]\in \rmM_{d,k}:\ |\langle x_i,x_i\rangle - 1/k| < \eps/k\} \\ \cap \bigcap_{1 \le i < j \le k} \{[x_1,\dots,x_k]\in \rmM_{d,k}:\ |\langle x_i,x_j\rangle| < \eps/k\}.
\end{multline*}
This is an intersection of at most $k^2$ sets, and for each of them its complement has $\s_{2kd-1}$-measure controlled by Lemma~\ref{lem:conc}.
\end{proof}

Lemma~\ref{lem:conc0} below is a generalization of Corollary~\ref{cor:conc} that depends on Lemma~\ref{lem:conc} in the same way.

Subsection~\ref{subs:analogies} discusses the analogy between the Gram matrix of a tuple of vectors and the joint distribution of a tuple of finite-valued random variables.  The next theorem adds another layer to this analogy: a `method of types' interpretation for $\log \det Q$ when $Q$ is a positive semi-definite matrix.  In information theory, the method of types gives a basic combinatorial interpretation of discrete Shannon entropy~\cite[Sec. 11.1]{CovTho06}, and a similar interpretation of the differential entropy of a jointly Gaussian random vector in terms of volumes~\cite[Sec. 8.2]{CovTho06}.

\begin{thm}\label{thm:types-1}
	Let $Q$ be a $k$-by-$k$ positive semi-definite matrix.
	\begin{itemize}
		\item[a.] If $O$ is any neighbourhood of $Q$ in $\rmM_{k+}$, then
		\[\frac{\vol_{2kd}T(d,O)}{v(d)^k} \ge  (\det Q)^{d - o(d)}.\]
		\item[b.] For any $a > \det Q$ there is a neighbourhood $O$ of $Q$ in $\rmM_{k+}$ such that
		\[\frac{\vol_{2kd}T(d,O)}{v(d)^k} \le a^{d + o(d)}.\]
	\end{itemize}
\end{thm}

Variants of Theorem~\ref{thm:types-1} are widely known, but I have not found a convenient reference for this particular one, so I include its proof.

\begin{proof}
We write a typical element of $\rmM_{d,k}$ as $X = [x_1,\dots,x_k]$, and write
\[\|X\|_2^2 := \sum_{i=1}^k \|x_i\|^2.\]

\vspace{7pt}

\emph{Step 1.}\quad We first prove part (a) for $Q = I_k$.  By shrinking the neighbourhood $O$ if necessary, we may assume that
\[O = \big\{Q' \in\ \rmM_{k+}:\ e^{-2\eps} < \tr_k Q' < e^{2\eps}\ \hbox{and}\ (\tr_k Q')^{-1}\cdot Q' \in U\big\}\]
for some $\eps > 0$ and some other neighbourhood $U$ of $I_k$ in $\rmM_{k+}$.  This turns into
\[T(d,O) = \big\{X \in \rmM_{d,k}:\ \sqrt{k}e^{-\eps} < \|X\|_2 < \sqrt{k}e^\eps\ \hbox{and}\ X/\|X\|_2 \in k^{-1/2}T(d,U)\big\}.\]
Identifying $\rmM_{d,k}$ with $\bbC^{dk}$ and integrating in polar coordinates~\cite[Subsec. 1.4.3]{RudinFTUB}, we obtain
\begin{align}\label{eq:use-of-polar-1}
\vol_{2kd}T(d,O) &= 2kd\cdot v(kd)\cdot \s_{2kd-1}\big(k^{-1/2}T(d,U)\big)\cdot \int_{\sqrt{k}e^{-\eps}}^{\sqrt{k}e^\eps}r^{2kd-1} dr\\
&= k^{kd}\cdot (e^{2kd\eps} - e^{-2kd\eps})\cdot v(kd)\cdot \s_{2kd-1}\big(k^{-1/2}T(n,U)\big) \nonumber.
\end{align}
By Corollary~\ref{cor:conc} and the asymptotic~\eqref{eq:ball-vol-approx}, this is greater than $v(d)^k$ for all sufficiently large $d$.

\vspace{7pt}

\emph{Step 2.}\quad On the other hand, for any $\eps > 0$, the set
\[O := \{Q' \in \rmM_{k+}:\ \tr_k Q' < e^{2\eps}\}\]
is a neighbourhood of $I_k$ in $\rmM_{k+}$, and it satisfies
\[\vol_{2kd}T(d,O) = \vol_{2kd} B_{\sqrt{k}e^\eps}(0) = k^{kd}\cdot e^{2kd\eps}\cdot v(kd) = e^{2kd\eps + o(d)}\cdot v(d)^k,\]
using~\eqref{eq:ball-vol-approx} again. Since $\eps$ is arbitrary, this proves part (b) for $Q = I_k$.

\vspace{7pt}

\emph{Step 3.}\quad Now let $Q \in \rmM_{k+}$, let $R \in \rmM_k$, and let $Q' := R^\ast Q R$.  Then $Q'$ also lies in $\rmM_{k+}$.  By the continuity of matrix multiplication, if $O'$ is any neighbourhood of $Q'$, then $Q$ has a neighbourhood $O$ such that
\[O' \supset R^\ast \cdot O \cdot R.\]
In terms of tuples of vectors, this turns into
\begin{equation}\label{eq:tuples-inclusion}
T(d,O') \supset \{XR:\ X \in T(d,O)\}.
\end{equation}
If we regard a $d$-by-$k$ matrix as the $d$-tuple of its rows, then right-multiplication by $R$ on $\rmM_{d,k}$ becomes the direct sum of $d$ copies of $R^\rm{T}$ acting on $\bbC^{\oplus k}$.  Regarded as a real linear transformation acting on $2kd$ real linear dimensions, this direct sum transformation has Jacobian $|\det R|^{2d}$: see, for instance,~\cite[Subsec. 1.3.5]{RudinFTUB}. Therefore~\eqref{eq:tuples-inclusion} gives
\[\vol_{2kd}T(d,O') \ge |\det R|^{2d}\cdot \vol_{2kd}T(d,O).\]
Since $O'$ is an arbitrary neighbourhood of $Q'$, and
\[\det Q' = |\det R|^2\cdot \det Q,\]
this shows that part (a) for $Q'$ follows if we already know part (a) for $Q$.  Similarly, if $R$ is invertible, then we may reverse the roles of $Q$ and $Q'$ above and deduce that part (b) for $Q'$ follows if we already know part (b) for $Q$.

In particular, combining this reasoning with Steps 1 and 2 and making the choice $R = Q^{-1/2}$, we conclude parts (a) and (b) whenever $Q$ is nonsingular.

\vspace{7pt}

\emph{Step 4.}\quad Finally, assume that $Q = [q_{ij}]$ is singular.  Then part (a) is vacuous.

Applying Step 3 with $R$ a unitary matrix, we may assume that $Qe_1 = 0$ for the standard basis $e_1$, \dots, $e_k$ of $\bbC^k$.  Having done so, let $r > \max_i \sqrt{q_{ii}}$, let $\eps > 0$, and let
\[O := \{Q' \in \rmM_{k+}:\ q_{11}' < \eps^2\ \hbox{and}\ \max_i q'_{ii} < r^2\}.\]
Then $O$ is a neighbourhood of $Q$, and
\[T(d,O) = \big\{X \in\rmM_{d,k}:\ \|x_1\| < \eps\ \hbox{and}\ \max_i \|x_i\| < r\big\}.\]
Therefore this neighbourhood satisfies
\[\frac{\vol_{2kd}T(d,O)}{v(d)^k} \le \frac{(\eps^{2d}\cdot v(d))\cdot (r^{2d}\cdot v(d))^{k-1}}{v(d)^k} = \eps^{2d}\cdot r^{2(k-1)d}.\]
Since $\eps$ can be chosen independently of $r$, this completes the proof of part (b).
\end{proof}

Theorem~\ref{thm:types-1} is a template for Theorem~\ref{mainthm:det-form}, as well as a special case of that theorem with $\A = \rmM_k$.  The proof given above, which avoids evaluating any integrals over spaces of matrices exactly, is also a precursor to the proof of Theorem~\ref{mainthm:det-form}.

For discrete Shannon entropy, the usual proofs in the method of types involve counting the strings that have exactly a given empirical distribution and then applying Stirling's approximation.  By contrast, in Step 3 of the proof of Theorem~\ref{thm:types-1}, we use a change of variables to transport the desired estimates from $I_k$ to any other non-singular element of $\rmM_{k+}$.  This use of the symmetries of $\vol_{2kn}$ has no obvious analog for probability distributions over finite sets.  A related use of symmetry is also essential to our later proof of Theorem~\ref{mainthm:det-form}.

\subsection{Definition and first properties of almost periodic entropy}\label{subs:APent}

Consider again a separable, unital C*-algebra $\A$.  Fix an AP sequence $\bspi = (\pi_n)_{n\ge 1}$, a positive integer $k$, and an element $\phi$ of $\frL(\A,\rmM_k)_+$.  Let $d_n$ be the dimension of $\pi_n$ for each $n$.  Recall Definition~\ref{dfn:typical} of the sets $\X(\pi_n,O)$.

\begin{dfn}\label{dfn:APent}
The \textbf{almost periodic} (`\textbf{AP}') \textbf{entropy of $\phi$ along $\bs{\pi}$} is the quantity
\begin{equation}\label{eq:APent}
 \rmh_{\bspi}(\phi) := \inf_O \limsup_{n \to\infty} \frac{1}{d_n}\log \frac{\vol_{2kd_n}\X(\pi_n,O)}{v(d_n)^k},
\end{equation}
where the infimum runs over all neighbourhoods of $\phi$.
\end{dfn}

\begin{rmk}\label{rmk:any-base-OK}
The expression on the right-hand side of~\eqref{eq:APent} is monotone in $O$, so we may restrict the infimum to any base of neighbourhoods around $\phi$ without changing the value of $\rmh_{\bspi}(\phi)$. \fin
\end{rmk}

The normalization outside the logarithm in~\eqref{eq:APent} is somewhat arbitrary: other natural choices would be $2d_n$ or $2kd_n$.  But the present choice seems to make for fewer explicit factors of $k$ later, for example in Lemma~\ref{lem:e-upper-bd} below.

	We use `$\limsup$' in Definition~\ref{dfn:APent} to allow for possible non-convergence. This matches our earlier choice to use $\rm{T}\limsup$ rather than $\rm{T}\liminf$ in the definition of asymptotic association (Definition~\ref{dfn:asymp-assoc}). Having made no extra assumptions on $\bspi$, there is no reason why using `$\liminf$' should give the same value in Definition~\ref{dfn:APent}.  Indeed, one would expect this to be false in case either (i) the sequence of tracial functionals $\tr_{d_n}\circ \pi_n$ does not converge in $\A^\ast$ or (ii) the sequence $\bspi$ does not strong-quotient converge.  However, once we account for these two possibilities, we find that using `$\limsup$' or `$\liminf$' does give the same quantity in~\eqref{eq:APent}: this is Corollary~\ref{cor:h-pi-and-strong-quot} below.

\begin{lem}\label{lem:upper-semicontinuity}
For any $\bspi$ and $k$, the function $\rmh_{\bspi}$ is upper semicontinuous on $\frL(\A,\rmM_k)_+$.
\end{lem}

\begin{proof}
The value $\rmh_{\bspi}(\phi)$ is an infimum of values associated to neighbourhoods of $\phi$.
\end{proof}

\begin{lem}\label{lem:e-upper-bd}
Any $\phi \in \frL(\A,\rmM_k)_+$ satisfies
	\[\rmh_{\bspi}(\phi) \le \log \det \phi(1).\]
	In particular, if $\phi(1)$ is singular then $\rmh_{\bspi}(\phi) = -\infty$.
	\end{lem}

\begin{proof}
If $O$ is any neighbourhood of $\phi(1)$ in $\rmM_k$, then the set
\[U := \{\psi \in \frL(\A,\rmM_k)_+:\ \psi(1) \in O\}\]
is a neighbourhood of $\phi$ in $\frL(\A,\rmM_k)_+$.  If $\pi$ is any representation on $\bbC^{\oplus d}$, then this $U$ satisfies
\[\X(\pi,U) = \{V^\rm{T} = [x_1,\dots,x_k]^{\rm{T}} \in (\bbC^{\oplus d})^{\oplus k}:\ V^\ast V \in O\}.\]
Therefore, for any $h > \log \det \phi(1)$, Theorem~\ref{thm:types-1}(b) gives
\[\log\frac{\vol_{2kd_n}\X(\pi_n,U)}{v(k)^{d_n}} < hd_n\]
for all sufficiently large $n$.  Since $\rmh_{\bspi}(\phi)$ is an infimum over all neighbourhoods of $\phi$ in $\frL(\A,\rmM_k)_+$, the particular neighbourhoods considered above show that $\rmh_{\bspi}(\phi) < h$ whenever $h > \log \det \phi(1)$.
	\end{proof}

We determine the cases of equality in Proposition~\ref{prop:max-ent} below.  Lemma~\ref{lem:e-upper-bd} is analogous to the fact that the entropy of a finite-valued stationary process is bounded above by the discrete Shannon entropy of a single letter.

For a given AP sequence $\bspi = (\pi_n)_{n\ge 1}$ and completely positive map $\phi$, whether $\phi$ is asymptotically associated to $\bspi$ depends only on the equivalence class of $\pi_\phi$.  However, the actual value $\rmh_{\bspi}(\phi)$ is more sensitive.  This is an important point where our story diverges from ergodic theory.  In ergodic theory, one of the most essential properties of sofic entropy is its independence of the choice of generating observable~\cite{Bowen10}, and hence its invariance under measure-theoretic isomorphism. The analog of this for AP entropy is false. Instead, AP entropy enjoys a general transformation law when one cyclic tuple is exchanged for another: see Proposition~\ref{prop:transform}.

Some proofs about AP entropy are easier to digest in the special case $k=1$, if only because the notation is lighter.  The next result sometimes lets us make this simplification without losing any generality.  It accompanies Lemmas~\ref{lem:strong-quot-for-pairing} and~\ref{lem:trace-conv-for-pairing}.

\begin{lem}\label{lem:AP-for-pairing}
Let $k$ be a positive integer, let $\bspi = (\pi_n)_{n\ge 1}$ be an AP sequence, and let $\bspi^{(k)} := (\pi_n^{(k)})_{n\ge 1}$.  Let $\phi\in\frL(\A,\rmM_k)_+$, and define $\langle \phi,\cdot\rangle$ as in equation~\eqref{eq:pairing}.  Then
\[\rmh_{\bspi^{(k)}}(\langle \phi,\cdot\rangle) = \frac{1}{k} \rmh_{\bspi}(\phi).\]
\end{lem}

\begin{proof}
Let $O$ be a neighbourhood of $\phi$, and let $\t{O}$ be the associated neighbourhood of $\langle \phi,\cdot\rangle$ under the pairing isomorphism. Then we have
	\[\frac{\vol_{2kd_n}\X(\pi_n^{(k)},\t{O})}{v(kd_n)} = \frac{k^{-kd_n}\cdot\vol_{2kd_n}\X(\pi_n,O)}{v(kd_n)} = \frac{e^{o(d_n)}\cdot \vol_{2kd_n}\X(\pi_n,O)}{v(d_n)^k},\]
	using~\eqref{eq:X-and-X} for the first equality and~\eqref{eq:ball-vol-approx} for the second.  Now take logarithms, normalize by $kd_n$ (the dimension of $\pi_n^{(k)}$), and insert into Definition~\ref{dfn:APent}.
	\end{proof}

\begin{rmk}
The Introduction discusses AP entropy as a representation theoretic analog of sofic entropy.  Let us comment further on their relationship.

Theorem~\ref{mainthm:det-form} requires an AP sequence whose pulled-back traces converge to the state $\tau$, which then determines the Fuglede--Kadison determinant that appears.  Suppose that $\A = C^\ast \G$ and that $\tau$ is given by the regular charater.  Then $\tau$ is a limit of finite-dimensional characters if and only if $\G$ has an AP sequence $(\pi_n)_{n\ge 1}$ that separates its elements.  Such groups were called `maximally almost periodic' by von Neumann~\cite{vonNeu34}; a textbook introduction is~\cite[Secs. 16.4--5]{Dix--Cstar}, where they are called `injectable'. 

However, we can also study regular characters on a larger class of groups as follows.  If $\G$ is any countable group, then we can write it as $F/N$ for some free group $F$ and normal subgroup $N$. Now we can look for finite-dimensional representations of $F$ whose characters converge to the quasi-regular character $1_N$, rather than finite-dimensional representations of $\G$ whose characters converge to $1_{\{e\}}$.  This offers more flexibility, because those finite-dimensional representations of $F$ need not have trivial restriction to $N$ until we take their limit.  Allowing convergence in this sense, the availability of finite-dimensional approximants to the regular representation of $\G$ is equivalent to $\G$ being `hyperlinear' in the terminology of~\cite{Rad08}.  In particular, it does not depend on the choice of presentation $F/N$.  This class of groups is the `linear' analog of the sofic groups, and includes all sofic groups~\cite{EleSza05}.  Both classes of groups are described in the survey~\cite{KwiPes13}. 

Thus, specialized to positive definite functions on groups, AP entropy is most naturally applied for `hyperlinear' groups via associated quasi-regular characters on free groups.  For these its definition is a direct analog of the definition of sofic entropy via observables or partitions (see~\cite[Def. 5]{Bowen--survey}, for example). \fin
\end{rmk}

\subsection{First transformation formula}

Fix an AP sequence $\bspi = (\pi_n)_{n\ge 1}$, and let $d_n$ be the dimension of $\pi_n$.  At this point, some of our work starts to need the assumption that the tracial states $\tr_{d_n}\circ \pi_n$ converge to a limit in $\A^\ast$.

The next proposition is a basic change-of-variables formula for AP entropy.  It can be seen as a cousin of Voiculescu's change-of-variables formulas for his free entropy~\cite[Prop. 3.5]{Voi94}, although the proof in our `linear' setting is simpler.

\begin{prop}\label{prop:transform}
Fix $\phi \in \frL(\A,\rmM_k)_+$.
\begin{enumerate}
\item[a.] Let $Q \in \rmM_k$ be invertible, and define
\[\psi(b) := (Q^\rm{T})^\ast \phi(b) Q^\rm{T} \qquad (b \in \A).\]
Then
\[\rmh_{\bspi}(\psi) = 2\log |\det Q| + \rmh_{\bspi}(\phi).\]

\item[b.] Assume further that $\tr_{d_n}\circ \pi_n \to \tau$.  Let $a \in \rmM_k(\A)$ be invertible, and define $\psi \in \frL(\A,\rmM_k)_+$ in terms of $\phi$ and $a$ as in Lemma~\ref{lem:new-type}. Then $\phi$ is asymptotically associated to $\bspi$ if and only if $\psi$ is, and in that case
	\[\rmh_{\bspi}(\psi) = 2\log \Delta_{\tau\otimes \Tr_k}|a| + \rmh_{\bspi}(\phi).\]
\end{enumerate}
\end{prop}

\begin{proof}
\emph{Part (a).}\quad For any neighbourhood $O$ of $\psi$, the second part of Lemma~\ref{lem:lin-maps} gives a neighbourhood $U$ of $\phi$ such that
\[(I_d\otimes Q)[\X(\pi,U)] \subset \X(\pi,O)\]
for any representation $\pi$ on $\bbC^{\oplus d}$.  For each $n$, it follows that
\begin{align*}
\vol_{2kd_n}\X(\pi_n,O) &\ge |\det (I_{d_n}\otimes Q)|^2\cdot \vol_{2kd_n}\X(\pi_n,U) \\ &= |\det Q|^{2d_n}\cdot \vol_{2kd_n}\X(\pi_n,U).
\end{align*}
The determinants are squared here because $I_{d_n}\otimes Q$ is a linear transformation in $kd_n$ complex dimensions, but we must treat it as a real linear map in $2kd_n$ real dimensions for the purpose of computing volumes (see, for instance,~\cite[Subsec. 1.3.5]{RudinFTUB}).  Inserting this inequality into Definition~\ref{dfn:APent} and taking the infimum over $O$, it follows that
\[\rmh_{\bspi}(\psi) \ge 2\log|\det Q| + \rmh_{\bspi}(\phi).\]
Applying the same reasoning with the roles of $\phi$ and $\psi$ reversed and with $Q^{-1}$ in place of $Q$, we obtain the reverse inequality as well.

\vspace{7pt}

\emph{Part (b).}\quad This time, if $O$ is a neighbourhood of $\psi$, then Lemma~\ref{lem:lin-maps} gives a neighbourhood $O'$ of $\phi$ such that
	\begin{align}\label{eq:vol-compar}
\vol_{2kd_n}\X(\pi_n,O) &\ge \vol_{2kd_n}\big(\pi_n^{(k)}(a)[\X(\pi_n,O')]\big) \nonumber \\ &= \Det\,|\pi_n^{(k)}(a)|^2\cdot \vol_{2kd_n}\X(\pi_n,O')
\end{align}
for every $n$.  Once again, the determinant is squared because we must treat $\pi_n^{(k)}(a)$ as a linear map in $2kd_n$ real dimensions for the purpose of computing volumes.

Since $a$ is invertible and $\pi_n^{(k)}$ is a unital C*-algebra homomorphism, we have
\[\Det\,|\pi_n^{(k)}(a)| = \exp(\Tr_{kd_n}(\log |\pi_n^{(k)}(a)|)) = \exp\big(\Tr_{kd_n}(\pi_n^{(k)}(\log|a|))\big).\]
By our assumption on $\bspi$ and Lemma~\ref{lem:trace-conv-for-pairing}, this is equal to
\[\exp\big(d_n\big((\tau\otimes \Tr_k) (\log |a|) + o(1)\big)\big) = (\Delta_{\tau\otimes \Tr_k}|a|)^{d_n + o(d_n)} \qquad \hbox{as}\ n\to\infty.\]
Therefore, normalizing and taking logarithms in~\eqref{eq:vol-compar}, that inequality becomes
\[\frac{1}{d_n}\log \frac{\vol_{2kd_n}\X(\pi_n,O)}{v(d_n)^k} \ge 2\log \Delta_{\tau\otimes \Tr_k}|a| + o(1) + \frac{1}{d_n}\log \frac{\vol_{2kd_n}\X(\pi_n,O')}{v(d_n)^k}.\]
Letting $n\to\infty$ and then taking the infimum over $O$, this gives
	\[\rmh_{\bspi}(\psi) \ge 2\log \Delta_{\tau\otimes \Tr_k}|a| + \rmh_{\bspi}(\phi).\]

The reverse of this inequality also holds by swapping the roles of $\phi$ and $\psi$ and replacing $a$ with $a^{-1}$, which satisfies $\Delta_{\tau\otimes \Tr_k}|a^{-1}| = (\Delta_{\tau\otimes \Tr_k}|a|)^{-1}$ (see, for instance,~\cite[Thm. I.6.10(iii)]{Dix--vN}).
\end{proof}

\begin{rmk}
If ${\tr_{d_n}\circ \pi_n}$ converges to $\tau$, then we can recognize part (a) above as a special case of part (b) by letting $a := 1_\A\otimes Q$ in $\rmM_k(\A) = \A\otimes\rmM_k$ and checking that $\Delta_{\tau\otimes \Tr_k}|a| = \det |Q| = |\det Q|$.  We formulate part (a) separately because it holds without the assumption of trace convergence. \fin
\end{rmk}

\subsection{Spherical measures and concentration}

When $k=1$ and $\phi$ is normalized, the next lemma gives an alternative to using Lebesgue measure in Definition~\ref{dfn:APent}.

\begin{lem}\label{lem:normalized}
Let $\phi \in \S_1(\A)$ and let $\bspi$ be an AP sequence. Let $\cal{O}$ be a base of neighbourhoods around $\phi$ in $\S_1(\A)$. Then
\begin{equation}\label{eq:normalized1}
 \rmh_{\bspi}(\phi) = \inf_{O \in \cal{O}} \limsup_{n \to\infty} \frac{1}{d_n}\log \s_{2d_n-1}\X(\pi_n,O).
 \end{equation}
\end{lem}

\begin{proof}
Let $\cal{U}$ be the family of all sets that have the form
\begin{equation}\label{eq:U-delta-O}
\big\{\psi \in \A_+^\ast:\ e^{-2\delta} < \psi(1) < e^{2\delta}\ \hbox{and}\ \psi(1)^{-1}\cdot \psi \in O\big\}
\end{equation}
for some $O \in \cal{O}$ and $\delta > 0$.  Then $\cal{U}$ is a base of neighbourhoods around $\phi$ in $\A_+^\ast$.  We may therefore restrict attention to neighbourhoods from $\cal{U}$ when evaluating $\rmh_{\bspi}(\phi)$ (see Remark~\ref{rmk:any-base-OK}).

So now let $U$ be the set in~\eqref{eq:U-delta-O} for some $\delta > 0$ and $O \in \cal{O}$, and let $\pi$ be a representation on $\bbC^{\oplus d}$.  Then the special form of $U$ gives
\[1_{\X(\pi,U)}(ry) = 1_{(e^{-\delta},e^\delta)}(r)\cdot 1_{\X(\pi,O)}(y) \qquad (r > 0,\ y \in \rm{S}^{2d-1}).\]
As a result, integrating in polar coordinates~\cite[Subsec. 1.4.3]{RudinFTUB} gives
\begin{align*}
\vol_{2d}\X(\pi,U) &= 2d\cdot v(d)\cdot\s_{2d-1}\X(\pi,O)\cdot \int_{e^{-\delta}}^{e^\delta} r^{2d-1}\ dr\\
&= v(d)\cdot (e^{2\delta d} - e^{-2\delta d})\cdot\s_{2d-1}\X(\pi,O).
\end{align*}
This implies that
\[\s_{2d-1}\X(\pi,O) \le \frac{\vol_{2d}\X(\pi,U)}{v(d)} \le e^{2\delta d}\cdot\s_{2d-1}\X(\pi,O)\]
for all sufficiently large $d$. Inserting this into Definition~\ref{dfn:APent} and taking the limit supremum over $n$ and then the infimum over elements of $\cal{U}$, we obtain~\eqref{eq:normalized1}.
\end{proof}

If $\phi \in \S_k(\A)$ for some $k > 1$, then we may apply Lemma~\ref{lem:normalized} to the positive functional $\langle \phi,\cdot\rangle$ on $\rmM_k(\A)$, whose sets of approximately typical vectors along $\bspi^{(k)}$ are given by equation~\eqref{eq:X-and-X}.  We can then deduce a result for $\phi$ itself via Lemma~\ref{lem:AP-for-pairing}.

For a state, we now have a choice between the original definition of AP entropy and the alternative given by Lemma~\ref{lem:normalized}.  Each has its advantages.  A major advantage of the measures $\s_{2d-1}$ is the concentration estimate from Lemma~\ref{lem:conc}.  As above, we discuss this only for $k=1$ to lighten notation.  Let $\pi$ be a representation on $\bbC^{\oplus d}$. Since the trace of a matrix is invariant under unitary conjugation, the average of the type $\Phi^\pi_x$ with respect to the spherical measure $\s_{2d-1}$ is equal to $\tr_d\circ \pi$.  When $d$ is large, Lemma~\ref{lem:conc} improves this conclusion considerably: $\Phi^\pi_x$ is actually close to $\tr_d\circ \pi$ for most individual $x \in \rmS^{2d-1}$.

Let $\tau$ be a tracial state on $\A$, and let $\l$ be its GNS representation.

\begin{lem}\label{lem:conc0}
If $\tr_{d_n}\circ \pi_n \to \tau$, then for every neighbourhood $U$ of $\tau$ there are positive constants $C$ and $c$ such that
	\[\s_{2d_n-1}\X(\pi,U) \ge 1 - Ce^{-cd_n}.\]
\end{lem}

\begin{proof}
It suffices to prove this for $U$ belonging to some sub-base of neighbourhoods of $\tau$, since any other neighbourhood contains a finite intersection of these.  We may therefore assume that
\[U = \{\psi\in \S_1(\A):\ |\psi(a) - \tau(a)| < \eps\}\]
for some $a \in \A$ and $\eps > 0$.

Since $\tr_{d_n}\circ \pi_n(a) \to \tau(a)$, this $U$ satisfies
\[\X(\pi_n,U) \supset  \{v \in \rmS^{2d_n-1}:\ |\langle \pi_n(a)v,v\rangle - \tr_{d_n}\pi_n(a)| < \eps/2\}\]
for all sufficiently large $n$. Now the result follows from Lemma~\ref{lem:conc}.
\end{proof}

\begin{cor}\label{cor:char-and-weak-global}
Assume that $\tr_{d_n}\circ \pi_n \to \tau$.  If $\phi \in\frL(\A,\rmM_k)_+$ and $\phi$ is approximately associated to $\l^{\oplus \infty}$, then it is asymptotically associated to $(\pi_n)_{n\ge 1}$.
\end{cor}

\begin{proof}
We may assume that $\phi$ is normalized.  Proposition~\ref{prop:RadNik} gives that $\S_k(\l^{\oplus \infty})$ equals $\S_k(\l^{\oplus k})$, so now our assumption actually says that $\phi \in \ol{\S_k(\l^{\oplus k})}$.

Since $\l^{\oplus k}$ is the minimal dilation of $\tau \otimes I_k := \rm{diag}(\tau,\dots,\tau)$, by Corollary~\ref{cor:typ-trans} it suffices to show that $\tau\otimes I_k$ itself is asymptotically associated to $\bspi$.  When $k = 1$, this follows from Lemma~\ref{lem:conc0}: indeed, a random vector drawn from $\s_{2d_n-1}$ is approximately typical for $\tau$ with high probability once $d_n$ is large enough.  Finally, if $k > 1$, then we can apply the previous case to the sequence $\bspi^{(k)}$ using Lemmas~\ref{lem:trace-conv-for-pairing} and~\ref{lem:strong-quot-for-pairing} and the fact that
	\[\langle \tau\otimes I_k,\cdot\rangle = \tau\otimes \tr_k.\]
\end{proof}

In the notation of Subsection~\ref{subs:strong-quotient}, Corollary~\ref{cor:char-and-weak-global} gives a lower bound on the set $\rm{T}\liminf_n\ol{\S_k(\pi_n)}$ for each $k$.  This can be an equality: for example, this is so if $\A = C^\ast \G$, the states $\tr_{d_n}\circ \pi_n$ converge to the regular tracial state, and $(\pi_n)_{n\ge 1}$ also converges strongly to the regular representation.  Some important examples satisfying these three conditions are surveyed in~\cite{Magee--survey}.  In other cases, the inclusion can be strict.  Nevertheless, we do always obtain the following.

\begin{cor}\label{cor:singular-part-controls}
Assume that $\tr_{d_n}\circ \pi_n \to \tau$.  Let $\phi \in \frL(\A,\rmM_k)_+$, and consider its Lebesgue decomposition $\phi_{\rm{ac}} + \phi_{\rm{sing}}$ relative to $\tau$.  Then $\phi$ is asymptotically associated to $\bspi$ if and only if $\phi_{\rm{sing}}$ is asymptotically associated to $\bspi$.
\end{cor}

\begin{proof}
First, by~\eqref{eq:pairing-decomp} and Lemmas~\ref{lem:strong-quot-for-pairing} and~\ref{lem:trace-conv-for-pairing}, it is equivalent to prove that $\langle \phi,\cdot\rangle$ is asymptotically associated to $\bspi^{(k)}$ if and only if $\langle\phi,\cdot\rangle_{\rm{sing}}$ is asymptotically associated to $\bspi^{(k)}$.  So we may assume that $k=1$ for the rest of the proof.

By construction, $\phi_{\rm{sing}}$ is associated to $\pi_\phi$.  Therefore, if $\phi$ is asymptotically associated to $\bspi$, then $\phi_{\rm{sing}}$ is as well by Corollary~\ref{cor:typ-trans}.

On the other hand, Corollary~\ref{cor:char-and-weak-global} tells us that $\phi_{\rm{ac}}$ is always asymptotically associated to $\bspi$, and $\phi_{\rm{ac}}$ and $\phi_{\rm{sing}}$ are disjoint by construction.  Therefore, if $\phi_{\rm{sing}}$ is asymptotically associated to $\bspi$, then so is $\phi$ itself by Lemma~\ref{lem:diag-and-sum-still-assoc}.
\end{proof}

Lemmas~\ref{lem:normalized} and~\ref{lem:conc0} also let us determine the cases of equality in Lemma~\ref{lem:e-upper-bd}.

\begin{prop}\label{prop:max-ent}
Assume that $\tr_{d_n}\circ \pi_n\to \tau$, and let $\phi \in \frL(\A,\rmM_k)_+$ with $\phi(1)$ nonsingular. Then equality holds in Lemma~\ref{lem:e-upper-bd} if and only if $\phi = \tau\otimes \phi(1)$.
\end{prop}

\begin{proof}
\emph{Step 1.}\quad First assume that $k=1$.  For any neighbourhood $U$ of $\tau$, Lemma~\ref{lem:conc0} shows that
\[\s_{2d_n-1}\X(\pi_n,U) \to 1.\]
Therefore $\rmh_{\bspi}(\tau) = 0$, by Lemma~\ref{lem:normalized}.

On the other hand, suppose that $\phi \ne \phi(1)\cdot \tau$.  After normalizing via Proposition~\ref{prop:transform}(a), we may assume that $\phi(1) = 1$.  Since $\phi \ne \tau$, they have disjoint neighbourhoods, say $U$ and $U'$ respectively.  Applying Lemma~\ref{lem:conc0} to the neighbourhood $U'$, there are positive constants $C$ and $c$ such that
\[\s_{2d_n-1}\X(\pi_n,U) \le\s_{2d_n-1}\big(\rm{S}^{2d_n-1}\setminus \X(\pi_n,U')\big)\le Ce^{-cd_n}.\]
This turns into $\rmh_{\bspi}(\phi) \le -c < 0$, again by Lemma~\ref{lem:normalized}.

\vspace{7pt}

\emph{Step 2.}\quad Now suppose that $k > 1$.  Then the case $k=1$ and Lemmas~\ref{lem:trace-conv-for-pairing} and~\ref{lem:AP-for-pairing} show that $\rmh_{\bspi}(\tau \otimes I_k) = 0$.

Let $\phi$ be any other element of $\frL(\A,\rmM_k)_+$ such that $\phi(1)$ is non-singular and equality is achieved in Lemma~\ref{lem:e-upper-bd}.  Applying Proposition~\ref{prop:transform}(a) with $Q := \phi(1)^{-1/2}$, we may this time assume that $\phi(1) = I_k$ and $\rmh_{\bspi}(\phi) = 0$.
	
Having done so, consider the sequence $\bspi^{(k)}$ and the pairing functional $\langle \phi,\cdot\rangle$ on $\rmM_k(\A)$.   Since $\phi$ achieves equality in Lemma~\ref{lem:e-upper-bd}, Lemma~\ref{lem:AP-for-pairing} gives
\[\rmh_{\bspi^{(k)}}(\langle \phi,\cdot\rangle) = \frac{1}{k}\rmh_{\bspi}(\phi) = 0.\]
Now the case $k=1$ applied to $\bspi^{(k)}$ shows that $\langle\phi,\cdot\rangle$ must equal
\[\tau\otimes \tr_k = \langle \tau\otimes I_k,\cdot\rangle,\]
and hence that $\phi = \tau\otimes I_k$.
\end{proof}

\begin{cor}\label{cor:pre-B}
	Assume that $\tr_{d_n}\circ \pi_n \to \tau$.  Let $\tau$ be associated to $\l$ by $\xi$, let $a$ be a positive invertible element of $\A$, and let $\phi := \Phi^\l_{\l(a)\xi}$.  Then $\rmh_{\bspi}(\phi) = 2\log\Delta_\tau a$.
\end{cor}

\begin{proof}
The expression for $\phi$ in terms of $\tau$ and $a$ is a special case of the relationship from Lemma~\ref{lem:new-type}.  Therefore Proposition~\ref{prop:transform}(b) gives
\[\rmh_{\bspi}(\phi) = 2 \log \Delta_\tau a + \rmh_{\bspi}(\tau).\]
On the right-hand side, the second term vanishes by Proposition~\ref{prop:max-ent}.
\end{proof}

Corollary~\ref{cor:pre-B} is a special case of Theorem~\ref{mainthm:det-form}.  In the next subsection we use it during the full proof of that theorem.  This is somewhat similar to the use of Proposition~\ref{prop:Den} in the proof of Theorem~\ref{mainthm:amenable}.

\subsection{Proof of Theorem~\ref{mainthm:det-form}}\label{subs:det-form}

Let $\tau$ be associated to the representation $\l$ by the cyclic tracial vector $\xi$, and let $\t{\tau}$ be the normal tracial state on $\l(\A)'$ defined by $\xi$.  We use the letter $\Delta$ for both (i) the Fuglede--Kadison determinant on $\A$ defined by $\tau$ and also (ii) the Fuglede--Kadison determinant on square-integrable operators affiliated to $\l(\A)'$.

\begin{proof}[Proof of Theorem~\ref{mainthm:det-form}]
\emph{Step 1.}\quad First we prove the inequality ``$\ge$'' when $k=1$ and $\phi$ is $\l$-normal.  In this case parts (i) and (iv) of Lemma~\ref{lem:Kap+2} give a sequence of positive invertible elements $b_1$, $b_2$, \dots of $\A$ such that
\[\tau(b_i(\cdot)b_i) \to \phi \qquad \hbox{and} \qquad (\Delta b_i)^2 \to \Delta \phi \qquad \hbox{as}\ i\to\infty.\]
Then Corollary~\ref{cor:pre-B} gives $\rmh_{\bspi}(\phi_i) = 2\log \Delta b_i$ for each $i$. Letting $i\to \infty$, the upper semicontinuity from Lemma~\ref{lem:upper-semicontinuity} turns this into
\[\rmh_{\bspi}(\phi) \ge 2\lim_{i\to\infty}\log \Delta b_i = \log \Delta \phi.\]

\emph{Step 2.}\quad Next we prove the inequality ``$\ge$'' when $k=1$ but $\phi$ is otherwise arbitrary.  Let $O$ be any neighbourhood of $\phi$.  Since $\phi_{\rm{ac}}$ and $\phi_{\rm{sing}}$ are disjoint,  Corollary~\ref{cor:sums}(a) gives neighbourhoods $U$ of $\phi_{\rm{sing}}$ and $W$ of $\phi_{\rm{ac}}$ such that
\[	\X(\pi_n,O) \supset \X(\pi_n,U) + \X(\pi_n,W)\]
for every $n$.  Since $\phi$ is asymptotically associated to $\bspi$ by assumption, so is $\phi_{\rm{sing}}$ by Corollary~\ref{cor:singular-part-controls}.  Therefore $\X(\pi_n,U)$ is nonempty along an infinite subsequence of values of $n$, say $n_1 < n_2 < \dots$.  Let $\bspi'$ be the corresponding AP subsequence of $\bspi$.  For each $n_i$, the set $\X(\pi_{n_i},O)$ contains a translate of $\X(\pi_{n_i},W)$, and so
\begin{align*}
\limsup_{n\to\infty}\frac{1}{d_n}\log \frac{\vol_{2d_n}\X(\pi_n,O)}{v(d_n)} &\ge\limsup_{i\to\infty}\frac{1}{d_{n_i}}\log \frac{\vol_{2d_{n_i}}\X(\pi_{n_i},O)}{v(d_{n_i})}\\
&\ge \limsup_{i\to\infty}\frac{1}{d_{n_i}}\log \frac{\vol_{2d_{n_i}}\X(\pi_{n_i},W)}{v(d_{n_i})} \\
&\ge \rmh_{\bspi'}(\phi_{\rm{ac}}).
\end{align*}
This lower bound is at least $\log \Delta \phi_{\rm{ac}} = \log\Delta\phi$ by applying Step 1 along the AP sequence $\bspi'$.  Since $O$ is arbitrary, this proves that $\rmh_{\bspi}(\phi)\ge \log \Delta \phi$.

\vspace{7pt}

\emph{Step 3.}\quad We now prove the inequality ``$\le$'' in case $k=1$.  This proof is quickest via the variational principle from Proposition~\ref{prop:det-var-princ}.

Let $a \in \A$ be positive and invertible and satisfy $\Delta a \ge 1$.  Define a new positive functional by $\psi:= \phi(\sqrt{a}(\cdot)\sqrt{a})$.  Then we have
\begin{align*}
\log \phi(a) &= \log \psi(1) \\
&\ge \rmh_{\bspi}(\psi) &\qquad &\hbox{(Lemma~\ref{lem:e-upper-bd})}\\
&= 2\log \Delta \sqrt{a} + \rmh_{\bspi}(\phi) &\qquad &\hbox{(Proposition~\ref{prop:transform}(b))}\\
&\ge \rmh_{\bspi}(\phi) &\qquad &\hbox{(because $\Delta \sqrt{a} = \sqrt{\Delta a} \ge 1$)}.
\end{align*}
Taking the infimum over $a$, Proposition~\ref{prop:det-var-princ} turns this into $\log \Delta \phi \ge \rmh_{\bspi}(\phi)$.

\vspace{7pt}

\emph{Step 4.}\quad Finally, if $k > 1$, then we can apply the previous steps to the functional $\langle \phi,\cdot\rangle$ on $\rmM_k(\A)$.  First, Lemma~\ref{lem:trace-conv-for-pairing} gives that $\tr_{kd_n}\circ \pi_n^{(k)}\to \tau \otimes \tr_k$.  Secondly, Lemma~\ref{lem:strong-quot-for-pairing} gives that $\langle \phi,\cdot\rangle$ is asymptotically associated to $\bspi^{(k)}$.  Finally, we have
\begin{align*}
\rmh_{\bspi}(\phi) &= k\cdot \rmh_{\bspi^{(k)}}(\langle \phi,\cdot\rangle) &\qquad &\hbox{(Lemma~\ref{lem:AP-for-pairing})}\\
&= k\cdot \log \Delta_{\tau\otimes \tr_k}(\langle \phi,\cdot\rangle) &\qquad &\hbox{(case $k=1$ of Theorem~\ref{mainthm:det-form})}\\
&= \log \Delta \phi &\qquad &\hbox{(equation~\eqref{eq:det-phi-det-pairing})}.
\end{align*}
\end{proof}

\subsection{Some consequences of Theorem~\ref{mainthm:det-form}}\label{subs:C-cors}

\subsubsection*{\emph{Basic consequences}}

\begin{cor}\label{cor:det-form}
	Assume that $\tr_{d_n}\circ \pi_n \to \tau$, and let $\phi,\psi\in \frL(\A,\rmM_k)_+$.
	\begin{enumerate}
		\item[a.] We have
		\[\rmh_{\bspi}(\phi) = \left\{\begin{array}{ll}\rmh_{\bspi}(\phi_{\rm{ac}}) & \quad \hbox{if $\phi_{\rm{sing}}$ is asymptotically associated to $\bspi$}\\ -\infty & \quad \hbox{otherwise.} \end{array}\right.\]
		\item[b.] If $\rmh_{\bspi}(\phi) > -\infty$, then $\phi_{\rm{ac}}$ is associated to $\l^{\oplus k}$ by a $k$-tuple that is cyclic and separating for $\l^{\oplus k}(\A)''$, and so $\pi_\phi \gtrsim \pi_{\phi_{\rm{ac}}} \simeq \l^{\oplus k}$.
		\item[c.] If $\phi \ge \psi$ in the positive definite ordering, and if $\phi$ is asymptotically associated to $\bspi$, then $\rmh_{\bspi}(\phi) \ge \rmh_{\bspi}(\psi)$.
	\end{enumerate}
\end{cor}

\begin{proof}
Part (a) is just the combination of Corollary~\ref{cor:singular-part-controls} and Theorem~\ref{mainthm:det-form}.

If $\rmh_{\bspi}(\phi) > -\infty$, then Theorem~\ref{mainthm:det-form} tells us that $\Delta \phi = \Delta \phi_{\rm{ac}} > 0$.  Expressing this Fuglede--Kadison determinant in terms of a non-negative affiliated operator $T$ from Proposition~\ref{prop:RadNik}, it follows that $\Delta T$ is also positive, and hence $T$ is nonsingular.  Since $T$ is self-adjoint, it therefore also has dense image.  Since $\pi_{\phi_{\rm{ac}}}$ is equivalent to the subrepresentation of $\l^{\oplus k}$ defined by $\ol{\rm{img}\,T}$, this proves part (b).

Finally, if $\phi \ge \psi$ and $\phi$ is asymptotically associated to $\bspi$, then so is $\psi$, and Theorem~\ref{mainthm:det-form} shows that part (c) follows from Proposition~\ref{prop:det-props}(a).
\end{proof}

\subsubsection*{\emph{Modes of convergence}}

If an AP sequence $\bspi$ strong-quotient converges, then Theorem~\ref{mainthm:det-form} shows that the AP entropy function $\rmh_{\bspi}$ depends only on the values of the limits
\[\tau = \lim_n \tr_{d_n}\circ \pi_n \qquad \hbox{and} \qquad \lim_n \S_k(\pi_n) \qquad \hbox{for}\ k=1,2,\dots.\]
If $\tr_{d_n}\circ \pi_n \to \tau$ but $\bspi$ does not strong-quotient converge, then Theorem~\ref{mainthm:det-form} may give different values of AP entropy along different subsequences, but the only two possible values are $\log\Delta_\tau \phi$ or $-\infty$ (which may still happen to be equal).

In ergodic theory, an important open problem asks whether the sofic entropy of a measure-preserving system must always equal either one particular value or $-\infty$, independently of the choice of sofic approximation.  Corollary~\ref{cor:det-form} answers the analogous question positively for AP entropy.  On the other hand, it is known that the answer is negative for the topological variant of sofic entropy~\cite{AirBowLin22}.

\begin{rmk}\label{rmk:LD-conv}
Subsection~\ref{subs:strong-quotient} mentions the analogy between strong-quotient convergence of AP sequences and local-global convergence in graph theory, as discussed further in~\cite{AbeEle11}.  If $(\pi_n)_{n\ge 1}$ strong-quotient converges, then Theorem~\ref{mainthm:det-form} shows that the expressions appearing in the definition of AP entropy also converge in a sense much like a large deviations principle.  This is actually the analog of another mode of convergence for graph sequences of uniformly bounded degree: `large deviation convergence', introduced in~\cite{BorChaGam17}.  Whether local-global convergence implies large deviations convergence for such graph sequences remains open.  Our results answer the analogous question for AP sequences. \fin
\end{rmk}

Now ssume again that $\tr_{d_n}\circ \pi_n \to \tau$.  We can reverse the discussion above by asking whether the function $\rmh_{\bspi}$ determines the topological upper limit of the sequence $(\S_k(\pi_n))_{n\ge 1}$ for each $k$.  This is slightly subtle, because for $\phi \in \S_k(\A)$ the value $\log \Delta_\tau \phi$ may equal $-\infty$ even if $\pi_\phi \gtrsim \l^{\oplus k}$.  To evade this problem, we can instead use $\phi$ to form the perturbations
\[\phi_t := \tau\otimes I_k + t\phi \qquad (t\ge 0).\]

\begin{cor}\label{cor:mollify}
Under the assumptions above, the following hold.
\begin{enumerate}
\item[a.] If $\phi_t$ is asymptotically associated to $\bspi$ for some $t  > 0$, then this holds for all $t > 0$, and we have
\[0 \le \rmh_{\bspi}(\phi_t) \le k\cdot \log (1 + t\cdot \tr_k \phi(1)) \qquad (t \ge 0).\]
\item[b.] If $\phi_t$ is not asymptotically associated to $\bspi$ for any $t > 0$, then $\rmh_{\bspi}(\phi_t) = -\infty$ for every $t > 0$.
\end{enumerate}
\end{cor}

\begin{proof}
The linearity of the Lebesgue decomposition gives
\[(\phi_t)_{\rm{ac}} = \tau\otimes I_k + t \phi_{\rm{ac}} \qquad \hbox{and} \qquad (\phi_t)_{\rm{sing}} = t\phi_{\rm{sing}}\]
for every $t \ge 0$.  Therefore, if $\phi_t$ is asymptotically associated to $\bspi$ for some $t > 0$, then so is $\phi_{\rm{sing}}$, and hence so is $\phi_s$ for every other $s$, by two applications of Corollary~\ref{cor:singular-part-controls}.

Now the lower bound on $\rmh_{\bspi}(\phi_t)$ in part (a) follows from Corollary~\ref{cor:det-form}(c) and Proposition~\ref{prop:max-ent}, because $\phi_t \ge \tau\otimes I_k$ in the positive definite ordering.  On the other hand, since $\Delta_{\tau\otimes \tr_k}(1\otimes I_k) = 1$, the element $1\otimes I_k$ of $\rmM_k(\A)$ is allowed inside the infimum of the variational principle from Proposition~\ref{prop:det-var-princ}. Combined with Theorem~\ref{mainthm:det-form}, this gives
\[\rmh_{\bspi}(\phi_t) = \log \Delta_\tau \phi_t \le k\cdot \log \langle \phi_t,1\otimes I_k\rangle.\]
This gives the desired upper bound on $\rmh_{\bspi}(\phi_t)$, because
\[\langle \phi_t,1\otimes I_k\rangle = \langle \tau\otimes I_k,1\otimes I_k\rangle + t\langle \phi,1\otimes I_k\rangle = 1 + t\cdot\tr_k \phi(1).\]

Finally, part (b) follows from Definition~\ref{dfn:APent}, applied when the sets $\X(\pi_n,O)$ are empty for some sufficiently small neighbourhood $O$ of $\phi_t$.
\end{proof}

\begin{rmk}
Alternatively, one can prove Corollary~\ref{cor:mollify}(a) using an extension of the concentration result in Lemma~\ref{lem:conc0} rather than the full strength of Theorem~\ref{mainthm:det-form}.  Referring to the case when $k=1$ and $t=1$ for simplicity, the idea is as follows.  Let $\pi$ be a representation of large dimension $d$, and let $x \in \bbC^{\oplus d}$ be such that $\Phi^\pi_x$ is close to $\phi$. Using essentially the same proof as for Lemma~\ref{lem:conc0}, one can show that, according to the measure $\s_{2d-1}$, most unit vectors $y \in \rmS^{2d-1}$ satisfy $\Phi^\pi_{[x,y]} \approx \rm{diag}(\phi,\tau)$, and hence $\Phi^\pi_{x+y} \approx \tau + \phi$.  By using polar coordinates to integrate over $y$ with respect to $\vol_{2d}$, this fact turns into the same lower bound as in Corollary~\ref{cor:mollify}(a). \fin
\end{rmk}

Now we can see how the function $\rmh_{\bspi}$ determines topological upper limits.

\begin{cor}\label{cor:h-pi-and-strong-quot}
Assume that $\tr_{d_n}\circ \pi \to\tau$.  Then
\begin{align*}
\rm{T}\limsup_{n\to\infty} \S_k(\pi_n) &= \{\phi \in \S_k(\A):\ \rmh_{\bspi}(\phi_t) \ge 0\ \forall t \ge 0\}\\
&= \{\phi \in \S_k(\A):\ \rmh_{\bspi}(\phi_t) \to 0\ \hbox{as}\ t\downarrow 0\}
\end{align*} 
for each positive integer $k$. Consequently, $\bspi$ strong-quotient converges if and only if we have $\rmh_{\bspi'} = \rmh_{\bspi}$ whenever $\bspi'$ is a subsequence of $\bspi$. \qed
\end{cor}

\subsubsection*{\emph{Additivity and concavity}}

In ergodic theory, Kolmogorov--Sinai entropy for single transformations is additive under Cartesian products~\cite[Thm. 4.23]{Walters--book}.  By contrast, sofic entropy is subadditive for joinings, but it may fail to be additive, even for a product joining~\cite{Aus--soficentadd}.  The same phenomena hold for AP entropy, for essentially the same reasons.  Let us assume again that $\tr_{d_n}\circ \pi_n \to \tau$ and that $\bspi$ strong-quotient converges; if these conditions fail then we can always pass to a subsequence that satisfies them.  

\begin{cor}\label{cor:AP-ent-subadd}
Let $\phi \in \frL(\A,\rmM_k)_+$, let $\psi \in \frL(\A,\rmM_\ell)_+$, and let $\theta$ be a joining of them (recall Definition~\ref{dfn:joining}).  Then
\begin{equation}\label{eq:AP-ent-subadd}
\rmh_{\bspi}(\theta) \le \rmh_{\bspi}(\phi) + \rmh_{\bspi}(\psi).
\end{equation}
If $\theta_{\rm{ac}} = \rm{diag}(\phi_{\rm{ac}},\psi_{\rm{ac}})$, then equality holds provided $\theta$ is asymptotically associated to $\bspi$.
\end{cor}

\begin{proof}
If $\theta$ is not asymptotically associated to $\bspi$, then $\rmh_{\bspi}(\theta) = -\infty$ and there is nothing left to prove.  So assume that $\theta$ is asymptotically associated to $\bspi$.  Both $\phi$ and $\psi$ are associated to $\pi_\theta$, so these are also asymptotically associated to $\bspi$.  Now the AP entropies of all three maps are given by Theorem~\ref{mainthm:det-form}, and the desired conclusions follow from Proposition~\ref{prop:det-props}(d).
\end{proof}

\begin{rmk}
Alternatively, the subadditivity of $\rmh_{\bspi}$ may be proved by observing that, for any neighbourhoods $U$ of $\phi$ and $V$ of $\psi$, there is a neighbourhood $O$ of $\theta$ such that
\[\X(\pi,O)\subset \X(\pi,U)\times \X(\pi,V)\]
for any representation $\pi$.  Using this to compare volumes and inserting into the definition of AP entropy, it turns into~\eqref{eq:AP-ent-subadd}.  However, the case of equality in Corollary~\ref{cor:AP-ent-subadd} does seem to require a stronger ingredient such as Theorem~\ref{mainthm:det-form}. \fin
\end{rmk}

In the last part of Corollary~\ref{cor:AP-ent-subadd}, it is not enough to assume that $\phi$ and $\psi$ are both separately asymptotically associated to $\bspi$: see Example~\ref{ex:lim-set-not-convex} below.

Notions of entropy are often given by concave functions, so it is natural to ask whether the restrictions $\rmh_{\bspi}|\frL(\A,\rmM_k)_+$ have this property.  Under some conditions, they do.  We obtain the following from Theorem~\ref{mainthm:det-form} together with the log-concavity inequality~\eqref{eq:log-det-concave} (another consequence of Proposition~\ref{prop:det-props}).

\begin{cor}
Let $\phi,\psi \in \frL(\A,\rmM_k)_+$, and assume moreover that ${t\phi + (1-t)\psi}$ is asymptotically associated to $\bspi$ for every $t \in [0,1]$.  Then $\rmh_{\bspi}$ is concave along the line segment from $\phi$ to $\psi$:
\[\rmh_{\bspi}(t\phi + (1-t)\psi) \ge t\rmh_{\bspi}(\phi) + (1-t)\rmh_{\bspi}(\psi) \qquad (0\le t \le 1).\]
\qed
\end{cor}

However, the function $\rmh_{\bspi}$ need not be globally concave on $\frL(\A,\rmM_k)_+$, or even on its subset $\S_k(\A)$.  The problem is that the set $\lim_n\S_k(\pi_n)$ itself need not be convex.  We illustrate this by an extension of Example~\ref{ex:sigma-not-convex}.

\begin{ex}\label{ex:lim-set-not-convex}
Let $\G$ be a countable group with left regular representation $\l$.  Let $\pi$ be an irreducible representation of $\G$ such that (i) $2 \le \dim \pi < \infty$ and (ii) some element of $\S_1(\pi)$ is not approximately associated to $\l$.  Separately, let $(\rho_n)_{n\ge 1}$ be an AP sequence for $\G$ that strong-quotient converges to $\l$.  Suitable examples include randomly generated unitary or permutation representations of free groups, by the results of~\cite{ColMal14} and~\cite{BordCol19}, respectively; see also~\cite{Magee--survey}.

Now let $\pi_n:= \pi\oplus \rho_n$ for each $n$.  Checking the definitions shows that
\[\S_k(\pi_n) = \{t\phi + (1-t)\psi:\ \phi \in \S_k(\pi),\ \psi \in \S_k(\rho_n),\ 0\le t \le 1\}.\]
This is closed for every $k$ and $n$ because it is a continuous image of a finite-dimensional sphere.  One can now check that $(\pi_n)_{n\ge 1}$ strong-quotient converges to $\pi \oplus \l$, and so in particular
\[\S_k(\pi_n) \to \{t\phi + (1-t)\psi:\ \phi \in \S_k(\pi),\ \psi \in \ol{\S_k(\l)},\ 0\le t \le 1\}\]
in the Vietoris topology.  However, since $\pi$ is irreducible and not approximately associated to $\l$, it follows that $\lim_n\S_2(\pi_n)$ does not contain $\rm{diag}(\phi,\psi)$ when $\phi$ and $\psi$ are linearly independent elements of $\S_1(\pi)$.  Moreover, by the same reasoning as for Example~\ref{ex:sigma-not-convex}, it also follows that $\lim_n \S_1(\pi_n)$ does not contain $(\phi + \psi)/2$, so $\lim_n \S_1(\pi_n)$ is not convex.  Finally, Corollary~\ref{cor:mollify}(a) gives that $\rmh_{\bspi}(\phi + \tau)$ and $\rmh_{\bspi}(\psi + \tau)$ are both non-negative, but we have
\[\rmh_{\bspi}(\rm{diag}(\phi + \tau,\psi+\tau)) = \rmh_{\bspi}((\phi + \tau)/2 + (\psi + \tau)/2) = -\infty.\]
\qed
\end{ex}

Nevertheless, if $(\pi_n)_{n\ge 1}$ strong-quotient converges, then $\lim_n \S_k(\pi_n)$ is at least star-shaped around the element $\tau\otimes I_k$ (or any other $\l$-normal element of $\S_k(\A)$).  This is a consequence of Corollary~\ref{cor:singular-part-controls}, because $\tau\otimes I_k$ does not contribute to the singular part in a convex combination.  Overall, the limit set $\lim_n \S_k(\pi_n)$ can fail to be convex only because of its excess outside the lower bound on this set provided by Corollary~\ref{cor:char-and-weak-global}, as illustrated by Example~\ref{ex:lim-set-not-convex}.

Since $\rmh_{\bspi}$ may not be concave on the whole of $\S_k(\A)$, it may not be recoverable from its Legendre transform.  Nevertheless, it could be interesting to investigate how various features of the sequence $(\pi_n)_{n\ge 1}$ are reflected by that transform, and how it compares with the Legendre transform of $\log \Delta$ itself. See~\cite[Sec. I.6]{SimSMLG} for a general account of infinite-dimensional Legendre transforms with a view towards statistical mechanics, or~\cite{Chung13} for the resulting variational principle in the case of sofic entropy.  In the analogous setting of modes of convergence for graphs of uniformly bounded degree (see Subsection~\ref{subs:strong-quotient} and Remark~\ref{rmk:LD-conv}), convergence of these Legendre transforms would correspond to `right convergence'~\cite{BolRio11,BorChaKahLov13}.

\subsection{Further remarks}\label{subs:C-rmks}

\subsubsection*{\emph{Comparison with previous work}}

Just like Theorem~\ref{mainthm:amenable}, Theorem~\ref{mainthm:det-form} has a number of predecessors in the literature.  Some of the first that lie beyond the discussion in Subsection~\ref{subs:A-rmks} are Lyons' calculations in~\cite{Lyons05,Lyons10}.  These concern the problem of asymptotically counting spanning trees along sequences of finite connected graphs using their random weak limits.  Lyons shows that the resulting `tree entropy' of the random limit graph is given by a Fuglede--Kadison determinant of its Laplacian.  His setting does not require a group action, but it yields results for sofic groups as a special case.

More recent examples are continuations of Deninger's work on the entropies of certain algebraically-defined measure-preserving systems, but now for sofic groups and sofic entropy.  Kerr and Li took the first step in this direction with~\cite[Thm. 7.1]{KerLi11b}.  Then Hayes gave a more complete result in~\cite{Hayes16}, showing that the sofic entropy of $\boldsymbol{X}_f$ equals the Fuglede--Kadison determinant of $f \in \rmM_k(\bbZ[\G])$ whenever $\G$ is sofic and $\l^{(k)}(f)$ is injective.  An alternative proof with further refinements is included in~\cite{Hayes21}.  On the other hand, another result from~\cite{Hayes16} is that the Li--Thom theorem from~\cite{LiTho14} does not generalize fully to sofic groups.

Alongside those papers, Hayes has developed other connections between sofic entropy and representation theory.  In~\cite{Hayes18} he proved that an arbitrary measure-preserving $\G$-system can have completely positive sofic entropy only if its Koopman representation is $\l$-normal.  In~\cite{Hayes17}, he computed the sofic entropy of a stationary Gaussian process over $\G$ in terms of the real orthogonal representation that defines its first chaos, generalizing a result from~\cite{HamParTho08} about single transformations.  The main theorem in~\cite{Hayes17} is worth comparing with the way in which $\phi_{\rm{sing}}$ and $\phi_{\rm{ac}}$ determine whether $\rmh_{\bspi}(\phi)$ equals $-\infty$ in Corollary~\ref{cor:det-form} above.

\subsubsection*{\emph{Other aspects of Szeg\H{o}'s theorem}}

The setting of Theorem~\ref{mainthm:det-form} has taken us quite far from our original motivation in the form of Szeg\H{o}'s theorem.  AP entropy is not defined as a limit of finite-dimensional determinants, and related data such as Verblunsky coefficients have no obvious meaning in this generality.

However, some of those finer aspects of Szeg\H{o}'s theorem make a return in a sequel to the present paper~\cite{APE4b}.  That paper studies random AP sequences, and in particular an annealed version of AP entropy that can be defined using these.  In the special case of uniformly random finite-dimensional representations of free groups, this is a representation theoretic analog of Bowen's annealed sofic entropy (formerly called the `f-invariant') from~\cite{Bowen10free,Bowen10c}.  Like annealed sofic entropy, this instance of annealed AP entropy admits a precise formula.

In studying this entropy and its formula, many features of the theory of orthogonal polynomials on $\bbT$ reappear.  For example, sequences of `generalized Verblunsky coefficients' can be used to parametrize positive definite functions over free groups.  Rather than requiring a total ordering of the group, these generalized Verblunsky coefficients depend on the fact that free groups have tree-like Cayley graphs.  One of the first main theorems about annealed AP entropy is an infinite series expansion of it in terms of those coefficients.  This offers a more complete analog of Szeg\H{o}'s limit theorem for positive definite functions on free groups, with one entirely new feature: an additional term called `zeroth-order' entropy that reflects the non-amenability of the groups.

We leave further details to~\cite{APE4b}, which also develops applications to large deviations for tuples of random matrices.

\bibliography{bibfile}{}
\bibliographystyle{abbrv}

\begin{small}
\noindent Mathematics Institute\\ Zeeman Building\\ University of Warwick\\ Coventry CV4 7AL\\ United Kingdom\\ \href{mailto:Tim.Austin@warwick.ac.uk}{\texttt{Tim.Austin@warwick.ac.uk}}
\end{small}

\end{document}